\documentclass[11pt, twoside]{article}
\usepackage{mathrsfs}
\usepackage{amssymb}
\usepackage{amsmath}
\usepackage{mathrsfs}
\usepackage{amsthm}
\usepackage{amsfonts}
\usepackage{color}
\usepackage{latexsym}
\usepackage{txfonts}
\usepackage{indentfirst}

\usepackage{anysize}

\usepackage{comment}

\allowdisplaybreaks
\newcounter{assum}

\newtheorem{assumption}[assum]{Assumption}
\newtheorem{theorem}{Theorem}[section]
\newtheorem{lemma}[theorem]{Lemma}
\newtheorem{corollary}[theorem]{Corollary}
\newtheorem{proposition}[theorem]{Proposition}

\theoremstyle{definition}
\newtheorem{remark}[theorem]{Remark}
\newtheorem{definition}[theorem]{Definition}
\numberwithin{equation}{section}

\begin{document}
\title{\bf\Large Maz'ya--Shaposhnikova Representation of Quasi-Norms of
Ball Quasi-Banach Function Spaces on Spaces of Homogeneous Type
with Weak Reverse Doubling Property
\footnotetext{\hspace{-0.35cm} 2020 {\it Mathematics Subject Classification}.
Primary 46E36; Secondary 42B35, 42B25, 26D10, 30L99.
\endgraf {\it Key words and phrases.} space of homogeneous type,
ball quasi-Banach function space,
Maz'ya--Shaposhnikova representation,
weak reverse doubling condition,
weak measure density condition.
\endgraf This project is partially supported by
the National Natural Science Foundation of China
(Grant Nos. 12371093, 12431006, and 12501129),
and the Fundamental Research Funds for the
Central Universities (Grant No. 2253200028).}}
\author{Eiichi Nakai, Menghao Tang,
Dachun Yang\footnote{Corresponding author,
E-mail: \texttt{dcyang@bnu.edu.cn}/{\color{red}\today}/Final version.},\ \
Wen Yuan and Chenfeng Zhu}
\date{}
\maketitle

\vspace{-0.8cm}

\begin{center}
\begin{minipage}{13cm}
{\small {\bf Abstract}\quad
Let $Y(\mathcal{X})$ be a ball quasi-Banach function space on
the space of homogeneous type $(\mathcal{X},\rho,\mu)$ satisfying
some mild additional assumptions,
$q\in(0,\infty)$, and $\dot{W}^{s,q}_Y(\mathcal{X})$ with $s\in(0,1)$
be the homogeneous
fractional Sobolev space associated with $Y(\mathcal{X})$.
In this article, we show that, for
any $f\in Y(\mathcal{X})\cap\bigcup_{s\in(0,1)}
\dot{W}^{s,q}_Y(\mathcal{X})$,
\begin{align*}
\|f\|_{Y(\mathcal{X})}
&\lesssim\varliminf_{s \to 0^+}
s^{\frac{1}{q}}\left\|
\left\{\int_{\mathcal{X}}
\frac{|f(\cdot)-f(y)|^q}{U(\cdot,y)[\rho(\cdot,y)]^{sq}}
\, d\mu(y) \right\}^{\frac{1}{q}}\right\|_{{Y(\mathcal{X})}}\\
&\leq \varlimsup_{s \to 0^+}
s^{\frac{1}{q}}\left\|\left\{\int_{\mathcal{X}}
\frac{|f(\cdot)-f(y)|^q}{U(\cdot,y)[\rho(\cdot,y)]^{sq}}
\, d\mu(y) \right\}^{\frac{1}{q}}\right\|_{{Y(\mathcal{X})}}
\lesssim\|f\|_{Y(\mathcal{X})},
\end{align*}
where $U(x,y):=\min\{\mu(B(x,\rho(x,y))),\,\mu(B(y,\rho(x,y)))\}$
for any $x,y\in\mathcal{X}$ and
the implicit positive constants are
independent of $f$,
which is applied to ten specific ball quasi-Banach function spaces
and hence is of wide generality.
In particular,
when $Y(\mathcal{X})=L^q(\mathbb{R}^n)$
with $q\in[1,\infty)$, the above formula is
closely related to the celebrated result of Maz'ya and Shaposhnikova
in 2002. We also establish the above
representation formula on domains of $\mathcal{X}$.
The main novelty lies in
proposing two new concepts,
namely the weak reverse doubling
condition (for $\mathcal{X}$)
and the weak measure density condition
(for domains of $\mathcal{X}$), which are
proved to be necessary in some sense.
In addition, we find an interesting fact that,
when the underlying space
under consideration is bounded,
the above
Maz'ya--Shaposhnikova-type limit always tends to zero.	
}
\end{minipage}
\end{center}

\vspace{0.2cm}

\tableofcontents

\vspace{0.2cm}

\section{Introduction}
The study of the \emph{homogeneous fractional
Sobolev spaces} $\dot{W}^{s,p}(\mathbb{R}^n)$
with $s\in(0,1)$ and $p\in[1,\infty)$,
which is defined to be the set of all measurable
functions $f$ on $\mathbb{R}^n$ such that
\begin{align}\label{wsp}
\|f\|_{\dot{W}^{s,p}(\mathbb{R}^n)}:
=\left[\int_{\mathbb{R}^n}\int_{\mathbb{R}^n}\frac{|f(x)-f(y)|^p}
{|x-y|^{n+sp}}\,dx\,dy\right]^{\frac{1}{p}}
=:\left\|\frac{f(x)-f(y)}{|x-y|^{\frac{n}{p}+s}}
\right\|_{L^p(\mathbb{R}^n\times\mathbb{R}^n)}
\end{align}
is finite,
has garnered significant attention and there has been a
growing interests on this topic
over the past few decades. Indeed,
these spaces become an indispensable tool and plays a pivotal role
in both pure mathematical analysis and applied mathematics;
in particular, they have important applications
to many problems related to harmonic analysis,
variational calculus, and
partial differential equations;
see \cite{bbm01,bbm02,b02,bm18,bm19,bn16,bvy21b,npv12,dm21a,ls11}.

A well-known \emph{defect} of the family
of fractional semi-norms
$\{\|\cdot\|_{\dot{W}^{s,p}(\mathbb{R}^n)}\}_{s\in(0,1)}$
is the deficiency of the continuity at the endpoint $s=0$
[corresponding to $\|\cdot\|_{L^p(\mathbb{R}^n)}$];
see, for example, \cite{bsvy22,gy21}.
In 2002, Maz'ya and Shaposhnikova \cite{ms02}
gave a clever way to repair the above defect by
proving that, for any $f\in L^p(\mathbb{R}^n)
\cap\bigcup_{s\in(0,1)}\dot{W}^{s,p}(\mathbb{R}^n)$,
\begin{align}\label{msformula}
\lim_{s\to 0^+}
s^\frac{1}{p}\left\|f\right\|_{\dot{W}^{s,p}(\mathbb{R}^n)}
=C_{(p,n)}
\left\|f\right\|_{L^p(\mathbb{R}^n)},
\end{align}
where
$C_{(p,n)}$ is a positive constant depending only on both $p$ and $n$.
Here, and thereafter, $s\to 0^+$ means $s\in(0,1)$ and $s\to0$.
The identity \eqref{msformula} is nowadays
referred to as the celebrated
Maz'ya--Shaposhnikova formula.
We refer to
\cite{acps20,cdknp2023,dltyy23,hpxz24,k23,msv22,pme17}
for more studies of the Maz'ya--Shaposhnikova formula \eqref{msformula}
as well as its various variations and
to \cite{bsy21,bsvy21,bvy21a,dlyyz21b,ddfp23,dm20,dssvsy23,f22,k22,p04}
for more related studies
of (fractional) Sobolev spaces.

Since its inception,
the Maz'ya--Shaposhnikova formula \eqref{msformula}
has undergone continuous development and
has been generalized to many different non-Euclidean settings;
see, for instance,
Carnot groups \cite{cmsv21}, anisotropic spaces \cite{gh22},
non-collapsed RCD metric measure spaces \cite{hpxz24},
and metric measure spaces with $N$-dimensional Hausdorff measure \cite{h24}.
It is quite natural
to consider the Maz'ya--Shaposhnikova formula \eqref{msformula}
on more general settings, for example,
spaces of homogeneous type in the sense of
Coifman and Weiss \cite{cw71,cw77},
namely quasi-metric spaces
equipped with a doubling measure;
see Definition~\ref{Dehts} for the precise
definition of spaces of homogeneous type.
These spaces serve  as
a natural framework for the
studies of the boundedness of
operators and the
real-variable theory of function spaces in harmonic analysis.
In particular, the study of Sobolev spaces on spaces
of homogeneous type has attracted considerable attention
and led to significant progress;
see, for example, \cite{gkz2013,h2024,hkst2015,kkn2025}.

On the other hand,
there have been systematic development and extensive applications of
various (fractional)
Sobolev-type spaces in the study of partial differential
equations and harmonic analysis
over the past few decades,
such as weighted (or variable exponent) Sobolev spaces,
Orlicz--Sobolev spaces,
and Morrey--Sobolev spaces.
Significantly, these fundamental function spaces
(weighted Lebesgue spaces, variable exponent Lebesgue spaces,
Orlicz spaces, and Morrey spaces), on which
the aforementioned Sobolev spaces were built,
have been comprehensively incorporated into the unifying framework
of ball quasi-Banach function spaces
introduced by Sawano et al. \cite{shyy17};
see Definition~\ref{Debqfs} for the precise
definition of ball quasi-Banach function spaces.
It turns out that this inclusive framework
systematizes the investigations of many important function spaces
in both harmonic analysis and
partial differential equations.
For more studies related to ball quasi-Banach function
spaces,
we refer to \cite{cwyz20,lyh22,s18,shyy17,zhyy21}
for Hardy spaces associated with them,
to \cite{dgpyyz21,dlyyz21a,lyyzz24,pyyz24,zyy24,zlyyz24,zyy23,zyy23a,zyy23b}
for Sobolev spaces associated with them,
and to \cite{h21,hcy21,tyyz21,wyy20,zwyy20}
for the boundedness of operators on them.

Motivated by the aforementioned works,
in this article we devote to generalizing the
Maz'ya--Shaposhnikova formula \eqref{msformula}
to the setting of ball quasi-Banach function spaces
in the sense of \cite{shyy17}, respectively,
on the space of homogeneous type $(\mathcal{X},\rho,\mu)$
and on the domain $\Omega\subseteqq\mathcal{X}$,
where $(\mathcal{X},\rho,\mu)$ satisfies the additional weak reverse
doubling condition and $\Omega$ is assumed to satisfy
a weak version of the measure density condition.
It is worth mentioning that an analogous Maz'ya--Shaposhnikova formula
on ball quasi-Banach function spaces on $\mathbb{R}^n$
has been established in \cite{pyyz24}.

The Maz'ya--Shaposhnikova formula yields strikingly different results
depending on whether the underlying space or the domain
under consideration is bounded or unbounded.
We will first present an unexpected result for the bounded case;
see Definitions~\ref{1635} and~\ref{2200}
for the precise definitions of
$Y(\Omega)$ and $\dot{W}^{s,q}_Y(\Omega)$,
respectively. Recall that $\Omega\subseteqq\mathcal{X}$
is said to be \emph{bounded} with respect to the quasi-metric $\rho$
if $\mathrm{diam\,}(\Omega):=\sup_{x,y\in\Omega}\rho(x,y)<\infty$.

\begin{theorem}\label{Thm1}
Let $q\in(0,\infty)$,
$(\mathcal{X},\rho,\mu)$
be a quasi-metric measure space,
$\Omega\subseteqq\mathcal{X}$
a bounded measure set,
$Y(\mathcal{X})$ a ball quasi-Banach function space,
and $Y(\Omega)$ the restriction of $Y(\mathcal{X})$
to $\Omega$.
Then, for any
$f\in\bigcup_{s\in(0,1)}\dot{W}^{s,q}_Y(\Omega)$,
\begin{align*}
\lim_{s\to0^+}s^\frac{1}{q}
\left\|\left\{\int_\Omega\frac{|f(\cdot)-f(y)|^q}
{U(\cdot,y)[\rho(\cdot,y)]^{sq}}\,d\mu(y)
\right\}^\frac{1}{q}\right\|_{Y(\Omega)}=0,
\end{align*}
where, for any $x,y\in\mathcal{X}$,
$U(x,y):=\min\{\mu(B(x,\rho(x,y))),\,\mu(B(y,\rho(x,y)))\}$.
\end{theorem}

\begin{remark}
\begin{enumerate}
\item[\rm(i)]
When $\mathrm{diam\,}(\mathcal{X})<\infty$,
Theorem~\ref{Thm1} holds for any measurable set $\Omega\subseteqq\mathcal{X}$,
in particular, for the case $\Omega=\mathcal{X}$.
\item[\rm(ii)]
For the case where the underlying space or the domain
is unbounded, see Theorems~\ref{Thm2},~\ref{Thm3},
\ref{MS2}, and~\ref{MS3} below,
where we assume that the underlying space $(\mathcal{X},\rho,\mu)$
is a space of homogeneous type. However, when the underlying space
is unbounded but is not a space of homogeneous type
(for example, non-doubling metric spaces),
it is still unknown whether the
Maz'ya--Shaposhnikova representation as in Theorem~\ref{Thm2} holds.
\end{enumerate}
\end{remark}

Apart from Theorem~\ref{Thm1},
in the remainder of this section,
we always assume that the underlying space
$(\mathcal{X},\rho,\mu)$ is a space of homogeneous type
with $\mathrm{diam\,}(\mathcal{X})=\infty$
which is equivalent to $\mu(\mathcal{X})=\infty$
(see \cite[Lemma 5.1]{ny1997}). To establish
the Maz'ya--Shaposhnikova representation
in this case,
we begin with introducing the following weak
reverse doubling condition for measures
on quasi-metric measure spaces.

\begin{definition}\label{951}
A quasi-metric measure space
$(\mathcal{X},\rho,\mu)$ is said to
satisfy the
\emph{weak reverse doubling} (for short, WRD) \emph{condition}
if there exist $x_0\in\mathcal{X}$ and
$\lambda,C_{(\mu)}\in(1,\infty)$
such that
\begin{align}\label{reverseD}
\varliminf_{r\to\infty}\frac{\mu(B(x_0,\lambda r))}{\mu(B(x_0,r))}\ge
C_{(\mu)}.
\end{align}
\end{definition}

Here, and thereafter, we
denote by $\varliminf$ and $\varlimsup$ respectively
the \emph{limit inferior}
and the \emph{limit superior}.

\begin{remark}
\begin{enumerate}
\item[\rm(i)]
By \eqref{reverseD} and $C_{(\mu)}>1$, we are easy to prove that,
if the quasi-metric measure space $(\mathcal{X},\rho,\mu)$
satisfies the WRD condition, then $\mathrm{diam\,}(\mathcal{X})=\infty$.
Moreover, if a space of homogeneous type $(\mathcal{X},\rho,\mu)$
satisfies the WRD condition, then $\mu(\mathcal{X})=\infty$
because, in this case, $\mathrm{diam\,}(\mathcal{X})=\infty$
if and only if $\mu(\mathcal{X})=\infty$ (see \cite[Lemma 5.1]{ny1997}).
\item[\rm(ii)]
Assume that	$(\mathcal{X},\rho,\mu)$
is a quasi-metric measure space
satisfying the WRD condition. That is,
\eqref{reverseD} holds for some $x_0\in\mathcal{X}$ and $\lambda\in(1,\infty)$.
Then, from the quasi-triangle inequality of $\rho$
[see Definition~\ref{Deqms}(iii)],
we are easy to show that there exists
$\widetilde{\lambda}\in(0,\infty)$ such that,
for any $x\in\mathcal{X}$,
\eqref{reverseD} also holds with $x_0$ and $\lambda$ replaced,
respectively, by $x$ and $\widetilde{\lambda}$;
see Proposition~\ref{1017}.
\item[\rm(iii)]
Assume that $\mathrm{diam\,}(\mathcal{X})=\infty$.
Recall that $(\mathcal{X},\rho,\mu)$ is called an
\emph{$\mathrm{RD}$-space}, introduced by Han et al. \cite{hmy08},
if $(\mathcal{X},\rho,\mu)$ is a space of homogeneous type
and satisfies the following
\emph{reverse-doubling condition}:
there exists $C\in(1,\infty)$
such that, for any $x\in\mathcal{X}$ and
$r\in(0,\infty)$,
\begin{align}\label{RD}
C\mu(B(x,r))
\le\mu(B(x,2r)).
\end{align}
If the measure $\mu$ satisfies
\eqref{RD}, then $\mu$ always satisfies the WRD condition,
which can be deduced from Proposition~\ref{12346781}.
But the converse is not necessary to be true;
see the example in \cite[Proposition 2.3]{wyy23}
which satisfies the WRD condition but does not satisfy \eqref{RD}.
\end{enumerate}
\end{remark}

We also need the following basic and mild assumption on ball
Banach function spaces, which is the first one of
two key assumptions of the whole article; see
Definition~\ref{defXpie} for the precise definition
of $[Y(\mathcal{X})]'$.

\begin{assumption}\label{ma}
$Y(\mathcal{X})$ is a ball Banach function
space and
the Hardy--Littlewood maximal operator
${\mathcal M}$ is bounded on $[Y(\mathcal{X})]'$.
\end{assumption}

We next present the second main result of this article
as follows; see Definitions~\ref{Detuhua},~\ref{855}, and~\ref{quotient}
for the precise definitions of $Y^{\frac{1}{p}}(\mathcal{X})$,
$C_{\mathrm{b}}^\beta(\mathcal{X})$, and $Y(\mathcal{X})/\mathbb{R}$,
respectively.

\begin{theorem}\label{Thm2}
Let $(\mathcal{X},\rho,\mu)$ be a space
of homogeneous type satisfying the $\rm{WRD}$ condition.
Let $0<q\leq p<\infty$
and $Y(\mathcal{X})$ be a ball
quasi-Banach function space
such that
$Y^{\frac{1}{p}}(\mathcal{X})$
satisfies Assumption~\ref{ma}.
Then there exist two positive
constants $C$ and $\widetilde{C}$
such that the following
assertions hold.
\begin{enumerate}
\item[\rm (i)]
Let $\beta\in(0,\infty)$.
For any $f\in C_{\mathrm{b}}^\beta(\mathcal{X})\cap
\bigcup_{s\in(0,1)}\dot{W}^{s,q}_Y({\mathcal{X}})$,
\begin{align}\label{eq1.5}
C\|f\|_{Y(\mathcal{X})}
&\notag\leq \varliminf_{s \to 0^+}
s^{\frac{1}{q}}\left\| \left\{\int_{\mathcal{X}}
\frac{|f(\cdot)-f(y)|^q}{U(\cdot,y)[\rho(\cdot,y)]^{sq}}
\, d\mu(y)  \right\}^{\frac{1}{q}}\right\|_{{Y(\mathcal{X})}}\\
&\leq \varlimsup_{s \to 0^+}
s^{\frac{1}{q}}\left\| \left\{\int_{\mathcal{X}}
\frac{|f(\cdot)-f(y)|^q}{U(\cdot,y)[\rho(\cdot,y)]^{sq}}
\, d\mu(y) \right\}^{\frac{1}{q}}\right\|_{{Y(\mathcal{X})}}
\leq\widetilde{C}\|f\|_{Y(\mathcal{X})}.
\end{align}
\item[\rm (ii)]
If $Y(\mathcal{X})$ is further assumed to have an
absolutely continuous quasi-norm,
then, for any
$f\in [Y(\mathcal{X})/\mathbb{R}]\cap\bigcup_{s\in(0,1)}
\dot{W}^{s,q}_Y({\mathcal{X}})$,
\begin{align}\label{eq1.6}
C\|f\|_{Y(\mathcal{X})/\mathbb{R}}
&\notag\leq \varliminf_{s \to 0^+}
s^{\frac{1}{q}}\left\|
\left\{\int_{\mathcal{X}}
\frac{|f(\cdot)-f(y)|^q}{U(\cdot,y)[\rho(\cdot,y)]^{sq}}
\, d\mu(y) \right\}^{\frac{1}{q}}\right\|_{{Y(\mathcal{X})}}\\
&\leq \varlimsup_{s \to 0^+}
s^{\frac{1}{q}}\left\|\left\{\int_{\mathcal{X}}
\frac{|f(\cdot)-f(y)|^q}{U(\cdot,y)[\rho(\cdot,y)]^{sq}}
\, d\mu(y) \right\}^{\frac{1}{q}}\right\|_{{Y(\mathcal{X})}}
\leq\widetilde{C}\|f\|_{Y(\mathcal{X})/\mathbb{R}}.
\end{align}
\end{enumerate}
\end{theorem}

\begin{remark}\label{2250}
\begin{enumerate}
\item[\rm(i)]
There are many examples of $(\mathcal{X},\rho,\mu)$
such that Theorem~\ref{Thm2} holds,
namely spaces of homogeneous type satisfying the WRD condition,
such as Euclidean
spaces with $A_\infty$-weights (in particular, the standard Euclidean space
equipped with the Lebesgue measure),
Ahlfors-regular spaces, Lie groups of polynomial growth,
and Carnot--Carath\'eodory spaces with doubling measure,
finite dimensional Banach spaces,
and MCP spaces
(see \cite[Example~4.8]{h2024} for the definition
of MCP spaces).

\item[\rm(ii)]
We point out that the WRD condition in Theorem~\ref{Thm2}
is only used to obtain the lower estimate, namely
the first inequality
in both \eqref{eq1.5} and \eqref{eq1.6}.

\item[\rm(iii)]
If $\mathcal{X}=\mathbb{R}^n$,
$\rho$ is the standard Euclidean distance,
and $\mu$ is the $n$-dimensional Lebesgue measure,
then Theorem \ref{Thm2} reduces to \cite[Theorem 2.12]{pyyz24}.
Moreover, we point out that,
even in this special case,
the two constants $C$ and $\widetilde{C}$
in \eqref{eq1.5} [or in \eqref{eq1.6}]
are not equal;
see \cite[Example~2.20]{pyyz24}
for the case where
$Y(\mathcal{X})$ is the weighted Lebesgue space.

\item[\rm(iv)]
By \cite[Remark~2.13(iii)]{pyyz24},
we find that the requirement $q\leq p$
in \cite[Theorem~2.12]{pyyz24} is sharp,
which, combined with (i) of the present
remark, further implies that
the requirement $0<q\leq p<\infty$
in Theorem \ref{Thm2} is also \emph{sharp}
in some sense.

\item[\rm(v)]
The WRD condition in Theorem~\ref{Thm2}
is \emph{sharp} in the following sense: These exists an
example of spaces of homogeneous type
which does not satisfy the WRD condition,
and the conclusion of Theorem~\ref{Thm2} in this case fails;
see Remark~\ref{sharp1}.

\item[\rm(vi)]
Han \cite[Theorem~3.11]{h2024}
proved a generalized Maz'ya--Shaposhnikova formula
associated with general mollifiers
on non-compact metric measure spaces.
We point out that \cite[Theorem~3.11]{h2024}
and Theorem~\ref{Thm2} cannot completely cover each other.
Indeed, the volume growth condition on the measure
in \cite[Assumption~3.4]{h2024}
is stronger than the assumptions on $\mu$ in Theorem~\ref{Thm2}.
Moreover, the mollifiers $\{\rho_n(x,y)\}_{n\in\mathbb{N}}$
in \cite[Theorem~3.11]{h2024}
are assumed to be radial, that is,
for any $x,y\in X$,
$\rho_n(x,y)$ depends only on
the distance between $x$ and $y$.
In Theorem~\ref{Thm2}, we use the
mollifier $U(x,y)[\rho(x,y)]^{sq}$
which might not be radial with respect to $\rho(x,y)$.
Thus, Theorem~\ref{Thm2}
even in the special case $Y(\mathcal{X}):=L^p(\mathcal{X})$
cannot be covered by
\cite[Theorem~3.11]{h2024}.
On the other hand, it turns out that,
with some assumptions on the measure and the mollifiers,
Han \cite[Theorem~3.11]{h2024} obtained an equality
with the exact constant $2$ in \cite[(3.20)]{h2024}.
But, Theorem~\ref{Thm2} cannot give an equality
with an exact constant and hence Theorem~\ref{Thm2}
in the special case $Y(\mathcal{X}):=L^p(\mathcal{X})$
cannot cover \cite[Theorem~3.11]{h2024}.
\end{enumerate}
\end{remark}

When $p\in(0,\infty)$ and the Hardy--Littlewood maximal operator
${\mathcal M}$ is not known to be bounded
on the associate space of
$Y^{\frac{1}{p}}(\mathcal{X})$,
Theorem \ref{Thm2} seems to be inapplicable to this case;
for example, Morrey spaces.
To establish the corresponding conclusion of Theorem \ref{Thm2}
in the setting of the ball quasi-Banach function space $Y(\mathcal{X})$
on which $\mathcal{M}$ may not be
bounded on $[Y^{\frac{1}{p}}(\mathcal{X})]'$,
we need the following assumption,
which is the second (last) one of two key assumptions of the whole
article.

\begin{assumption}\label{ma2}
$Y(\mathcal{X})$ is a ball Banach function
space and
the Hardy--Littlewood maximal operator
${\mathcal M}$ is \emph{endpoint bounded} on $[Y(\mathcal{X})]'$,
that is, there exists a sequence
$\{\theta_m\}_{m\in{\mathbb N}}$ in $(0,1)$
such that $\lim_{m\to\infty}\theta_m=1$,
$Y^{\frac{1}{\theta_m}}(\mathcal{X})$ is a ball Banach function space and
the Hardy--Littlewood maximal operator ${\mathcal M}$
is bounded on its associate space $[Y^{\frac{1}{\theta_m}}(\mathcal{X})]'$
for any $m\in{\mathbb N}$, and
\begin{align*}
\lim_{m\to\infty}
\left\|{\mathcal M}\right\|_{[Y^{\frac{1}{\theta_m}}(\mathcal{X})]'\to
[Y^{\frac{1}{\theta_m}}(\mathcal{X})]'} <\infty.
\end{align*}
\end{assumption}

Based on Assumption~\ref{ma2} instead of Assumption~\ref{ma},
we present the third main result of this article.
Recall that a measure $\mu$ is said to be
\emph{Borel-semiregular}
if $\mu$ is a Borel measure on $\mathcal{X}$
and, for any measurable set $E\subseteqq\mathcal{X}$,
there exists a Borel set $F\subseteqq\mathcal{X}$
such that
$\mu(E\Delta F):=\mu(E\setminus F)+\mu(F\setminus E)=0$;
see, for instance, \cite[Definition~3.9]{AlMi2015}.

\begin{theorem}\label{Thm3}
Let $(\mathcal{X},\rho,\mu)$ be a space
of homogeneous type satisfying the $\rm{WRD}$ condition.
Let $0 < q \leq p < \infty$
and $Y(\mathcal{X})$ be a ball
quasi-Banach function space
such that
$Y^{\frac{1}{p}}(\mathcal{X})$
satisfies Assumption~\ref{ma2}.
\begin{enumerate}
\item[\rm (i)]
Let $\beta\in(0,\infty)$. Then,
for any $f\in C_{\mathrm{b}}^\beta(\mathcal{X})
\cap\bigcup_{s\in(0,1)} \dot{W}^{s,q}_Y({\mathcal{X}})$,
\eqref{eq1.5} holds.
\item[\rm (ii)]
If $Y(\mathcal{X})$ has an
absolutely continuous quasi-norm
and $\mu$ is Borel-semiregular,
then, for any
$f\in [Y(\mathcal{X})/\mathbb{R}] \cap\bigcup_{s\in(0,1)}
\dot{W}^{s,q}_Y({\mathcal{X}})$, \eqref{eq1.6} holds.
\end{enumerate}
\end{theorem}

\begin{remark}
If $\mathcal{X}=\mathbb{R}^n$,
$\rho$ is the standard Euclidean distance,
and $\mu$ is the $n$-dimensional Lebesgue measure,
then Theorem \ref{Thm3} reduces to \cite[Theorem 2.16]{pyyz24}.
\end{remark}

Now, we consider the case where $\mathcal{X}$ in \eqref{eq1.5}
and \eqref{eq1.6} is replaced by an unbounded measurable set
$\Omega\subseteqq\mathcal{X}$. We will also assume that
the underlying space under consideration
$(\mathcal{X},\rho,\mu)$ is a space of homogeneous type
satisfying the WRD condition.
To state the main results in this case,
we first introduce the concept of the weak
measure density condition.
Although motivated by the measure density condition in \cite[(1)]{hkt2008},
our definition differs slightly from its formulation.

\begin{definition}\label{1646}
Let $(\mathcal{X},\rho,\mu)$ be an unbounded space of homogeneous type.
A measurable set $\Omega\subseteqq\mathcal{X}$
is said to satisfy the \emph{weak measure density}
(for short, WMD) \emph{condition}
if there exist $x_0\in\mathcal{X}$ and $C_0\in(0,1]$ such that
\begin{align}\label{AssumptionOmega}
\varliminf_{r\to\infty}\frac{\mu(B(x_0,r)\cap\Omega)}{\mu(B(x_0,r))}\ge C_0.
\end{align}
\end{definition}

\begin{remark}\label{1933}
\begin{enumerate}
\item[\rm(i)]
If $\Omega\subseteqq\mathcal{X}$ satisfies the WMD condition,
then we are easy to show that
$\mathrm{diam\,}(\Omega)=\infty$ and hence $\Omega$ is unbounded.
In this case, we also find $\mu(\Omega)=\infty$.
Indeed, notice that $\mu(\mathcal{X})=\infty$ if and
only if $\mathrm{diam\,}(\mathcal{X})=\infty$ (see \cite[Lemma 5.1]{ny1997}).
If $\mu(\Omega)<\infty$, then
\begin{align*}
\varliminf_{r\to\infty}\frac{\mu(B(x_0,r)\cap\Omega)}{\mu(B(x_0,r))}
\leq\varliminf_{r\to\infty}\frac{\mu(\Omega)}{\mu(B(x_0,r))}=0,
\end{align*}
which contradicts \eqref{AssumptionOmega}.
\item[\rm(ii)]
If $(\mathcal{X},\rho,\mu)$ is bounded,
then $\Omega\subseteqq\mathcal{X}$ satisfies \eqref{AssumptionOmega}
if and only if there exists
a positive constant $C$ such that
\begin{align*}
\frac{\mu(\Omega)}{\mu(\mathcal{X})}\ge C,
\end{align*}
which is equivalent to $\mu(\Omega)\in(0,\infty)$
and hence the WMD condition in this case becomes a trivial condition.
\item[\rm(iii)]
It is easy to see that $\mathcal{X}$ itself satisfies the WMD condition
automatically.
\item[\rm(iv)]
In Definition~\ref{1646}, if we replace $x_0$ by
any fixed point in $\mathcal{X}$,
we then obtain an equivalent formulation of the WMD condition;
see \eqref{1549}.
\end{enumerate}
\end{remark}

Next, via assuming that the domain under consideration
satisfies the above WMD condition,
we give the last two main results of this article as follows.
For any $\beta\in(0,\infty)$,
denote by $C_{\mathrm{b}}^\beta(\Omega)$
the set of all functions on $\Omega$ satisfying that
there exists a function $g\in C_{\mathrm{b}}^\beta(\mathcal{X})$
such that $g=f$ on $\Omega$ and $g\equiv0$ on $\mathcal{X}\setminus\Omega$.

\begin{theorem}\label{MS2}
Let $0 < q \leq p < \infty$,
$(\mathcal{X},\rho,\mu)$ be a space of homogeneous type satisfying
the $\rm{WRD}$ condition,
$Y(\mathcal{X})$ be
a ball quasi-Banach function space,
and $Y(\Omega)$ the restriction of $Y(\mathcal{X})$
to $\Omega$, where $\Omega\subseteqq\mathcal{X}$
satisfies the $\rm{WMD}$ condition.
Assume that $Y^\frac{1}{p}(\mathcal{X})$ satisfies Assumption~\ref{ma}.
Then there exist two positive
constants $C$ and $\widetilde{C}$
such that the following
statements hold.
\begin{enumerate}
\item[\rm (i)]
Let $\beta\in(0,\infty)$.
For any $f\in C_{\mathrm{b}}^\beta(\Omega)\cap
\bigcup_{s\in(0,1)}\dot{W}^{s,q}_Y(\Omega)$,
\begin{align}\label{1750}
C\|f\|_{Y(\Omega)}&\leq
\varliminf_{s\to0^+}s^\frac{1}{q}
\left\|\left\{\int_\Omega\frac{|f(\cdot)-f(y)|^q}
{U(\cdot,y)[\rho(\cdot,y)]^{sq}}\,dy\right\}^\frac{1}{q}\right\|_{Y(\Omega)}
\nonumber\\
&\leq\varlimsup_{s\to0^+}s^\frac{1}{q}
\left\|\left\{\int_\Omega\frac{|f(\cdot)-f(y)|^q}
{U(\cdot,y)[\rho(\cdot,y)]^{sq}}\,dy\right\}^\frac{1}{q}\right\|_{Y(\Omega)}
\leq\widetilde{C}\|f\|_{Y(\Omega)}.
\end{align}
\item[\rm (ii)]
If $Y(\mathcal{X})$ is further assumed to have an
absolutely continuous quasi-norm,
then, for any
$f\in [Y(\Omega)/\mathbb{R}]\cap\bigcup_{s\in(0,1)}\dot{W}^{s,q}_Y(\Omega)$,
\begin{align}\label{1751}
C\|f\|_{Y(\Omega)/\mathbb{R}}&\leq
\varliminf_{s\to0^+}s^\frac{1}{q}
\left\|\left\{\int_\Omega\frac{|f(\cdot)-f(y)|^q}
{U(\cdot,y)[\rho(\cdot,y)]^{sq}}\,dy\right\}^\frac{1}{q}\right\|_{Y(\Omega)}
\nonumber\\
&\leq\varlimsup_{s\to0^+}s^\frac{1}{q}
\left\|\left\{\int_\Omega\frac{|f(\cdot)-f(y)|^q}
{U(\cdot,y)[\rho(\cdot,y)]^{sq}}\,dy\right\}^\frac{1}{q}\right\|_{Y(\Omega)}
\leq\widetilde{C}\|f\|_{Y(\Omega)/\mathbb{R}}.
\end{align}
\end{enumerate}
\end{theorem}

\begin{remark}
\begin{enumerate}
\item[\rm(i)]
Similar to Remark~\ref{2250}(iv),
the assumption $0<q\leq p<\infty$
in Theorem~\ref{MS2} is also \emph{sharp}
in some sense.
\item[\rm(ii)]
The WMD condition in Theorem~\ref{MS2}
is \emph{sharp} in the following sense: These exists an
example of open sets
which does not satisfy the WMD condition,
and the conclusion of Theorem~\ref{MS2} in this case fails;
see Remark~\ref{sharp2}.
\item[(iii)]
From Remark~\ref{1933}(iii),
we infer that Theorem~\ref{MS2} with $\Omega=\mathcal{X}$
holds, which in this case reduces to Theorem~\ref{Thm2}.
\end{enumerate}
\end{remark}

\begin{theorem}\label{MS3}
Let $0<q\leq p<\infty$,
$(\mathcal{X},\rho,\mu)$ be a space of homogeneous type satisfying
the $\rm{WRD}$ condition,
$Y(\mathcal{X})$ be
a ball quasi-Banach function space,
and $Y(\Omega)$ the restriction of $Y(\mathcal{X})$
to $\Omega$, where $\Omega\subseteqq\mathcal{X}$
satisfies the $\rm{WMD}$ condition.
Assume that $Y^\frac{1}{p}(\mathcal{X})$ satisfies Assumption~\ref{ma2}.
\begin{enumerate}
\item[\rm (i)]
Let $\beta\in(0,\infty)$. Then,
for any $f\in C_{\mathrm{b}}^\beta(\mathcal{X})
\cap\bigcup_{s\in(0,1)}\dot{W}^{s,q}_Y({\mathcal{X}})$,
\eqref{1750} holds.
\item[\rm (ii)]
If $Y(\mathcal{X})$ has an
absolutely continuous quasi-norm
and $\mu$ is Borel-semiregular,
then, for any
$f\in [Y(\mathcal{X})/\mathbb{R}] \cap\bigcup_{s\in(0,1)}
\dot{W}^{s,q}_Y({\mathcal{X}})$, \eqref{1751} holds.
\end{enumerate}
\end{theorem}

As a direct corollary of Theorems~\ref{MS2} and~\ref{MS3}
on the standard Euclidean space $\mathbb{R}^n$, we have the following result.
Notice that, in this case, for any $x,y\in\mathbb{R}^n$,
\begin{align*}
\rho(x,y)=|x-y|
\ \ \text{and}\ \
U(x,y)=\left|B(\mathbf{0},1)\right||x-y|^n,
\end{align*}
where $|B(\mathbf{0},1)|$ is the Lebesgue measure of
the unit ball in $\mathbb{R}^n$.
Denote by $C_{\mathrm{c}}(\Omega)$
the set of all continuous functions on $\Omega$
with compact support in $\Omega$.

\begin{corollary}
Let $0<q\leq p<\infty$, $Y(\mathbb{R}^n)$
be a ball quasi-Banach function space,
and $Y(\Omega)$ the restriction of $Y(\mathbb{R}^n)$
to $\Omega$,
where $\Omega\subseteqq\mathbb{R}^n$ satisfies the $\rm{WMD}$ condition.
Assume that $Y^\frac{1}{p}(\mathbb{R}^n)$ satisfies Assumption~\ref{ma}
or Assumption~\ref{ma2}.
Then there exist two positive
constants $C$ and $\widetilde{C}$
such that the following
assertions hold.
\begin{enumerate}
\item[\rm (i)]
For any $f\in C_{\mathrm{c}}(\Omega)\cap
\bigcup_{s\in(0,1)}\dot{W}^{s,q}_Y(\Omega)$,
\begin{align*}
C\|f\|_{Y(\Omega)}&\leq
\varliminf_{s\to0^+}s^\frac{1}{q}
\left\|\left\{\int_\Omega\frac{|f(\cdot)-f(y)|^q}
{|\cdot-y|^{n+sq}}\,dy\right\}^\frac{1}{q}\right\|_{Y(\Omega)}\\
&\leq\varlimsup_{s\to0^+}s^\frac{1}{q}
\left\|\left\{\int_\Omega\frac{|f(\cdot)-f(y)|^q}
{|\cdot-y|^{n+sq}}\,dy\right\}^\frac{1}{q}\right\|_{Y(\Omega)}
\leq\widetilde{C}\|f\|_{Y(\Omega)}.\nonumber
\end{align*}
\item[\rm (ii)]
If $Y(\mathbb{R}^n)$ is further assumed to have an
absolutely continuous quasi-norm,
then, for any
$f\in [Y(\Omega)/\mathbb{R}]\cap\bigcup_{s\in(0,1)}\dot{W}^{s,q}_Y(\Omega)$,
\begin{align*}
C\|f\|_{Y(\Omega)/\mathbb{R}}&\leq
\varliminf_{s\to0^+}s^\frac{1}{q}
\left\|\left\{\int_\Omega\frac{|f(\cdot)-f(y)|^q}
{|\cdot-y|^{n+sq}}\,dy\right\}^\frac{1}{q}\right\|_{Y(\Omega)}\\
&\leq\varlimsup_{s\to0^+}s^\frac{1}{q}
\left\|\left\{\int_\Omega\frac{|f(\cdot)-f(y)|^q}
{|\cdot-y|^{n+sq}}\,dy\right\}^\frac{1}{q}\right\|_{Y(\Omega)}
\leq\widetilde{C}\|f\|_{Y(\Omega)/\mathbb{R}}.
\end{align*}
\end{enumerate}
\end{corollary}

We prove Theorem~\ref{Thm1} in Section~\ref{3212},
Theorem~\ref{Thm2} in Subsections~\ref{sec-2-1}
and~\ref{sec-2-2},
Theorem~\ref{Thm3} in Subsection~\ref{sec-2-3},
Theorem~\ref{MS2} in Subsection~\ref{3.1},
and Theorem~\ref{MS3} in Subsection~\ref{3.2}.
We show these
by using the method of extrapolation
to overcome the difficulties caused by that
the quasi-norm of $Y(\mathcal{X})$ has no explicit expression.
The primary novelty of this article lies in that, to obtain these results,
we propose two new concepts, namely the weak reverse doubling
condition (see Definition~\ref{951})
and the weak measure density condition
(see Definition~\ref{1646}), which are proved to be necessary in some sense.
These results are of wide generality and are
applied to ten specific ball quasi-Banach function spaces,
most of which are new.

It is also worth mentioning that
there exists another way to remedy the aforementioned
defect of the family $\{\|\cdot\|_{\dot{W}^{s,p}(\mathbb{R}^n)}\}_{s\in(0,1)}$	
at the endpoint $s=0$, that is, to let $s=0$ in \eqref{wsp}
and replace the product $L^p$-norm in \eqref{wsp}
by the product weak $L^p$-norm  simultaneously;
see \cite{bsvy22,gy21,pyyz24} for the Euclidean
space case and \cite{sb24a,gh22,kms2023,m23,tnyyz24}
for the metric measure space case.
We will not pursue more about this issue here.

The remainder of this article is organized as follows.

Section \ref{3212} is devoting to showing
Theorems~\ref{Thm1},~\ref{Thm2}, and~\ref{Thm3}.
We first recall some preliminary concepts on
spaces of homogeneous type
and ball quasi-Banach function spaces.
We then prove Theorem~\ref{Thm1} directly.
In Subsection~\ref{sec-2-1},
we first recall some preliminaries on
Muckenhoupt weights
and then give the proof of Theorem~\ref{Thm2}.
To this end,
we first use
some essential properties of
Muckenhoupt weights
to show a similar
upper estimate of \eqref{eq1.5}
in weighted Lebesgue spaces; see Lemma \ref{keythmofwss}.
Combining this
and a key extrapolation lemma (see Lemma \ref{4.6}),
the latter of which is a bridge connecting
ball quasi-Banach function spaces and weighted Lebesgue spaces,
we obtain the upper estimate of \eqref{eq1.5};
see Theorem \ref{04021646}.
In Subsection \ref{sec-2-2},
applying the WRD condition of $\mu$,
we first establish a key lower estimate
on spaces of homogeneous type;
see Lemma \ref{2225}.
Using this, the Fatou property of $Y(\mathcal{X})$,
the corresponding estimates on
weighted Lebesgue spaces,
and a method of extrapolation,
we establish the lower estimate of \eqref{eq1.5};
see Theorem~\ref{04021714}.
In Subsection \ref{sec-2-3},
we use the endpoint bounded assumption on the Hardy--Littlewood
maximal operator and a density argument to prove
Theorem~\ref{Thm3}.
In Subsection~\ref{example},
we provide an example to show that
the Maz'ya--Shaposhnikova representation of quasi-norms
of ball quasi-Banach function spaces fails
on certain spaces of homogeneous type,
which indicates the
weak reverse doubling assumption on $\mu$
is necessary in some sense (see Proposition~\ref{835}).

In Section \ref{32123},
we first give some preliminaries on
the restriction of ball quasi-Banach function spaces.
In Subsection~\ref{3.1},
we establish an exquisite lower estimate on domains
to prove Theorem~\ref{MS2}.
In Subsection~\ref{3.2},
we also consider the corresponding result
under the endpoint bounded assumption on the Hardy--Littlewood
maximal operator. That is, we give the proof of Theorem~\ref{MS3}.
In Subsection~\ref{example2}, we construct an example to show that
the Maz'ya--Shaposhnikova representation of quasi-norms
of ball quasi-Banach function spaces fails
on certain domains $\Omega\subseteqq\mathbb{R}^n$,
which indicates the
weak measure density assumption on $\Omega$
is necessary in some sense (see Proposition~\ref{1233}).

In Section \ref{sec4}, we apply
the main results of this article,
namely Theorems~\ref{Thm2},~\ref{Thm3},
\ref{MS2}, and~\ref{MS3}, to
ten specific ball quasi-Banach function
spaces, namely
weighted Lebesgue spaces,
variable Lebesgue spaces,
weighted Lorentz spaces,
Orlicz spaces,
generalized Morrey spaces,
generalized block spaces,
generalized Lorentz--Morrey spaces,
generalized
Lorentz-block spaces,
generalized Orlicz--Morrey spaces,
and
generalized
Orlicz-block spaces,
most of which are new.
To obtain the last two applications,
we first prove that the generalized Orlicz--Morrey space
and the generalized Orlicz-block space are mutually
the associate space (see Theorem \ref{thm:K dual}).
Moreover, we also establish the boundedness
of the Hardy--Littlewood maximal operator
on generalized Orlicz--Morrey spaces (see Theorem \ref{thm:Mf OM})
and generalized Orlicz-block spaces (see Theorem \ref{thm:Mf OB}).
Duo to the generality and the
flexibility of the main results of this
article,
more applications (for example,
to some newfound ball quasi-Banach function spaces) are
predictable.

At the end of this section, we make some conventions on symbols.
Let ${\mathbb N}:=\{1,2,\ldots\}$ and $\mathbb{Z}_+:=\mathbb{N}\cup\{0\}$.
Let $(\mathcal{X},\rho,\mu)$ be a space of homogeneous type.
For any $E\subseteqq\mathcal{X}$,
we denote by $E^{\complement}$
the set $\mathcal{X}\setminus E$
and by $\mathbf{1}_E$ its \emph{characteristic function}.
The ball $B(x,r)$ of $\mathcal{X}$,
with center $x\in\mathcal{X}$ and radius $r\in(0,\infty)$,
is defined by setting
$B(x,r):=\{y\in\mathcal{X}:\rho(x,y)<r\}$.
For any $\lambda\in(0,\infty)$ and any ball $B:=B(x_B,r_B)$
with $x_B \in\mathcal{X}$ and $r_B \in (0,\infty)$,
let $\lambda B:=B(x_B,\lambda r_B)$.
For any $x,y\in\mathcal{X}$, let
$$
U(x,y):=\min\left\{\mu(B(x,\rho(x,y))),\,\mu(B(y,\rho(x,y)))\right\}.
$$
The diameter $\mathrm{diam\,}(\mathcal{X})$ of $(\mathcal{X},\rho)$
is defined by setting
$\mathrm{diam\,}(\mathcal{X}):=\sup_{x,y\in\mathcal{X}}\rho(x,y)$.
Denote by $\mathscr{M}(\mathcal{X})$
the set of all $\mu$-measurable functions on $\mathcal{X}$.
For any $f\in\mathscr{M}(\mathcal{X})$,
the support $\mathrm{supp\,}f$ of $f$ is defined by setting
$\mathrm{supp\,}f:=\{x\in\mathcal{X}:f(x)\neq 0\}$.
For any $p\in(0,\infty]$, we denote by the \emph{symbol}
$L_{{\mathop\mathrm{\,loc\,}}}^p(\mathcal{X})$
the set of all locally $p$-integrable functions on $\mathcal{X}$.
The \emph{Hardy--Littlewood maximal operator}
${\mathcal M}$
is defined by setting, for any
$f\in L_{{\mathop\mathrm{\,loc\,}}}^1(\mathcal{X})$
and $x\in\mathcal{X}$,
\begin{equation*}
{\mathcal M}(f)(x):=\sup_{B\ni x}\frac1{\mu(B)}\int_B|f(y)|\,d\mu(y),
\end{equation*}
where the supremum is taken over
all balls $B\subset\mathcal{X}$ containing $x$.
We always use $C$ to denote a \emph{positive constant},
independent of the main parameters involved,
but perhaps varying from line to line.
We also use
$C_{(\alpha,\beta,\ldots)}$ to denote a positive
constant depending on the indicated parameters $\alpha,
\beta,\ldots.$ The \emph{symbol} $f\lesssim g$ means
that $f\leq Cg$. If $f\lesssim g$ and $g\lesssim f$,
we then write $f\sim g$. If $f\leq Cg$ and $g=h$ or
$g\leq h$, we then write $f\lesssim g=h$ or $f\lesssim g\leq h$.
For any $q\in[1,\infty]$, we denote by $q'$ its
\emph{conjugate exponent}, that is, $\frac{1}{q}+\frac{1}{q'}=1$.
We also use $\varliminf$ and $\varlimsup$ to denote the
\emph{limit inferior} and the \emph{limit superior}, respectively.
Finally, in all proofs we
consistently retain the symbols introduced
in the original theorem (or related statement).

\section{Maz'ya--Shaposhnikova Representation
on Spaces of \\Homogeneous Type Satisfying
the WRD Condition}\label{3212}

In this section, we aim to establish the
Maz'ya--Shaposhnikova representation
of quasi-norms of ball quasi-Banach function spaces
on spaces of homogeneous type satisfying the WRD condition.
In Subsection~\ref{sec-2-1}, we show the upper estimate
of Theorem~\ref{Thm2}.
In Subsection~\ref{sec-2-2}, we prove the lower estimate
of Theorem~\ref{Thm2}.
In Subsection~\ref{sec-2-3}, we show Theorem~\ref{Thm3}.
In Subsection~\ref{example}, we prove that the
weak reverse doubling assumption on $\mu$ is necessary in some sense.

We begin with recalling the following
definition of quasi-metric spaces.

\begin{definition}\label{Deqms}	
A \emph{quasi-metric space} $(\mathcal{X},\rho)$ is a non-empty
set ${\mathcal{X}}$ equipped with a \emph{quasi-metric} $\rho$,
that is, a nonnegative function defined on $\mathcal{X} \times \mathcal{X}$
satisfies that, for any $x,y,z\in \mathcal{X}$,
\begin{enumerate}
\item[\rm (i)]
${\mathcal{\rho}}(x,y)=0$ if and only if $x=y$,
\item[\rm (ii)]
${\mathcal{\rho}}(x,y)={\mathcal{\rho}}(y,x)$,
\item[\rm (iii)]
there exists a constant $K_0 \in [1,\infty)$,
independent of $x,y,$ and $z$, such that
\begin{align*}
{\mathcal{\rho}}(x,z)\leq K_0\left[{\mathcal{\rho}}(x,y)
+{\mathcal{\rho}}(y,z)\right].
\end{align*}
\end{enumerate}
\end{definition}

Now, we recall the concept of spaces of homogeneous
type in the sense of Coifman and Weiss \cite{cw71,cw77}.

\begin{definition}\label{Dehts}	
A triplet $(\mathcal{X},\rho,\mu)$ is called
a \emph{space of homogeneous type} if $(\mathcal{X},\rho)$
is a quasi-metric space and $\mu$ is a nonnegative measure,
defined on a $\sigma$-algebra of subsets of $\mathcal{X}$
which contains all $\rho$-balls, such that the following
\emph{doubling condition} holds, that is,
there exists a constant $L_{(\mu)}\in [1,\infty)$ such that,
for any $x\in\mathcal{X}$ and $r\in(0,\infty)$,
\begin{align}\label{dc}
0<\mu(B(x,2r)) \leq L_{(\mu)}\mu(B(x,r))<\infty.
\end{align}
\end{definition}

\begin{remark}
By \eqref{dc}, we conclude that,
for any $\lambda \in [1,\infty)$, $x\in\mathcal{X}$,
and $r\in(0,\infty)$,
\begin{align}\label{ud}
\mu(B(x,\lambda r))\leq L_{(\mu)} \lambda^d \mu(B(x,r)),
\end{align}
where $d:=\log_{2}L_{(\mu)}$ is called the
\emph{upper dimension} of $\mathcal{X}$.
Moreover,
if the cardinality of the set $\mathcal{X}$ is at least $2$,
then, from \eqref{dc},
we deduce that $L_{(\mu)}\in(1,\infty)$ and hence $d\in(0,\infty)$;
see \cite[p.\,72]{AlMi2015}.
\end{remark}

Next, we recall the definition of H\"older spaces
with bounded support.

\begin{definition}\label{855}
Let $\beta\in(0,\infty)$.
The \emph{H\"older space
$C_{\mathrm{b}}^\beta(\mathcal{X})$ with bounded support}
is defined to be the set of all
complex-valued functions $f$ on $\mathcal{X}$ with bounded support such that
\begin{align*}
\|f\|_{\dot{C}^\beta(\mathcal{X})}
:=\sup_{\genfrac{}{}{0pt}{}{x,y\in\mathcal{X}}{x\neq y}}
\frac{|f(x)-f(y)|}{[\rho(x,y)]^\beta}<\infty.
\end{align*}
\end{definition}

The following definition can be found in
\cite[Definition 2.4]{yhyy22},
which with $\mathcal{X}=\mathbb{R}^n$
was introduced by Sawano et al. \cite{shyy17}.
Let $\mathscr{M}(\mathcal{X})$
be the set of all
measurable functions on $\mathcal{X}$.

\begin{definition}\label{Debqfs}
A quasi-Banach space $Y(\mathcal{X})$, equipped with
a quasi-norm $\|\cdot\|_{Y(\mathcal{X})}$
which makes sense for all $f\in\mathscr{M}(\mathcal{X})$,
is called a \emph{ball quasi-Banach function space}		
if
\begin{enumerate}
\item[\rmfamily(i)]
$f\in\mathscr{M}(\mathcal{X})$ and $\|f\|_{Y(\mathcal{X})}=0$ imply that $f=0$
$\mu$-almost everywhere on $\mathcal{X}$;
\item[\rmfamily(ii)]
$f,g\in\mathscr{M}(\mathcal{X})$ and $|g|\leq|f|$
$\mu$-almost everywhere on $\mathcal{X}$ imply
$\|g\|_{Y(\mathcal{X})}\le\|f\|_{Y(\mathcal{X})}$;
\item[\rmfamily(iii)]
for any sequence
$\{f_{m}\}_{m\in\mathbb{N}}$ in $\mathscr{M}(\mathcal{X})$
and any function $f\in\mathscr{M}(\mathcal{X})$ satisfying that
$0\leq f_m\uparrow f$ $\mu$-almost
everywhere on $\mathcal{X}$, one has
$\|f_m\|_{Y(\mathcal{X})}\uparrow\|f\|_{Y(\mathcal{X})}$;
\item[\rmfamily(iv)]
$\mathbf{1}_B\in Y(\mathcal{X})$ for any ball $B\subset\mathcal{X}$.
\end{enumerate}
Moreover, a ball quasi-Banach function space $Y(\mathcal{X})$ is
called a \emph{ball Banach function space} if
\begin{enumerate}
\item[\rmfamily(v)]
for any $f,g\in Y(\mathcal{X})$,
$\|f+g\|_{Y(\mathcal{X})}\leq \|f\|_{Y(\mathcal{X})}+ \|g\|_{Y(\mathcal{X})}$;
\item[\rmfamily(vi)]
for any ball $B\subset\mathcal{X}$, there exists
$C_{(B)}\in(0,\infty)$ such that, for any $f\in Y(\mathcal{X})$,
\begin{align*}
\int_B|f(x)|\,d{\mu(x)}\leq C_{(B)}\|f\|_{Y(\mathcal{X})}.
\end{align*}
\end{enumerate}
\end{definition}

\begin{remark}
\begin{enumerate}
\item[\rm(i)]
By (i) and (ii) of Definition~\ref{Debqfs},
we find that, if $f\in \mathscr{M}(\mathcal{X})$,
then $\|f\|_{Y(\mathcal{X})}=0$ if and only if
$f=0$ $\mu$-almost everywhere on $\mathcal{X}$.

\item[\rm(ii)]
In Definition \ref{Debqfs},
if we replace any ball $B$ by
any bounded measurable set $E$,
we then obtain its another equivalent formulation.

\item[\rm(iii)]
As pointed in \cite[Theorem 2]{dfmn21},
both (ii) and (iii) of Definition \ref{Debqfs}
imply that any ball quasi-Banach function space is complete;
see also \cite{ln20231}.

\item[\rm(iv)]
From (ii) and (iv) of Definition \ref{Debqfs},
we infer that, for any $\beta\in(0,\infty)$,
$C_{\mathrm{b}}^\beta(\mathcal{X})\subset Y(\mathcal{X})$.

\item[\rm(v)]
In Definition~\ref{Debqfs},
if we replace (iv)
by the \emph{saturation property} that,
for any measurable set $E\subseteqq\mathcal{X}$
with $\mu(E)\in(0,\infty)$, there exists a measurable set $F\subseteqq E$
with $\mu(F)\in(0,\infty)$ satisfying that $\mathbf{1}_F\in Y(\mathcal{X})$,
we then obtain the definition of quasi-Banach function spaces
in \cite{ln20231}.
Moreover, by \cite[Proposition~2.5 and Remark~2.6]{zyy23a}
(see also \cite[Proposition~4.21]{n23}
and \cite[Remark~2.8]{ln20231}),
we conclude that, if $Y(\mathcal{X})$
satisfies the additional assumption that
the Hardy--Littlewood maximal operator $\mathcal{M}$ is weakly bounded on
one of its convexification,
then the definition of quasi-Banach function
spaces in \cite{ln20231} coincides with
the definition of ball quasi-Banach function spaces. Thus,
under this additional assumption,
working with quasi-Banach function spaces in \cite{ln20231}
or ball quasi-Banach function spaces
would yield exactly the same results.
\end{enumerate}
\end{remark}

Now, we recall the definition of the
$p$-convexification of ball quasi-Banach
function spaces; see
\cite[Definition 2.10]{yhyy21}.

\begin{definition}\label{Detuhua}
Let $Y(\mathcal{X})$ be a ball quasi-Banach function space.
For any given $p\in(0,\infty)$,
the \emph{$p$-convexification} $Y^p(\mathcal{X})$
of $Y(\mathcal{X})$ is defined by setting
$Y^p(\mathcal{X}):=\{f\in\mathscr{M}(\mathcal{X}):
|f|^p\in Y(\mathcal{X})\}$
equipped with the quasi-norm $\|f\|_{Y^p(\mathcal{X})}
:=\|\,|f|^p\|_{Y(\mathcal{X})}^{1/p}$.
\end{definition}

The following concept of
the absolutely continuous quasi-norm
of ball quasi-Banach function spaces
with $\mathcal{X}=\mathbb{R}^n$ is precisely
\cite[Definition 3.2]{wyy20}.	

\begin{definition}
A ball quasi-Banach function space $Y(\mathcal{X})$ is said to have an
\emph{absolutely continuous quasi-norm} on $\mathcal{X}$
if, for any $f\in Y(\mathcal{X})$
and any sequence $\{E_j\}_{j\in{\mathbb N}}$ of measurable
sets in $\mathcal{X}$
satisfying that $\mathbf{1}_{E_j}\to 0$ $\mu$-almost
everywhere as $j\to\infty $,
one has $\|f\mathbf{1}_{E_j}\|_{Y(\mathcal{X})}\to 0$ as $j\to\infty$.
\end{definition}

The following definition can be found in
\cite[(2.9)]{yhyy21}.

\begin{definition}\label{defXpie}
Let $Y(\mathcal{X})$ be a ball Banach function space.
The \emph{associate space} (also called the
\emph{K\"othe dual}) $[Y(\mathcal{X})]'$ of $Y(\mathcal{X})$
is defined by setting
\begin{align*}
[Y(\mathcal{X})]':=\left\{f\in\mathscr M({\mathcal{X}}):
\|f\|_{[Y(\mathcal{X})]'}
:=\sup_{\genfrac{}{}{0pt}{}{g\in Y(\mathcal{X})}{\|g\|_{Y(\mathcal{X})}\leq 1}}
\|fg\|_{L^1({\mathcal{X}})}<\infty\right\},
\end{align*}
where $\|\cdot\|_{[Y(\mathcal{X})]'}$ is called the
\emph{associate norm} of $\|\cdot\|_{Y(\mathcal{X})}$.
\end{definition}

The following concept of
quotient spaces in the case where
$\mathcal{X}$ is a space of homogeneous type
can be found in \cite[Definition 2.10]{tnyyz24}
and where $\mathcal{X}=\mathbb{R}^n$
can be found in \cite[Definition 2.9]{pyyz24}.

\begin{definition}\label{quotient}
Let $Y(\mathcal{X})$ be a ball quasi-Banach function space.
The \emph{quotient space $Y(\mathcal{X})/\mathbb{R}$} is defined
to be the set of all equivalent classes $[f]$ of
measurable functions on ${\mathcal{X}}$ such that
\begin{align*}
\|[f]\|_{Y(\mathcal{X})/\mathbb{R}}:=\inf_{a\in\mathbb{R}}
\|f+a\|_{Y(\mathcal{X})} <\infty.
\end{align*}
\end{definition}

Throughout this article, we simply use the symbols
$f\in Y(\mathcal{X})/\mathbb{R}$ and $\|f\|_{Y(\mathcal{X})/\mathbb{R}}$,
respectively, to replace $[f]\in Y(\mathcal{X})/\mathbb{R}$
and $\|[f]\|_{Y(\mathcal{X})/\mathbb{R}}$.
In addition, it is worth noticing that
$f\in Y(\mathcal{X})/\mathbb{R}$ if and only
if there exists $a\in\mathbb{R}$
such that $f+a\in Y(\mathcal{X})$. Moreover,
we point out that, under some assumptions on $Y(\mathcal{X})$,
if $f\in Y(\mathcal{X})/\mathbb{R}$, then
there exists exactly one constant $a$ such that
$\|f+a\|_{Y(\mathcal{X})}<\infty$.
To formalize this, we state the following proposition.

\begin{proposition}\label{2357}
Let $(\mathcal{X},\rho,\mu)$ be a space of homogeneous type
with $\mathrm{diam\,}(\mathcal{X})=\infty$ and let
$Y(\mathcal{X})$ be a ball quasi-Banach function space.
Assume that there exists $p\in(0,\infty)$ such that
$Y^{\frac{1}{p}}(\mathcal{X})$
satisfies Assumption~\ref{ma}.
Then $f\in Y(\mathcal{X})$ implies that,
for any $a\in\mathbb{R}\setminus\{0\}$,
$f+a\notin Y(\mathcal{X})$.
\end{proposition}

\begin{proof}
We first claim that, for any $a\in\mathbb{R}\setminus\{0\}$,
$a\notin Y(\mathcal{X})$. We show this claim by
contradiction. Assume that there exists $a\in\mathbb{R}\setminus\{0\}$
such that $a\in Y(\mathcal{X})$. Then we have $1\in Y(\mathcal{X})$
and hence $1\in Y^{\frac{1}{p}}(\mathcal{X})$.
From this and Definition~\ref{defXpie},
we deduce that, for any $f\in[Y^{\frac{1}{p}}(\mathcal{X})]'$,
\begin{align*}
\left\|f\right\|_{L^1(\mathcal{X})}
\leq\left\|1\right\|_{Y^{\frac{1}{p}}(\mathcal{X})}
\left\|f\right\|_{[Y^{\frac{1}{p}}(\mathcal{X})]'}<\infty
\end{align*}
and hence $f\in L^1(\mathcal{X})$.
Thus, $[Y^{\frac{1}{p}}(\mathcal{X})]'\subset L^1(\mathcal{X})$.
On the other hand, by the assumption that $Y^{\frac{1}{p}}(\mathcal{X})$
is a ball Banach function space and \cite[Proposition~2.3]{shyy17}
(whose proof remains true on spaces of homogeneous type),
we find that $[Y^{\frac{1}{p}}(\mathcal{X})]'$
is also a ball Banach function space
and hence, for any ball $B\subset\mathcal{X}$,
$\mathbf{1}_B\in [Y^{\frac{1}{p}}(\mathcal{X})]'$.
From \cite[Lemma 3.32]{tnyyz24}, we infer that, for any ball
$B\subset\mathcal{X}$, $\mathcal{M}\mathbf{1}_B\notin L^1(\mathcal{X})$.
This, together with the proven conclusion that
$[Y^{\frac{1}{p}}(\mathcal{X})]'\subset L^1(\mathcal{X})$,
implies that
$\mathcal{M}\mathbf{1}_B\notin [Y^{\frac{1}{p}}(\mathcal{X})]'$,
which contradicts the assumption that $\mathcal{M}$ is bounded on
$[Y^{\frac{1}{p}}(\mathcal{X})]'$. Thus,
for any $a\in\mathbb{R}\setminus\{0\}$,
$a\notin Y(\mathcal{X})$.
This finishes the proof of the above claim.

Let $f\in Y(\mathcal{X})$. We prove the present proposition by
contradiction. Assume that there exists $a\in\mathbb{R}\setminus\{0\}$
such that $f+a\in Y(\mathcal{X})$. Then, by the quasi-triangle
inequality of $\|\cdot\|_{Y(\mathcal{X})}$,
we conclude that $
\|a\|_{Y(\mathcal{X})}\lesssim\|f\|_{Y(\mathcal{X})}
+\|f+a\|_{Y(\mathcal{X})}<\infty$
and hence $a\in Y(\mathcal{X})$,
which contradicts the above claim. Thus, $f+a\notin Y(\mathcal{X})$
for any $a\in\mathbb{R}\setminus\{0\}$.
This finishes the proof of Proposition~\ref{2357}.
\end{proof}

Next, we introduce the concept of homogeneous fractional
ball quasi-Banach Sobolev spaces.

\begin{definition}\label{12311709}
Let $s\in(0,1)$, $q\in(0,\infty)$, and $Y(\mathcal{X})$
be a ball quasi-Banach function space.
The \emph{homogeneous fractional ball quasi-Banach Sobolev space
$\dot{W}^{s,q}_Y({\mathcal{X}})$} is defined
to be the set of all $f\in\mathscr{M}(\mathcal{X})$ such that
\begin{equation*}
\|f\|_{\dot{W}^{s,q}_Y({\mathcal{X}})}:=
\left\|\left\{ \int_{\mathcal{X}}
\frac{|f(\cdot)-f(y)|^q}{U(\cdot,y)[\rho(\cdot,y)]^{sq}}
\, d\mu(y) \right\}^{\frac{1}{q}}\right\|_{Y(\mathcal{X})}<\infty.
\end{equation*}
\end{definition}

\begin{remark}
If $Y(\mathcal{X}):=L^q(\mathbb{R}^n)$ with $q\in[1,\infty)$,
then ${\dot{W}^{s,q}_{{Y}}(\mathcal{X})}$
is precisely $\dot{W}^{s,q}({\mathbb{R}^n})$.
\end{remark}

Now, we show Theorem~\ref{Thm1}.

\begin{proof}[Proof of Theorem~\ref{Thm1}]
Let $f\in\bigcup_{s\in(0,1)}\dot{W}^{s,q}_Y(\Omega)$.
Then there exists $s_0\in(0,1)$ such that
$f\in\dot{W}^{s_0,q}_Y(\Omega)$.
Since that $\Omega$ is bounded, it follows that,
for any $x,y\in\Omega$,
$\rho(x,y)\leq\mathrm{diam\,}(\Omega)<\infty$.
By this and $f\in\dot{W}^{s_0,q}_Y(\Omega)$,
we find that, for any $s\in(0,s_0)$,
\begin{align*}
s^\frac{1}{q}
\left\|\left\{\int_\Omega\frac{|f(\cdot)-f(y)|^q}
{U(\cdot,y)[\rho(\cdot,y)]^{sq}}\,dy
\right\}^\frac{1}{q}\right\|_{Y(\Omega)}
&=s^\frac{1}{q}
\left\|\left\{\int_\Omega\frac{|f(\cdot)-f(y)|^q}
{U(\cdot,y)[\rho(\cdot,y)]^{s_0q}}[\rho(\cdot,y)]^{(s_0-s)q}\,dy
\right\}^\frac{1}{q}\right\|_{Y(\Omega)}\\
&\leq s^\frac{1}{q}\left[\mathrm{diam\,}(\Omega)\right]^{s_0-s}
\|f\|_{\dot{W}^{s_0,q}_Y(\Omega)}\to0
\end{align*}
as $s\to0^+$. This finishes the proof of Theorem~\ref{Thm1}.
\end{proof}

\subsection{Proof of Theorem \ref{Thm2}: Upper Estimate}\label{sec-2-1}

This subsection is devoted to establishing
the upper bound estimate
of Theorem \ref{Thm2}.
We first recall the concepts of both Muckenhoupt weights
and weighted Lebesgue spaces.

\begin{definition}
\begin{enumerate}
\item[\rm(i)]
A nonnegative locally integral function $\omega$ is called an
\emph{$A_1({\mathcal{X}})$-weight},
denoted by $\omega\in A_1({\mathcal{X}})$, if
\begin{equation*}
[\omega]_{A_1({\mathcal{X}})}:=
\sup_{B\subset{\mathcal{X}}}
\frac{1}{\mu(B)}\int_B\omega(x)\,d\mu(x)
\left[
\mathop{\mathrm{ess\,inf\,}}
_{y\in B}\omega(y)\right]^{-1}<\infty,
\end{equation*}
where the supremum is taken over all
balls $B\subset{\mathcal{X}}$.

\item[\rm(ii)]
Let $p\in(1,\infty)$.
A nonnegative locally integral function $\omega$ is called an
\emph{$A_p({\mathcal{X}})$-weight}, denoted
by $\omega\in A_p({\mathcal{X}})$, if
\begin{equation*}
[\omega]_{A_p({\mathcal{X}})}:=\sup_{B\subset{\mathcal{X}}}
\frac{1}{\mu(B)}\int_{B}\omega(x)\,d\mu(x)
\left\{\frac{1}{\mu(B)}\int_B\left[\omega(x)\right]^
{\frac{1}{1-p}}\,d\mu(y)\right\}^{p-1}<\infty,
\end{equation*} 	
where the supremum is taken over all
balls $B\subset{\mathcal{X}}$.
\end{enumerate}
\end{definition}

\begin{definition}\label{twl}
Let $p\in(0,\infty)$ and $\omega$ be a weight.
The \emph{weighted Lebesgue space} $L^p_\omega(\mathcal{X})$
is defined to be the set of all $f\in\mathscr{M}(\mathcal{X})$ such that
\begin{align*}
\|f\|_{L^p_\omega(\mathcal{X})}:=
\left[\int_{\mathcal{X}}\left|f(x)\right|^p
\omega(x)\,d\mu(x)\right]^{\frac{1}{p}}<\infty.
\end{align*}
\end{definition}

Next, we recall some basic properties
and related conclusions of Muckenhoupt weights;
see, for instance, \cite[Lemma 2.1]{bbd20},
\cite[Chapter 1]{st89},
and \cite[Subsection 3.4]{ins19}.

\begin{lemma}\label{Lemma2.1}
Let $p\in[1,\infty)$ and $\omega\in A_p({\mathcal{X}}).$
\begin{enumerate}
\item[{\rm(i)}]
If $p=1$,
then $[\omega]_{A_1(\mathcal{X})}\in[1,\infty)$
and, for $\mu$-almost every $x\in\mathcal{X}$,
$\mathcal{M}\omega(x)\leq[\omega]_{A_1(\mathcal{X})}\omega(x)$.

\item[{\rm(ii)}]
\begin{align*}
[\omega]_{A_p({\mathcal{X}})}
=\sup_{B\subset{\mathcal{X}},\,\|f\mathbf{1}_B\|_{L^p_\omega({\mathcal{X}})}\in(0,\infty)}
\frac{[\frac{1}{\mu(B)}\int_B|f(t)|\,d\mu(t)]^p}
{\frac{1}{\omega(B)}\int_B|f(t)|^p\omega(t)\,d\mu(t)},
\end{align*}
where the supremum is taken over all balls
$B\subset{\mathcal{X}}$ and all $\mu$-measurable functions $f$
such that $\|f\mathbf{1}_B\|_{L^p_\omega(\mathcal{X})}\in(0,\infty)$.

\item[{\rm(iii)}]
For any $\lambda\in(1,\infty)$
and any ball $B\subset{\mathcal{X}}$,
one has
\begin{align*}
\frac{\omega(\lambda B)}{\omega(B)}\leq [\omega]_{A_p({\mathcal{X}})}
\left[\frac{\mu(\lambda B)}{\mu(B)}\right]^p.
\end{align*}

\item[{\rm(iv)}]
For any $q\in[p,\infty),$
one has $\omega\in A_q({\mathcal{X}})$ and
$[\omega]_{A_q({\mathcal{X}})}\leq
[\omega]_{A_p({\mathcal{X}})}$.
\end{enumerate}
\end{lemma}

The following is the main theorem
of this subsection.

\begin{theorem}\label{04021646}
Let $(\mathcal{X},\rho,\mu)$ be a space of homogeneous type.
Let $0 < q \leq p < \infty$
and $Y(\mathcal{X})$ be a ball
quasi-Banach function space
such that
$Y^{\frac{1}{p}}(\mathcal{X})$
satisfies Assumption~\ref{ma}.
Then there exists a positive constant $C$ such that,
for any $f\in [Y(\mathcal{X})/\mathbb{R}]
\cap\bigcup_{s\in(0,1)} \dot{W}^{s,q}_Y({\mathcal{X}})$,
\begin{align*}
\varlimsup_{s \to 0^+}
s^{\frac{1}{q}}\left\|\left\{\int_{\mathcal{X}}
\frac{|f(\cdot)-f(y)|^q}{U(\cdot,y)[\rho(\cdot,y)]^{sq}}
\, d\mu(y)\right\}^{\frac{1}{q}}\right\|_{{Y(\mathcal{X})}}
\leq C
\|{\mathcal M}\|_{[Y^{\frac{1}{p}}(\mathcal{X})]'\to
[Y^{\frac{1}{p}}(\mathcal{X})]'}^{\frac{2}{p}}
\|f\|_{Y(\mathcal{X})/\mathbb{R}}\notag.
\end{align*}
\end{theorem}
To show Theorem \ref{04021646},
we need the following two technical lemmas.

\begin{lemma}\label{12202129}
Let $(\mathcal{X},\rho,\mu)$ be a space of homogeneous type.
Let $q\in(0,\infty)$ and $Y(\mathcal{X})$
be a ball quasi-Banach function space.
Then, for any $R\in(0,\infty)$ and $f\in \bigcup_{s\in(0,1)}
\dot{W}^{s,q}_Y({\mathcal{X}})$,
\begin{align*}
\lim_{s \to 0^+}
s^{\frac{1}{q}}\left\|\left\{\int_{B(\cdot,R)}
\frac{|f(\cdot)-f(y)|^q}{U(\cdot,y)[\rho(\cdot,y)]^{sq}}
\, d\mu(y)
\right\}^{\frac{1}{q}}\right\|_{{Y(\mathcal{X})}}=0.
\end{align*}
\end{lemma}

\begin{proof}
From $f\in \bigcup_{s\in(0,1)}
\dot{W}^{s,q}_Y({\mathcal{X}})$,
we deduce that there exists $s_0\in(0,1)$
such that $f\in \dot{W}^{s_0,q}_Y({\mathcal{X}})$,
which further implies that
\begin{align*}
0&\leq \varliminf_{s \to 0^+}
s^{\frac{1}{q}}\left\|
\left\{\int_{B(\cdot,R)}
\frac{|f(\cdot)-f(y)|^q}{U(\cdot,y)[\rho(\cdot,y)]^{sq}}
\, d\mu(y)
\right\}^{\frac{1}{q}}\right\|_{{Y(\mathcal{X})}}\\
\notag
& \leq \varlimsup_
{s \to 0^+}s^{\frac{1}{q}}\left\|
\left\{\int_{B(\cdot,R)}
\frac{|f(\cdot)-f(y)|^q}{U(\cdot,y)[\rho(\cdot,y)]^{s_0q}}
[\rho(\cdot,y)]^{(s_0-s)q}\, d\mu(y)
\right\}^{\frac{1}{q}}\right\|_{{Y(\mathcal{X})}}\\
\notag& \leq \varlimsup_{s \to 0^+}
s^{\frac{1}{q}}R^{(s_0-s)}
\|f\|_{\dot{W}^{s_0,q}_Y({\mathcal{X}})}
=0.
\end{align*}
This finishes the proof of Lemma \ref{12202129}.
\end{proof}

\begin{lemma}\label{12202120}
Let $(\mathcal{X},\rho,\mu)$ be a space of homogeneous type.
Let $0 < q \leq p < \infty$
and $Y(\mathcal{X})$ be a ball
quasi-Banach function space
such that
$Y^{\frac{1}{p}}(\mathcal{X})$
satisfies Assumption~\ref{ma}.
Then there exists a positive constant $C$ such that,
for any $R\in(1,\infty)$ and
$f\in Y({\mathcal{X}}) / \mathbb{R}$,
\begin{align}\label{06171257}
\sup_{s\in(0,1)} s^{\frac{1}{q}}\left\|
\left\{\int_{[B(\cdot,R)]^\complement}
\frac{|f(\cdot)-f(y)|^q}{U(\cdot,y)[\rho(\cdot,y)]^{sq}}
\, d\mu(y)
\right\}^{\frac{1}{q}}\right\|_{{Y(\mathcal{X})}}
\leq C \|{\mathcal M}\|_{[Y^{\frac{1}{p}}(\mathcal{X})]'\to
[Y^{\frac{1}{p}}(\mathcal{X})]'}^{\frac{2}{p}}
\|f\|_{Y(\mathcal{X})/ \mathbb{R}}.
\end{align}
\end{lemma}

To prove Lemma \ref{12202120}, we need
the following lemma, which is a special case of
Lemma~\ref{12202120} with
$Y(\mathcal{X})=L^p_{\omega}(\mathcal{X})$.

\begin{lemma}\label{keythmofwss}
Let $(\mathcal{X},\rho,\mu)$ be a space of homogeneous type.
Let $0 < q \leq p < \infty$ and
$\omega \in A_1({\mathcal{X}})$.
Then there exists a positive constant $C$
such that, for any $R\in(1,\infty)$
and $f\in L^p_{\omega}(\mathcal{X})$,
\begin{align}\label{keyineqwss}
\sup_{s\in(0,1)}s^{\frac{p}{q}}\int_{\mathcal{X}}
\left\{\int_{[B(x,R)]^\complement}
\frac{|f(x)-f(y)|^q}{U(x,y)[\rho(x,y)]^{sq}}
\, d\mu(y)\right\}^{\frac{p}{q}}\omega(x)\,d\mu(x)
\leq C[\omega]_{A_1({\mathcal{X}})}^2
\|f\|^{p}_{L^p_{\omega}(\mathcal{X}) }.
\end{align}
\end{lemma}

To show Lemma \ref{keythmofwss}, we need the
following useful lemma, which also plays
an important role in the proof of the lower
estimate in Theorem \ref{Thm2}.

\begin{lemma}\label{02201609}
Let $(\mathcal{X},\rho,\mu)$ be a space of homogeneous type,
$0 < q \leq p < \infty$, and
$\omega \in A_1({\mathcal{X}})$.
Then there exists a positive constant $C$ such that,
for any $s\in(0,1)$, $R\in(0,\infty)$,
and $f\in L^p_{\omega}(\mathcal{X})$,
\begin{align*}
s^{\frac{p}{q}}\int_{\mathcal{X}}
\left\{\int_{[B(x,R)]^\complement}
\frac{|f(y)|^q}{U(x,y)[\rho(x,y)]^{sq}}
\, d\mu(y)\right\}^{\frac{p}{q}}\omega(x)\,d\mu(x)
\leq C\left[\frac{s 2^{sq}}{R^{sq}
(2^{sq}-1)}\right]^{\frac{p}{q}}
[\omega]_{A_1({\mathcal{X}})}^2
\|f\|_{L^p_{\omega}(\mathcal{X})}^p.
\end{align*}
\end{lemma}

\begin{proof}
Let $s\in(0,1)$, $R\in(0,\infty)$,
and $f\in L^p_{\omega}(\mathcal{X})$.
By \eqref{dc}, both (ii) and (iv) of Lemma~\ref{Lemma2.1},
H\"older's inequality,
and Fubini's theorem,
we conclude that
\begin{align}\label{12202108}
&\nonumber\int_{\mathcal{X}}
\left\{\int_{[B(x,R)]^\complement}
\frac{|f(y)|^q}{U(x,y)[\rho(x,y)]^{sq}}
\, d\mu(y)\right\}^{\frac{p}{q}}\omega(x)\,d\mu(x)\\
\notag&\quad= \int_{\mathcal{X}} \left\{\sum_{j=0}^{\infty}
\int_{B(x,2^{j+1}R) \setminus B(x,2^{j}R)}
\frac{|f(y)|^q}{U(x,y)[\rho(x,y)]^{sq}}
\, d\mu(y)\right\}^{\frac{p}{q}}\omega(x)\,d\mu(x)\\
\notag&\quad\lesssim\int_{\mathcal{X}}
\left[\sum_{j=0}^{\infty} \frac{(2^{j}R)^{-sq}}{\mu(B(x,2^{j+1}R))}
\int_ {B(x,2^{j+1}R)}
{|f(y)|^q}\,d\mu(y)\right]^{\frac{p}{q}}\omega(x)\,d\mu(x)\\
\notag&\quad\leq \int_{\mathcal{X}}
\left\{\sum_{j=0}^{\infty} (2^{j}R)^{-sq}
\left[\frac{[\omega]_{A_{p/q}({\mathcal{X}})}}{\omega(B(x,2^{j+1}R))}
\int_{B(x,2^{j+1}R)}{|f(y)|^p}\omega(y)\,d\mu(y)\right]
^{\frac{q}{p}}\right\}^{\frac{p}{q}}
\omega(x)\,d\mu(x)\\
\notag&\quad\leq  [\omega]_{A_{1}({\mathcal{X}})}\int_{\mathcal{X}}
\left\{\sum_{j=0}^{\infty} (2^{j}R)^{-sq}
\right.\\ \notag
&\left.\qquad\times
\left[\frac{1}{\omega(B(x,2^{j+1}R))}
\int_{B(x,2^{j+1}R)}	{|f(y)|^p}\omega(y)\,d\mu(y)
\right]^{\frac{q}{p}}\right\}
^{\frac{p}{q}}\omega(x)\,d\mu(x)\\
\notag&\quad\leq[\omega]_{A_{1}({\mathcal{X}})}\int_{\mathcal{X}}\left\{
\left[\sum_{j=0}^{\infty} (2^{j}R)^{-sq}
\right]^{\frac{1}{\left(\frac{p}{q}\right)'}}\right.\\
\notag&\qquad\left.\times\left[\sum_{j=0}^{\infty}
\frac{(2^{j}R)^{-sq}}{\omega(B(x,2^{j+1}R))}
\int_{B(x,2^{j+1}R)}|f(y)|^p\omega(y)\,d\mu(y)
\right]^{\frac{q}{p}}\right\}^{\frac{p}{q}}
\omega(x)\,d\mu(x)\\
\notag&\quad=[\omega]_{A_{1}({\mathcal{X}})}
\left[\frac{ 2^{sq}}{R^{sq} (2^{sq}-1)}\right]^{\frac{p}{q}-1}
\sum_{j=0}^{\infty} (2^{j}R)^{-sq}\\
\notag&\qquad\times
\int_{\mathcal{X}}
\left[\frac{1}{\omega(B(x,2^{j+1}R))}
\int_{B(x,2^{j+1}R)}|f(y)|^p\omega(y)\,d\mu(y)\right] \omega(x)\,d\mu(x)\\
\notag&\quad=
[\omega]_{A_{1}({\mathcal{X}})}
\left[\frac{ 2^{sq}}{R^{sq} (2^{sq}-1)}\right]^{\frac{p}{q}-1}\\
&\qquad\times
\sum_{j=0}^{\infty} (2^{j}R)^{-sq}\int_{\mathcal{X}}
\left[\int_{B(y,2^{j+1}R)}
\frac{\omega(x)}{\omega(B(x,2^{j+1}R))}\,d\mu(x)\right]
|f(y)|^p\omega(y)\,d\mu(y).
\end{align}
From Definition \ref{Deqms}(iii),
we infer that,
for any $j\in\mathbb{Z}_+$,
$y\in{\mathcal{X}}$, and $x\in B(y,2^{j+1}R)$,
$$B(y,2^{j+1}R)\subset B(x,K_0 2^{j+2}R),$$
which, combined with Lemma~\ref{Lemma2.1}(iii)
and \eqref{ud}, further implies that
\begin{align*}
\omega(B(y,2^{j+1}R))
& \leq \omega(B(x,K_0 2^{j+2}R))
\leq [\omega]_{A_1({\mathcal{X}})}
\frac{\mu(B(x,K_0 2^{j+2}R))}{\mu(B(x,2^{j+1}R))}
\omega(B(x,2^{j+1}R))\\
&\lesssim[\omega]_{A_1({\mathcal{X}})}
\omega(B(x,2^{j+1}R)).
\end{align*}
By this and \eqref{12202108},
we find that
\begin{align*}
&s^{\frac{p}{q}}\int_{\mathcal{X}}
\left\{\int_{[B(x,R)]^\complement}
\frac{|f(y)|^q}{U(x,y)[\rho(x,y)]^{sq}}
\, d\mu(y)\right\}^{\frac{p}{q}}\omega(x)\,d\mu(x)\\
&\quad\lesssim
[\omega]^2_{A_1({\mathcal{X}})}
\left[\frac{s 2^{sq}}{R^{sq} (2^{sq}-1)}\right]^{\frac{p}{q}-1}\\
&\qquad\times
s\sum_{j=0}^{\infty} (2^{j}R)^{-sq}\int_{\mathcal{X}}
\left[\int_{B(y,2^{j+1}R)}
\frac{\omega(x)}{\omega(B(y,2^{j+1}R))}\,d\mu(x)\right]
|f(y)|^p\omega(y)\,d\mu(y)\\
&\quad =
[\omega]^2_{A_1({\mathcal{X}})}
\left[\frac{s 2^{sq}}{R^{sq} (2^{sq}-1)}\right]^{\frac{p}{q}}
\|f\|^{p}_{L^p_{\omega}(\mathcal{X})},
\end{align*}
which completes the proof of Lemma \ref{02201609}.
\end{proof}

Now, we use Lemma \ref{02201609} to prove Lemma \ref{keythmofwss}.

\begin{proof}[Proof of Lemma \ref{keythmofwss}]
Let $s\in(0,1)$, $R \in (1,\infty)$, and $f\in L^p_\omega(\mathcal{X})$.
Then we have
\begin{align}\label{12241742}
&\nonumber s^{\frac{p}{q}}\int_{\mathcal{X}}
\left\{\int_{[B(x,R)]^\complement}
\frac{|f(x)-f(y)|^q}{U(x,y)[\rho(x,y)]^{sq}}
\, d\mu(y)\right\}^{\frac{p}{q}}\omega(x)\,d\mu(x)
\\
\notag&\quad\lesssim s^{\frac{p}{q}}\int_{\mathcal{X}}
\left\{\int_{[B(x,R)]^\complement}
\frac{1}{U(x,y)[\rho(x,y)]^{sq}}
\, d\mu(y)\right\}^{\frac{p}{q}}|f(x)|^p\omega(x)\,d\mu(x)\\
\notag&\qquad+ s^{\frac{p}{q}}\int_{\mathcal{X}}
\left\{\int_{[B(x,R)]^\complement}
\frac{|f(y)|^q}{U(x,y)[\rho(x,y)]^{sq}}
\, d\mu(y)\right\}^{\frac{p}{q}}\omega(x)\,d\mu(x)\\
&\quad=: I_1(s) + I_2(s).
\end{align}
To estimate $I_1(s)$,
from \eqref{dc}, we deduce that,
for any $x\in\mathcal{X}$,
\begin{align}\label{QW1}
&\notag\int_{[B(x,R)]^\complement}
\frac{1}{U(x,y)[\rho(x,y)]^{sq}}
\, d\mu(y)\\
&\quad\nonumber=\sum_{j=0}^{\infty}
\int_{B(x,2^{j+1}R) \setminus B(x,2^{j}R)}
\frac{1}{U(x,y)[\rho(x,y)]^{sq}}
\, d\mu(y)\\
&\quad \lesssim \sum_{j=0}^{\infty}
(2^jR)^{-sq}
\int_{B(x,2^{j+1}R) \setminus B(x,2^{j}R)}
\frac{1}{\mu(B(x,2^jR))}
\, d\mu(y)
\lesssim\frac{2^{sq}}{R^{sq}(2^{sq}-1)}.
\end{align}
By this and Lemma~\ref{Lemma2.1}(i),
we conclude that, for any $s\in(0,1)$,
\begin{align}\label{12241741}
I_1(s)\lesssim\left[\frac{s 2^{sq}}{R^{sq}
(2^{sq}-1)}\right]^{\frac{p}{q}}[\omega]_{A_1({\mathcal{X}})}^2
\|f\|^{p}_{L^p_{\omega}(\mathcal{X})}.
\end{align}
Observe that
\begin{align}\label{1530}
\sup_{s\in(0,1)}\frac{s 2^{sq}}{R^{sq}
(2^{sq}-1)}<\infty.
\end{align}
From this, \eqref{12241742},
\eqref{12241741}, Lemma \ref{02201609},
and $R\in(1,\infty)$,
we infer that \eqref{keyineqwss}
holds,
which then completes the proof of Lemma
\ref{keythmofwss}.
\end{proof}

To show Lemma \ref{12202120},
we also need two technical lemmas.
By a slight modification of the proof of
\cite[Lemmas 4.6 and 4.7]{dlyyz21a} with
$\mathbb{R}^n$ therein replaced by $\mathcal{X}$,
we obtain the following lemmas.

\begin{lemma}\label{lemma3.7}
Let $(\mathcal{X},\rho,\mu)$ be a space of homogeneous type.
Let $Y(\mathcal{X})\subset\mathscr{M}(\mathcal{X})$
be a
linear normed space, equipped with the norm
$\|\cdot\|_{{Y(\mathcal{X})}}$
which makes sense for all functions in $\mathscr{M}(\mathcal{X})$.
Assume that the Hardy--Littlewood maximal operator $\mathcal{M}$
is bounded on $Y(\mathcal{X})$ with its operator norm
denoted by $\|\mathcal{M}\|_{Y(\mathcal{X})\to Y(\mathcal{X})}$.
For any $g\in Y(\mathcal{X})$ and $x\in\mathcal{X}$, let
\begin{align*}
R_{Y(\mathcal{X})}g(x):=\sum_{k=0}^\infty
\frac{\mathcal{M}^kg(x)}{2^k\|\mathcal{M}
\|^k_{Y(\mathcal{X})\to Y(\mathcal{X})}},
\end{align*}
where, for any $k\in\mathbb{N}$,
$\mathcal{M}^k$ is the $k$-fold iteration of $\mathcal{M}$
and $\mathcal{M}^0g:=|g|$.
Then, for any $g\in Y(\mathcal{X})$,
the following assertions hold:
\begin{enumerate}
\item[\textup{(i)}]
for any $x\in\mathcal{X}$, $|g(x)|\leq R_{Y(\mathcal{X})}g(x)$,
\item[\textup{(ii)}]
$R_{{Y(\mathcal{X})}}g\in A_1(\mathcal{X})$ and
$[R_{{Y(\mathcal{X})}}g]_{A_1(\mathcal{X})}
\leq2\|\mathcal{M}\|_{Y(\mathcal{X})\to Y(\mathcal{X})}$,
\item[\textup{(iii)}]
$\|R_{Y(\mathcal{X})}g\|_{{Y(\mathcal{X})}}\leq2\|g\|_{{Y(\mathcal{X})}}$.
\end{enumerate}
\end{lemma}

\begin{lemma}\label{4.6}
Let $(\mathcal{X},\rho,\mu)$ be a space of homogeneous type.
Let $p\in(0,\infty)$ and
$Y(\mathcal{X})$ be a ball
quasi-Banach function space
such that
$Y^{\frac{1}{p}}(\mathcal{X})$
satisfies Assumption~\ref{ma}.
Then, for any $f\in Y(\mathcal{X})$,
\begin{equation*}
\|f\|_{{Y(\mathcal{X})}}\leq
\sup_{\|g\|_{[Y^{\frac{1}{p}}(\mathcal{X})]'}\leq1}
\left\{\int_{\mathcal{X}}
\left|f(x)\right|^pR_{[Y^{\frac{1}{p}}(\mathcal{X})]'}
g(x)\,d\mu(x)\right\}^\frac{1}{p}
\leq2^\frac{1}{p}\|f\|_{{Y(\mathcal{X})}}.
\end{equation*}
\end{lemma}

Next, we are ready to prove Lemma \ref{12202120}.

\begin{proof}[Proof of Lemma \ref{12202120}]
Let $f\in Y({\mathcal{X}})/\mathbb{R}$.
Then, from Proposition~\ref{2357},
we deduce that there exists a constant $a\in\mathbb{R}$ such that
$f+a\in Y(\mathcal{X})$ and
$\|f+a\|_{{Y(\mathcal{X})}}=\|f\|_{Y(\mathcal{X})/\mathbb{R}}$.
Notice that \eqref{06171257} still holds
if we replace $f$ by $f+C$ for any $C\in\mathbb{R}$.
Thus, without loss of generality,
we may assume that $a=0$.
In this case, we have $f\in Y(\mathcal{X})$ and
\begin{align}\label{2220}
\|f\|_{{Y(\mathcal{X})}}=\|f\|_{Y(\mathcal{X})/\mathbb{R}}.
\end{align}
By this and Lemma \ref{4.6}, we find that,
for any $s\in(0,1)$,
\begin{align*}
&\left\|\left\{
\int_{B(\cdot,R)^\complement}
\frac{|f(\cdot)-f(y)|^q}{U(\cdot,y)[\rho(\cdot,y)]^{sq}}
\, d\mu(y)\right\}^{\frac{1}{q}}\right\|^p_{{Y(\mathcal{X})}}\\
&\quad\leq\sup_{\|g\|_{[Y^{\frac{1}{p}}(\mathcal{X})]'}\leq1}
\int_{\mathcal{X}} \left\{
\int_{[B(x,R)]^\complement}
\frac{|f(x)-f(y)|^q}{U(x,y)[\rho(x,y)]^{sq}}
\, d\mu(y)
\right\}^{\frac{p}{q}}R_{[Y^{\frac{1}{p}}(\mathcal{X})]'}g(x)\,d\mu(x).
\end{align*}
From this, Lemma \ref{keythmofwss}
with $\omega$ therein replaced by
$R_{[Y^{\frac{1}{p}}(\mathcal{X})]'}g$,
Lemmas \ref{lemma3.7}(ii) and~\ref{4.6},
and \eqref{2220},
we infer that, for any $s\in(0,1)$,
\begin{align*}
&s^{\frac{p}{q}}
\left\|\left\{\int_{B(\cdot,R)^\complement}
\frac{|f(\cdot)-f(y)|^q}{U(\cdot,y)[\rho(\cdot,y)]^{sq}}
\, d\mu(y) \right\}^{\frac{1}{q}}\right\|^p_{{Y(\mathcal{X})}}\\
&\quad\lesssim \sup_{\|g\|_{[Y^{\frac{1}{p}}(\mathcal{X})]'}\leq 1}
[R_{[Y^{\frac{1}{p}}(\mathcal{X})]'}g]_{A_1({\mathcal{X}})}^2\int_{\mathcal{X}}
|f(x)|^p R_{[Y^{\frac{1}{p}}(\mathcal{X})]'}g(x)\,d\mu(x)\\
&\quad\lesssim \|{\mathcal M}\|_
{[Y^{\frac{1}{p}}(\mathcal{X})]'\to
[Y^{\frac{1}{p}}(\mathcal{X})]'}^2\|f\|^{p}_
{{Y(\mathcal{X})}}
=\|{\mathcal M}\|_
{[Y^{\frac{1}{p}}(\mathcal{X})]'\to
[Y^{\frac{1}{p}}(\mathcal{X})]'}^2\|f\|^{p}_
{{Y(\mathcal{X})/\mathbb{R}}}.
\end{align*}
This finishes the proof of Lemma \ref{12202120}.
\end{proof}

Now, we use Lemma \ref{12202120} to show Theorem \ref{04021646}.

\begin{proof}[Proof of Theorem \ref{04021646}]
By the quasi-triangle inequality of
$\|\cdot\|_{Y(\mathcal{X})}$ and
Lemmas \ref{12202129} and \ref{12202120},
we conclude that,
for any $R\in(1,\infty)$ and
$f\in [Y(\mathcal{X})/\mathbb{R}]
\cap\bigcup_{s\in(0,1)} \dot{W}^{s,q}_Y({\mathcal{X}})$,
\begin{align*}
&\varlimsup_{s \to 0^+}
s^{\frac{1}{q}}\left\|\left\{\int_{\mathcal{X}}
\frac{|f(\cdot)-f(y)|^q}{U(\cdot,y)[\rho(\cdot,
y)]^{sq}}
\, d\mu(y)\right\}^{\frac{1}{q}}\right\|_{{Y(\mathcal{X})}}\\
&\quad\lesssim \varlimsup_{s \to 0^+}
s^{\frac{1}{q}}\left\|
\left\{\int_{B(\cdot,R)}
\frac{|f(\cdot)-f(y)|^q}{U(\cdot,y)[\rho(\cdot,y)]^{sq}}
\, d\mu(y) \right\}^{\frac{1}{q}}\right\|_{{Y(\mathcal{X})}}\\
&\qquad+ \varlimsup_{s \to 0^+}
s^{\frac{1}{q}}\left\|
\left\{\int_{[B(\cdot,R)]^\complement}
\frac{|f(\cdot)-f(y)|^q}{U(\cdot,y)[\rho(\cdot,y)]^{sq}}
\, d\mu(y) \right\}^{\frac{1}{q}}\right\|_{{Y(\mathcal{X})}}\\
&\quad \leq \sup_{s\in(0,1)} s^{\frac{1}{q}}\left\|
\left\{\int_{[B(\cdot,R)]^\complement}
\frac{|f(\cdot)-f(y)|^q}{U(\cdot,y)[\rho(\cdot,y)]^{sq}}
\, d\mu(y)
\right\}^{\frac{1}{q}}\right\|_{{Y(\mathcal{X})}}\\
&\quad\lesssim
\|{\mathcal
M}\|_{[Y^{\frac{1}{p}}(\mathcal{X})]'\to
[Y^{\frac{1}{p}}(\mathcal{X})]'}^{\frac{2}{
p}}\|f\|_{Y(\mathcal{X})/\mathbb{R}}\notag.
\end{align*}
This finishes the proof of Theorem \ref{04021646}.
\end{proof}

\subsection{Proof of Theorem \ref{Thm2}: Lower Estimate}
\label{sec-2-2}

The target of this subsection is to establish
the lower estimate in \eqref{eq1.5}.
The main result of this subsection reads as follows.

\begin{theorem}\label{04021714}
Let $(\mathcal{X},\rho,\mu)$ be a space
of homogeneous type satisfying the {\rm WRD} condition.
Let $0 < q \leq p < \infty$
and $Y(\mathcal{X})$ be a ball
quasi-Banach function space
such that
$Y^{\frac{1}{p}}(\mathcal{X})$
satisfies Assumption~\ref{ma}.
Then there exists a positive constant $C$
such that the following statements hold.
\begin{enumerate}
\item[\rm (i)]
Let $\beta\in(0,\infty)$. Then,
for any $f\in C_{\mathrm{b}}^\beta(\mathcal{X})$,
\begin{equation*}
C\|f\|_{{Y(\mathcal{X})}}
\leq \varliminf_{s \to 0^+}
s^{\frac{1}{q}}\left\|\left\{\int_{\mathcal{X}}
\frac{|f(\cdot)-f(y)|^q}{U(\cdot,y)[\rho(\cdot,y)]^{sq}}
\, d\mu(y) \right\}^{\frac{1}{q}}\right\|_{{Y(\mathcal{X})}}.
\end{equation*}
\item[\rm (ii)]
If $Y(\mathcal{X})$
has an absolutely continuous quasi-norm
and $\mu$ is Borel-semiregular,
then, for any $f\in Y({\mathcal{X}})/\mathbb{R}$,
\begin{equation}\label{04231458}
C\|f\|_{Y(\mathcal{X})/\mathbb{R}}
\leq \varliminf_{s \to 0^+}
s^{\frac{1}{q}}\left\|\left\{\int_{\mathcal{X}}
\frac{|f(\cdot)-f(y)|^q}{U(\cdot,y)[\rho(\cdot,y)]^{sq}}
\, d\mu(y) \right\}^{\frac{1}{q}}\right\|_{{Y(\mathcal{X})}}.
\end{equation}
\end{enumerate}
\end{theorem}

To prove Theorem \ref{04021714}, we need
the following three technical lemmas.

\begin{lemma}\label{02201535}
Let $(\mathcal{X},\rho,\mu)$ be a space of homogeneous type.
Let $0 < q \leq p < \infty$
and $Y(\mathcal{X})$ be a ball
quasi-Banach function space
such that
$Y^{\frac{1}{p}}(\mathcal{X})$
satisfies Assumption~\ref{ma}.
Fix $x_0\in {\mathcal{X}}$.
Then the following assertions hold.
\begin{enumerate}
\item[\rm (i)]
Let $\beta\in(0,\infty)$. Then,
for any $f\in C_{\mathrm{b}}^\beta(\mathcal{X})$,
\begin{equation}\label{02201541}
\lim_{s \to 0^+}s^{\frac{1}{q}}\left\|\left\{
\int_{[B(\cdot,K_0s^{-\frac{1}{q}})]^\complement
\cap [B(x_0,s^{-\frac{1}{q}})]^\complement}
\frac{|f(y)|^q}{U(\cdot,y)[\rho(\cdot,y)]^{sq}}
\, d\mu(y) \right\}^{\frac{1}{q}}\right\|_{{Y(\mathcal{X})}}=0.
\end{equation}
\item[\rm (ii)]
If $Y(\mathcal{X})$
has an absolutely continuous quasi-norm,
then, for any $f\in Y({\mathcal{X}})/\mathbb{R}$,
there exists a constant $a\in\mathbb{R}$ such that
\begin{equation}\label{02201542}
\lim_{s \to 0^+}
s^{\frac{1}{q}}\left\|\left\{\int_{[B(\cdot,K_0s^{-\frac{1}{q}})]^\complement
\cap [B(x_0,s^{-\frac{1}{q}})]^\complement}
\frac{|f(y)+a|^q}{U(\cdot,y)[\rho(\cdot,y)]^{sq}}
\, d\mu(y) \right\}^{\frac{1}{q}}\right\|_{{Y(\mathcal{X})}}=0.
\end{equation}
\end{enumerate}
\end{lemma}

\begin{proof}
We first show (i).
Since $f\in C_{\mathrm{b}}^\beta(\mathcal{X})$, it follows that
there exists $M\in(0,\infty)$ such that $\mathrm{supp\,}f\subset B(x_0,M)$,
which further implies that,
for any given $s\in(0,M^{-q})$ and for any $x\in X$,
$$
\int_{[B(x,K_0s^{-\frac{1}{q}})]^\complement
\cap [B(x_0,s^{-\frac{1}{q}})]^\complement}
\frac{|f(y)|^q}{U(x,y)[\rho(x,y)]^{sq}}
\, d\mu(y)=0
$$
and hence
$$
\lim_{s \to 0^+}
s^{\frac{1}{q}}\left\|\left\{\int_{[B(\cdot,K_0s^{-\frac{1}{q}})]^\complement
\cap [B(x_0,s^{-\frac{1}{q}})]^\complement}
\frac{|f(y)|^q}{U(\cdot,y)[\rho(\cdot,y)]^{sq}}
\, d\mu(y) \right\}^{\frac{1}{q}}\right\|_{{Y(\mathcal{X})}}=0,
$$
which completes the proof of \eqref{02201541}
and hence (i).

Next, we prove (ii).
Let $s\in(0,1)$ and
$f\in Y({\mathcal{X}})/\mathbb{R}$.
Then there exists a constant $a\in\mathbb{R}$
such that $f+a\in Y(\mathcal{X})$.
By Lemma \ref{4.6}, we find that
\begin{align*}
&\left\|\left\{\int_{[B(\cdot,K_0s^{-\frac{1}{q}})]^\complement
\cap [B(x_0,s^{-\frac{1}{q}})]^\complement}
\frac{|f(y)+a|^q}{U(\cdot,y)[\rho(\cdot,y)]^{sq}}
\, d\mu(y) \right\}^{\frac{1}{q}}\right\|^p_{{Y(\mathcal{X})}}\\
&\quad\leq\sup_{\|g\|_{[Y^{\frac{1}{p}}(\mathcal{X})]'}\leq 1}
\int_{\mathcal{X}}
\left\{\int_{[B(x,K_0s^{-\frac{1}{q}})]^\complement}
\frac{|f(y)+a|^q\mathbf{1}_
{[B(x_0,s^{-\frac{1}{q}})]^{\complement}}(y)}
{U(\cdot,y)[\rho(\cdot,y)]^{sq}}\,d\mu(y)
\right\}^{\frac{p}{q}}
R_{[Y^{\frac{1}{p}}(\mathcal{X})]'}g(x)\,d\mu(x),
\end{align*}
which, combined with
Lemma~\ref{02201609},
the fact that $\sup_{s\in(0,1)}\frac{s^{1+s}
2^{sq}}{K_0^{sq}(2^{sq}-1)}<\infty$,
and Lemmas~\ref{lemma3.7}(ii) and \ref{4.6},
further yields
\begin{align}\label{03281205}
&\notag s^{\frac{p}{q}}\left\|\left\{\int_{[B(\cdot,K_0s^{-\frac{1}{q}})]^\complement
\cap [B(x_0,s^{-\frac{1}{q}})]^\complement}
\frac{|f(y)+a|^q}{U(\cdot,y)[\rho(\cdot,y)]^{sq}}
\, d\mu(y) \right\}^{\frac{1}{q}}\right\|^p_{{Y(\mathcal{X})}}\\
&\quad\lesssim \left[\frac{s^{1+s}
2^{sq}}{K_0^{sq}(2^{sq}-1)}\right]^{\frac{p}{q}}
\sup_{\|g\|_{[Y^{\frac{1}{p}}(\mathcal{X})]'}\leq 1}
[R_{[Y^{\frac{1}{p}}(\mathcal{X})]'}g]_{A_1({\mathcal{X}})}^2\notag\\
&\qquad\times
\int_{\mathcal{X}}
|f(x)+a|^p\mathbf{1}_{[B(x_0,s^{-\frac{1}{q}})]
^{\complement}}(x)
R_{[Y^{\frac{1}{p}}(\mathcal{X})]'}g(x)\,d\mu(x) \notag \\
&\quad\lesssim \|{\mathcal M}\|_{[Y^{\frac{1}{p}}(\mathcal{X})]'\to
[Y^{\frac{1}{p}}(\mathcal{X})]'}^2
\left\|\left(f+a\right)\mathbf{1}_{[B(x_0,s^{-\frac{1}{q}}
)]^{\complement}}\right\|^{p}_{{Y(\mathcal{X})}}.
\end{align}
This, together with $f+a\in Y(\mathcal{X})$,
the fact that
$\mathbf{1}_{[B(x_0,s^{-\frac{1}{q}})]^{\complement}}\to 0$
as $s\to 0^+$,
and the assumption that $Y(\mathcal{X})$
has an absolutely continuous quasi-norm,
further implies that \eqref{02201542} holds,
which completes the proof of (ii) and hence Lemma \ref{02201535}.
\end{proof}

We also need the following key lower estimate
on spaces of homogeneous type by taking advantage
of the WRD condition satisfied by the measure.

\begin{lemma}\label{2225}
Let $q\in (0,\infty)$ and
$(\mathcal{X},\rho,\mu)$ be a space
of homogeneous type satisfying the {\rm WRD} condition.
Then there exists $C\in(0,\infty)$ such that, for any $x\in\mathcal{X}$,
\begin{align*}
\varliminf_{s\to 0^+}s\int_{[B(x,s^{-\frac{1}{q}})]^\complement}
\frac{1}{U(x,y)[\rho(x,y)]^{sq}}
\, d\mu(y)\ge C.
\end{align*}
\end{lemma}

To show Lemma~\ref{2225},
we need the following equivalent characterization of the WRD condition.

\begin{proposition}\label{12346781}
Let $(\mathcal{X},\rho,\mu)$ be a quasi-metric measure space.
Then $(\mathcal{X},\rho,\mu)$ satisfies the {\rm WRD} condition if and only if
there exist $x_0\in\mathcal{X}$, $\lambda,\widetilde{C}_{(\mu)}\in(1,\infty)$,
and $r_{x_0}\in(0,\infty)$ such that,
for any $r\in(r_{x_0},\infty)$,
$\mu(B(x_0,\lambda r))\geq\widetilde{C}_{(\mu)}\mu(B(x_0,r))$.
\end{proposition}
\begin{proof}
We first prove the sufficiency. Assume that
there exist $x_0\in\mathcal{X}$, $\lambda,\widetilde{C}_{(\mu)}\in(1,\infty)$,
and $r_{x_0}\in(0,\infty)$ such that,
for any $\widetilde{r}\in(r_{x_0},\infty)$,
$\mu(B(x_0,\lambda\widetilde{r}))\geq\widetilde{C}_{(\mu)}
\mu(B(x_0,\widetilde{r}))$.
From this, we deduce that,
for any $\widetilde{r}\in(r_{x_0},\infty)$,
$$
\inf_{r>\widetilde{r}}
\frac{\mu(B(x_0,\lambda r))}{\mu(B(x_0,r))}\ge\widetilde{C}_{(\mu)}.
$$
Letting $\widetilde{r}\to\infty$, we then conclude that
$(\mathcal{X},\rho,\mu)$ satisfies the WRD condition.
This finishes the proof of the sufficiency.

Now, we show the necessity.
Assume that $(\mathcal{X},\rho,\mu)$ satisfies the WRD condition.
Then there exist $x_0\in\mathcal{X}$ and $\lambda,C_{(\mu)}\in(1,\infty)$
such that \eqref{reverseD} holds.
Let $\epsilon_{0}\in(0,C_{(\mu)}-1)$.
Then \eqref{reverseD} implies that
$$
\sup_{r_{0}\in(0,\infty)}
\inf_{r>r_{0}}
\frac{\mu(B(x_0,\lambda r))}{\mu(B(x_0,r))}>C_{(\mu)}-\epsilon_{0}.
$$
By this, we conclude that exists $r_{x_0}\in(0,\infty)$
such that
$$
\inf_{r>r_{x_0}}
\frac{\mu(B(x_0,\lambda r))}{\mu(B(x_0,r))}>C_{(\mu)}-\epsilon_{0}.
$$
Using this and letting $\widetilde{C}_{(\mu)}:=C_{(\mu)}-\epsilon_{0}$
then complete the proof of the necessity
and hence Proposition~\ref{12346781}.
\end{proof}

Using Proposition~\ref{12346781},
we obtain the following equivalent formulation
of the WRD condition.

\begin{proposition}\label{1017}
Let $(\mathcal{X},\rho,\mu)$ be a quasi-metric measure
space. Then $(\mathcal{X},\rho,\mu)$
satisfies the {\rm WRD} condition if and only if
there exist
$\lambda,C_{(\mu)}\in(1,\infty)$
such that, for any $x\in\mathcal{X}$,
\begin{align*}
\varliminf_{r\to\infty}\frac{\mu(B(x,\lambda r))}{\mu(B(x,r))}
\ge C_{(\mu)}.
\end{align*}
\end{proposition}

\begin{proof}
We only prove the necessity because the sufficiency
is obvious. Assume that $(\mathcal{X},\rho,\mu)$
satisfies the WRD condition. Then, from Proposition~\ref{12346781},
we infer that
there exist $x_0\in\mathcal{X}$,
$\lambda,C_{(\mu)}\in(1,\infty)$,
and $r_{x_0}\in(0,\infty)$ such that,
for any $r\in(r_{x_0},\infty)$,
\begin{align}\label{1108}
\mu(B(x_0,\lambda r))\ge C_{(\mu)}\mu(B(x_0,r)).
\end{align}
Let $x\in\mathcal{X}$
and $r_x:=2K_0\max\{r_{x_0},\,\rho(x,x_0)\}$.
Then, by Definition~\ref{Deqms}(iii),
we are easy to show that,
for any $r\in(r_x,\infty)$,
\begin{align*}
B(x_0,\lambda r)\subset B(x,K_0[\lambda r+\rho(x,x_0)])
\subset B(x,2K_0\lambda r)
\end{align*}
and
\begin{align*}
B(x_0,r)\supset B\left(x,\frac{r}{K_0}-\rho(x,x_0)\right)
\supset B\left(x,\frac{r}{2K_0}\right),
\end{align*}
which, combined with \eqref{1108},
further implies that
\begin{align*}
\frac{\mu(B(x,2K_0\lambda r))}{\mu(B(x,\frac{r}{2K_0}))}
\ge\frac{\mu(B(x_0,\lambda r))}{\mu(B(x_0,r))}\ge C_{(\mu)}.
\end{align*}
This, together with Proposition~\ref{12346781},
then finishes the proof of Proposition~\ref{1017}.
\end{proof}

\begin{proof}[Proof of Lemma~\ref{2225}]
From Proposition~\ref{12346781} and the
assumption that $(\mathcal{X},\rho,\mu)$
satisfies the {\rm WRD} condition, we deduce that
there exist $x_0\in\mathcal{X}$,
$\lambda,\widetilde{C}_{(\mu)}\in(1,\infty)$,
and $r_{x_0}\in(0,\infty)$ such that,
for any $r\in(r_{x_0},\infty)$,
\begin{align}\label{2122}
\mu(B(x_0,\lambda r))\geq\widetilde{C}_{(\mu)}\mu(B(x_0,r)).
\end{align}
Let $x\in\mathcal{X}$. By Definition~\ref{Deqms}(iii),
we find that, for any given $s\in(0,\infty)$
and for any $y\in B(x,s^{-\frac{1}{q}})$,
\begin{align*}
\rho(x_0,y)\leq K_0\left[\rho(x_0,x)+\rho(x,y)\right]
\leq K_0\left[\rho(x_0,x)+s^{-\frac{1}{q}}\right]
\end{align*}
and hence
\begin{align}\label{2123}
B\left(x,s^{-\frac{1}{q}}\right)\subset B\left(x_0,K_0
\left[\rho(x_0,x)+s^{-\frac{1}{q}}\right]\right).
\end{align}
Let $s_0\in(0,\infty)$ be such that
$$
2K_0s_0^{-\frac{1}{q}}
\ge K_0\left[\rho(x_0,x)+s_0^{-\frac{1}{q}}\right]>r_{x_0}.
$$
From this, \eqref{2123}, and \eqref{2122}, we infer that
\begin{align*}
&\varliminf_{s\to 0^+}s\int_{[B(x,s^{-\frac{1}{q}})]^\complement}
\frac{1}{U(x,y)[\rho(x,y)]^{sq}}
\, d\mu(y)\\
&\quad\ge\varliminf_{s\to 0^+}s\int_{[B(x_0,K_0
[\rho(x_0,x)+s^{-\frac{1}{q}}])]^\complement}
\frac{1}{U(x,y)[\rho(x,y)]^{sq}}
\, d\mu(y)\\
&\quad\ge\varliminf_{\genfrac{}{}{0pt}{}{s\in(0,s_0)}{s\to 0}}
s\int_{[B(x_0,2K_0s^{-\frac{1}{q}})]^\complement}
\frac{1}{U(x,y)[\rho(x,y)]^{sq}}
\, d\mu(y)\\
&\quad=\varliminf_{\genfrac{}{}{0pt}{}{s\in(0,s_0)}{s\to 0}}
s\sum_{j=0}^{\infty}
\int_{B(x_0,\lambda^{j+1}2K_0s^{-\frac{1}{q}})
\setminus B(x_0,\lambda^{j}2K_0s^{-\frac{1}{q}})}
\frac{1}{U(x,y)[\rho(x,y)]^{sq}}
\, d\mu(y)\\
&\quad\gtrsim\varliminf_{\genfrac{}{}{0pt}{}{s\in(0,s_0)}{s\to 0}}
s\sum_{j=0}^{\infty}
\left(\lambda^{j+1}2K_0s^{-\frac{1}{q}}\right)^{-sq}
\frac{\mu(B(x_0,\lambda^{j+1}2K_0s^{-\frac{1}{q}}))
-\mu(B(x_0,\lambda^{j}2K_0s^{-\frac{1}{q}}))}
{\mu(B(x_0,\lambda^{j+1}2K_0s^{-\frac{1}{q}}))}\\
&\quad\ge\frac{\widetilde{C}_{(\mu)}-1}{\widetilde{C}_{(\mu)}}
\varliminf_{\genfrac{}{}{0pt}{}{s\in(0,s_0)}{s\to 0}}
(2K_0\lambda)^{-sq}\frac{s^{s+1}}{1-\lambda^{-sq}}
=\frac{\widetilde{C}_{(\mu)}-1}{\widetilde{C}_{(\mu)}}\frac{1}{q\ln\lambda}.
\end{align*}
This finishes the proof of Lemma \ref{2225}.
\end{proof}

The following Fatou's lemma on ball quasi-Banach function spaces
is exactly \cite[Lemma 2.9]{yhyy21};
see also \cite[Lemma~2.4]{wyy22}
for the case $\mathcal{X}=\mathbb{R}^n$.

\begin{lemma}\label{03281146}
Let $(\mathcal{X},\rho,\mu)$ be a space of homogeneous type and
$Y(\mathcal{X})$ a ball quasi-Banach function space.
Then, for any sequences
$\{f_m\}_{m\in \mathbb{N}}$ in $Y(\mathcal{X})$,
\begin{equation*}
\left\|\varliminf_{m\to\infty}
|f_m|\right\|_{{Y(\mathcal{X})}} \leq \varliminf_{m\to\infty}
\left\|f_m\right\|_{Y(\mathcal{X})}.
\end{equation*}
\end{lemma}

Next, we are ready to prove Theorem \ref{04021714}.

\begin{proof}[Proof of Theorem \ref{04021714}]
We first show (i).
Fix $x_0\in{\mathcal{X}}$ and let $s\in(0,1)$ and $x\in\mathcal{X}$.
By Definition~\ref{Deqms}(iii),
we easily find that
\begin{align*}
B\left(x,K_0\rho(x,x_0)+K_0s^{-\frac{1}{q}}\right)
\supset
\left[B\left(x,K_0 s^{-\frac{1}{q}}\right)
\cup
B\left(x_0,s^{-\frac{1}{q}}\right)\right].
\end{align*}
From this, Definition \ref{Debqfs}(ii),
and the quasi-triangle inequality of
$\|\cdot\|_{Y(\mathcal{X})}$,
we deduce that
\begin{align*}
&\left\|\left\{s\int_{
B(\cdot,K_0\rho(\cdot,x_0)+K_0s^{-\frac{1}{q}})^{\complement}}
\frac{|f(\cdot)|^q}{U(\cdot,y)[\rho(\cdot,y)]^{sq}}
\, d\mu(y) \right\}^{\frac{1}{q}}\right\|_{{Y(\mathcal{X})}}\\
\notag&\quad\leq\left\|\left\{s\int_{
[B(\cdot,K_0 s^{-\frac{1}{q}})]^{\complement}
\cap
[B(x_0,s^{-\frac{1}{q}})]^{\complement}
}
\frac{|f(\cdot)|^q}{U(\cdot,y)[\rho(\cdot,y)]^{sq}}
\, d\mu(y) \right\}^{\frac{1}{q}}\right\|_{{Y(\mathcal{X})}}\\
\notag&\quad\lesssim \left\|\left\{s\int_{
[B(\cdot,K_0 s^{-\frac{1}{q}})]^{\complement}
\cap
[B(x_0,s^{-\frac{1}{q}})]^{\complement}}
\frac{|f(\cdot)-f(y)|^q}{U(\cdot,y)[\rho(\cdot,y)]^{sq}}
\, d\mu(y) \right\}^{\frac{1}{q}}\right\|_{{Y(\mathcal{X})}}\\
\notag&\qquad + \left\|\left\{s\int_{
[B(\cdot,K_0 s^{-\frac{1}{q}})]^{\complement}
\cap
[B(x_0,s^{-\frac{1}{q}})]^{\complement}}
\frac{|f(y)|^q}{U(\cdot,y)[\rho(\cdot,y)]^{sq}}
\, d\mu(y) \right\}^{\frac{1}{q}}\right\|_{{Y(\mathcal{X})}}\\
\notag&\quad\leq \left\|\left\{s\int_{\mathcal{X}}
\frac{|f(\cdot)-f(y)|^q}{U(\cdot,y)[\rho(\cdot,y)]^{sq}}
\, d\mu(y) \right\}^{\frac{1}{q}}\right\|_{{Y(\mathcal{X})}}\\
&\qquad +\left\|\left\{s\int_{
[B(\cdot,K_0 s^{-\frac{1}{q}})]^{\complement}
\cap
[B(x_0,s^{-\frac{1}{q}})]^{\complement}}
\frac{|f(y)|^q}{U(\cdot,y)[\rho(\cdot,y)]^{sq}}
\, d\mu(y) \right\}^{\frac{1}{q}}\right\|_{{Y(\mathcal{X})}}.
\end{align*}
Using this, Lemma \ref{2225}, Definition \ref{Debqfs}(ii),
Lemma \ref{03281146},
and Lemma \ref{02201535}(i),
we conclude that
\begin{align*}
\|f\|_{{Y(\mathcal{X})}}
&\lesssim\left\|\varliminf_{s\to 0^+}\left\{s\int_{
[B(\cdot,2K_0s^{-\frac{1}{q}})]^{\complement}}
\frac{|f(\cdot)|^q}{U(\cdot,y)[\rho(\cdot,y)]^{sq}}
\, d\mu(y) \right\}^{\frac{1}{q}}\right\|_{{Y(\mathcal{X})}}\notag\\
&\leq\left\|\varliminf_{s\to 0^+}\left\{s\int_{
[B(\cdot,K_0\rho(\cdot,x_0)+K_0s^{-\frac{1}{q}})]^{\complement}}
\frac{|f(\cdot)|^q}{U(\cdot,y)[\rho(\cdot,y)]^{sq}}
\, d\mu(y) \right\}^{\frac{1}{q}}\right\|_{{Y(\mathcal{X})}}\notag\\
\notag& \leq \varliminf_{s\to 0^+}\left\|\left\{s\int_{
[B(\cdot,K_0\rho(\cdot,x_0)+K_0s^{-\frac{1}{q}})]^{\complement}}
\frac{|f(\cdot)|^q}{U(\cdot,y)[\rho(\cdot,y)]^{sq}}
\, d\mu(y) \right\}^{\frac{1}{q}}\right\|_{{Y(\mathcal{X})}} \\
\notag& \lesssim \varliminf_{s\to 0^+} \left\|\left\{s\int_{\mathcal{X}}
\frac{|f(\cdot)-f(y)|^q}{U(\cdot,y)[\rho(\cdot,y)]^{sq}}
\, d\mu(y) \right\}^{\frac{1}{q}}\right\|_{{Y(\mathcal{X})}}\\
&\quad +\lim_{s\to 0^+} \left\|\left\{s\int_{
[B(\cdot,K_0 s^{-\frac{1}{q}})]^{\complement}
\cap
[B(x_0,s^{-\frac{1}{q}})]^{\complement}}
\frac{|f(y)|^q}{U(\cdot,y)[\rho(\cdot,y)]^{sq}}
\, d\mu(y) \right\}^{\frac{1}{q}}\right\|_{{Y(\mathcal{X})}}\\
\notag&=\varliminf_{s\to 0^+} \left\|\left\{s\int_{\mathcal{X}}
\frac{|f(\cdot)-f(y)|^q}{U(\cdot,y)[\rho(\cdot,y)]^{sq}}
\, d\mu(y) \right\}^{\frac{1}{q}}\right\|_{{Y(\mathcal{X})}},
\end{align*}
which completes the proof of (i).

Now, we prove (ii).
Let $f\in Y({\mathcal{X}})/\mathbb{R}$.
Then there exists a constant $a\in\mathbb{R}$ such that
$f+a\in Y(\mathcal{X})$.
From this and an argument similar to that used in the proof of (i),
we infer that,
to show \eqref{04231458}, it is sufficient to prove
\begin{equation}\label{03290105}
\lim_{s \to 0^+}
\left\|\left\{s\int_{
[B(\cdot,K_0 s^{-\frac{1}{q}})]^{\complement}
\cap
[B(x_0,s^{-\frac{1}{q}})]^{\complement}
}
\frac{|f(y)+a|^q}{U(\cdot,y)[\rho(\cdot,y)]^{sq}}
\, d\mu(y) \right\}^{\frac{1}{q}}\right\|_{{Y(\mathcal{X})}}=0.
\end{equation}
By the assumption that $Y(\mathcal{X})$ has an absolutely
continuous quasi-norm and Lemma \ref{02201535}(ii),
we find that \eqref{03290105} holds
and hence complete the proof of (ii).
This finishes the proof of Theorem \ref{04021714}.
\end{proof}
Next, we use Theorems \ref{04021646} and \ref{04021714}
to show Theorem \ref{Thm2}.
\begin{proof}[Proof of Theorem \ref{Thm2}]
On the one hand,
from Theorems \ref{04021646} and \ref{04021714}(i),
it follows that Theorem \ref{Thm2}(i) holds.
On the other hand,
assume that $Y(\mathcal{X})$ has an
absolutely continuous quasi-norm. Then,
using Theorems \ref{04021646} and \ref{04021714}(ii),
we obtain Theorem \ref{Thm2}(ii) holds,
which completes the proof Theorem \ref{Thm2}.
\end{proof}

\subsection{Proof of Theorem \ref{Thm3}}\label{sec-2-3}

The aim of this subsection is to prove Theorem \ref{Thm3}.
We need two technical lemmas.
The following lemma is precisely \cite[Lemma 2.29]{tnyyz24}.

\begin{lemma}\label{03281616}
Let $(\mathcal{X},\rho,\mu)$ be a space of homogeneous type.
Let $p\in(0,\infty)$ ,
$\beta\in(0,\infty)$, and
$Y(\mathcal{X})$ be a ball
quasi-Banach function space
such that
$Y^{\frac{1}{p}}(\mathcal{X})$
satisfies Assumption~\ref{ma2}.
Assume that $g$ is a non-negative operator
on $\mathscr M({\mathcal{X}})$
and there exist positive constants $C$ and $r$ such that,
for any $m\in{\mathbb N}$ and $f\in C_{\mathrm{b}}^\beta(\mathcal{X})$,
\begin{align*}
\left\|\left[g(f)\right]^{\frac{1}{\theta_m}}
\right\|_{{Y(\mathcal{X})}}\leq C
\|{\mathcal M}\|_{[Y^{\frac{1}{p\theta_m}}(\mathcal{X})]'
\to[Y^{\frac{1}{p\theta_m}}(\mathcal{X})]'}^
{\frac{r}{\theta_m}}
\left\||f|^{\frac{1}{\theta_m}}\right\|_{Y(\mathcal{X})},
\end{align*}
where $\{\theta_{m}\}_{m\in\mathbb{N}}$
is the same as in Assumption~\ref{ma2}.
Then there exists a positive constant $\widetilde{C}$ such that,
for any $f\in C_{\mathrm{b}}^\beta(\mathcal{X})$,
$\|g(f)\|_{{Y(\mathcal{X})}}\leq
\widetilde{C}
\|f\|_{Y(\mathcal{X})}$.
\end{lemma}

Recall that the \emph{lower smoothness index}
$\mathrm{ind\,}(\mathcal{X},\rho)$,
introduced by Mitrea et al. \cite[Definition 4.26]{MiMiMiMo13},
of a quasi-metric space $(\mathcal{X},\rho)$ is defined by
setting
\begin{align}\label{2205}
\mathrm{ind\,}(\mathcal{X},\rho):=
\sup_{\varrho\sim\rho}\left(\log_2C_\varrho\right)^{-1},
\end{align}
where the supremum is taken over all quasi-metrics $\varrho$
on $\mathcal{X}\times\mathcal{X}$
which are equivalent to $\rho$ and,
for any given quasi-metric $\varrho$ defined on
$\mathcal{X}\times\mathcal{X}$,
$$
C_\varrho:=\sup_{
\genfrac{}{}{0pt}{}{x,y,z\in\mathcal{X}}{\text{not all equal}}}
\frac{\varrho(x,y)}{\max\{\varrho(x,z),\,\varrho(z,y)\}}.
$$
Throughout this article, $0<\beta\preceq\mathrm{ind\,}(\mathcal{X},\rho)$
means that $\beta\in(0,\infty)$ and
$\beta\leq\mathrm{ind\,}(\mathcal{X},\rho)$,
where the equality $\beta=\mathrm{ind\,}(\mathcal{X},\rho)$
is only permissible when the supremum in \eqref{2205} is attained.

The following lemma is exactly \cite[Lemma~2.27]{tnyyz24}.

\begin{lemma}\label{lemma2.30}
Let $(\mathcal{X},\rho,\mu)$ be a space of
homogeneous type with $\mu$ being a
Borel-semiregular measure. Assume that
$Y(\mathcal{X})$ is a ball quasi-Banach function
space having an absolutely continuous quasi-norm. Then,
for any $0<\beta\preceq\mathrm{ind\,}(\mathcal{X},\rho)$,
$C_{\mathrm{b}}^\beta(\mathcal{X})$
is dense in $Y(\mathcal{X})$.
\end{lemma}

Now, we are ready to show Theorem \ref{Thm3}.

\begin{proof}[Proof of Theorem \ref{Thm3}]
We first prove (i).
Fix $R \in (1,\infty)$.
From the proof of Theorem \ref{Thm2},
we deduce that,
to show \eqref{eq1.5},
it suffices to prove that,
for any $f\in C_{\mathrm{b}}^\beta(\mathcal{X})$,
\begin{equation}\label{03281117}
\sup_{s\in(0,1)} s^{\frac{1}{q}}
\left\|\left\{\int_{[B(\cdot,R)]^\complement}
\frac{|f(\cdot)-f(y)|^q}{U(\cdot,y)[\rho(\cdot,y)]^{sq}}\,
d\mu(y) \right\}^{\frac{1}{q}}\right\|_{{Y(\mathcal{X})}}
\lesssim \|f\|_{{Y(\mathcal{X})}}
\end{equation}
and
\begin{equation}\label{03281211}
\lim_{s \to 0^+}
s^{\frac{1}{q}}\left\|\left\{\int_{
B(\cdot,K_0 s^{-\frac{1}{q}})^{\complement}
\cap
B(x_0,s^{-\frac{1}{q}})^{\complement}}
\frac{|f(y)|^q}{U(\cdot,y)[\rho(\cdot,y)]^{sq}}
\, d\mu(y) \right\}^{\frac{1}{q}}\right\|_{{Y(\mathcal{X})}}=0.
\end{equation}
Let sequence $\{\theta_m\}_{m\in{\mathbb N}}$ in $(0,1)$ be
such that $\lim_{m\to\infty}\theta_m=1$.
By Lemma \ref{12202120}
and Assumption \ref{ma2},
we conclude that, for any $m\in{\mathbb N}$,
\begin{align*}
\sup_{s\in(0,1)} \left\|
\left\{s\int_{[B(\cdot,R)]^\complement}
\frac{|f(\cdot)-f(y)|^q}{U(\cdot,y)[\rho(\cdot,y)]^{sq}}\, d\mu(y)
\right\}^{\frac{1}{q}}\right\|_{Y^{\frac{1}{\theta_m}}(\mathcal{X})}\lesssim
\|{\mathcal M}\|_{[Y^{\frac{1}{p\theta_m}}(\mathcal{X})]'
\to[Y^{\frac{1}{p\theta_m}}(\mathcal{X})]'}^2
\|f\|_{Y^{\frac{1}{\theta_m}}(\mathcal{X})},
\end{align*}
which further implies that, for any $s\in(0,1)$,
\begin{align*}
\left\| \left\{ s\int_{[B(\cdot,R)]^\complement}
\frac{|f(\cdot)-f(y)|^q}{U(\cdot,y)[\rho(\cdot,y)]^{sq}} \, d\mu(y)
\right\}^{\frac{1}{q\theta_m}}\right\|_{{Y(\mathcal{X})}}\lesssim
\|{\mathcal M}\|_{[Y^{\frac{1}{p\theta_m}}(\mathcal{X})]'
\to[Y^{\frac{1}{p\theta_m}}(\mathcal{X})]'}^{\frac{2}{\theta_m}}
\left\|f^{\frac{1}{\theta_m}}\right\|_{Y(\mathcal{X})}.
\end{align*}
From this and Lemma \ref{03281616} with
$g(f)(\cdot)$ replaced by
\begin{equation*}
\left\{s\int_{[B(\cdot,R)]^\complement}
\frac{|f(\cdot)-f(y)|^q}{U(\cdot,y)[\rho(\cdot,y)
]^{sq}}\,d\mu(y)\right\}^{\frac{1}{q}} ,
\end{equation*}
we infer that, for any $s\in(0,1)$,
\begin{equation*}
s^{\frac{1}{q}}\left\|\left\{\int_{[B(\cdot,R)]^\complement}
\frac{|f(\cdot)-f(y)|^q}{U(\cdot,y)[\rho(\cdot,y)]^{sq}}\,d\mu(y)\right\}
^{\frac{1}{q}}\right\|_{{Y(\mathcal{X})}}
\lesssim \|f\|_{Y(\mathcal{X})},
\end{equation*}
where the implicit positive constant is independent of $s$.
Taking the supremum over all $s\in(0,1)$, we then obtain
\eqref{03281117}.

Next, we show \eqref{03281211}.
By \eqref{03281205}
and an argument similar to that used in
the proof of \eqref{03281117},
we find that, for any $s\in(0,1)$,
\begin{align}\label{03281728}
s^{\frac{1}{q}}\left\|\left\{\int_{
[B(\cdot,K_0 s^{-\frac{1}{q}})]^{\complement}
\cap
[B(x_0,s^{-\frac{1}{q}})]^{\complement}}
\frac{|f(y)|^q}{U(\cdot,y)[\rho(\cdot,y)]^{sq}}
\,d\mu(y)\right\}^{\frac{1}{q}}\right\|_{{Y(\mathcal{X})}}
\lesssim
\left\|f\mathbf{1}_{[B(x_0,s^{-\frac{1}{q}})
]^{\complement}}\right\|_{{Y(\mathcal{X})}},
\end{align}
which, combined with $f\in C_{\mathrm{b}}^\beta(\mathcal{X})$
and via letting $s\to0^+$,
further implies that \eqref{03281211}
holds
and hence completes the proof of (i).

Now, we prove (ii).
Let $f\in Y({\mathcal{X}})/\mathbb{R}$.
From Proposition~\ref{2357}, we deduce that there
exists a constant $a\in\mathbb{R}$ such that
$f+a\in Y(\mathcal{X})$ and
$\|f+a\|_{{Y(\mathcal{X})}}=\|f\|_{Y(\mathcal{X})/\mathbb{R}}$.
By the proof of Theorem \ref{Thm2}, we conclude that,
to show \eqref{eq1.6},
it suffices to prove that
\begin{equation}\label{06102109}
\sup_{s\in(0,1)} s^{\frac{1}{q}}
\left\|\left\{\int_{[B(\cdot,R)]^\complement}
\frac{|[f(\cdot)+a]-[f(y)+a]|^q}{U(\cdot,y)[\rho(\cdot,y)]^{sq}}\,
d\mu(y) \right\}^{\frac{1}{q}}\right\|_{{Y(\mathcal{X})}}
\lesssim \|f+a\|_{{Y(\mathcal{X})}}
\end{equation}
and
\begin{equation}\label{06102111}
\lim_{s \to 0^+}
s^{\frac{1}{q}}\left\|\left\{\int_{
[B(\cdot,K_0 s^{-\frac{1}{q}})]^{\complement}
\cap
[B(x_0,s^{-\frac{1}{q}})]^{\complement}}
\frac{|f(y)+a|^q}{U(\cdot,y)[\rho(\cdot,y)]^{sq}}
\, d\mu(y) \right\}^{\frac{1}{q}}\right\|_{{Y(\mathcal{X})}}=0.
\end{equation}

Next, we show \eqref{06102109}.
Let $R\in(1,\infty)$.
Fix $x_{0}\in \mathcal{X}$ and,
for any $N\in{\mathbb N}$, let
$$E_N:=B(x_{0},N)\times B(x_{0},N).$$
From the assumption that $Y(\mathcal{X})$
has an absolutely continuous quasi-norm and
Lemma \ref{lemma2.30},
it follows that $C_{\mathrm{b}}^\beta(\mathcal{X})$
is dense in $Y(\mathcal{X})$,
where the positive constant $\beta$ is the same as in Lemma \ref{lemma2.30}.
Thus, there exists a sequence
$\{f_k\}_{k\in{\mathbb N}}$ in $C_{\mathrm{b}}^\beta(\mathcal{X})$ such that
\begin{align}\label{03281218}
\lim_{k\to\infty}\|f-f_k\|_{Y(\mathcal{X})}=0.
\end{align}
Using the quasi-triangle inequality of
$\|\cdot\|_{Y(\mathcal{X})}$,
we find that,
for any $s\in(0,1)$ and $N,k\in{\mathbb N}$,
\begin{align}\label{03281716}
&\notag s^{\frac{1}{q}}\left\|
\left\{\int_{[B(\cdot,R)]^\complement}
\frac{|f(\cdot)-f(y)|^q}{U(\cdot,y)[\rho(\cdot,y)]^{sq}}
\mathbf{1}_{E_N}(\cdot,y)\, d\mu(y)
\right\}^{\frac{1}{q}}\right\|_{{Y(\mathcal{X})}}\\
\notag &\quad \lesssim \left\|
\left\{s\int_{[B(\cdot,R)]^\complement}
\frac{|f(\cdot)-f_k(\cdot)|^q}{U(\cdot,y)[\rho(\cdot,y)]^{sq}}
\mathbf{1}_{E_N}(\cdot,y)\, d\mu(y)
\right\}^{\frac{1}{q}}\right\|_{{Y(\mathcal{X})}}\\
\notag &\qquad +
\left\|\left\{s\int_{[B(\cdot,R)]^\complement}
\frac{|f_k(\cdot)-f_k(y)|^q}{U(\cdot,y)[\rho(\cdot,y)]^{sq}}
\mathbf{1}_{E_N}(\cdot,y)\, d\mu(y)
\right\}^{\frac{1}{q}}\right\|_{{Y(\mathcal{X})}}\\
\notag &\qquad +
\left\|\left\{s\int_{[B(\cdot,R)]^\complement}
\frac{|f_k(y)-f(y)|^q}{U(\cdot,y)[\rho(\cdot,y)]^{sq}}
\mathbf{1}_{E_N}(\cdot,y)\, d\mu(y)
\right\}^{\frac{1}{q}}\right\|_{{Y(\mathcal{X})}}\\
&\quad =: I_1(s,N,k) + I_2(s,N,k) + I_3(s,N,k).
\end{align}
To estimate $I_1(s,N,k)$,
from \eqref{QW1} and \eqref{1530},
we infer that, for any $s\in(0,1)$
and $N,k\in{\mathbb N}$,
\begin{equation}\label{03281717}
I_1(s,N,k) \lesssim\left\|(f-f_k)
\mathbf{1}_{B(x_{0},N)}\right\|_{{Y(\mathcal{X})}}.
\end{equation}
To estimate $I_2(s,N,k)$, applying \eqref{03281117}
with $f$ therein replaced by $f_k$ and
the quasi-triangle inequality of
$\|\cdot\|_{Y(\mathcal{X})}$,
we obtain,
for any $s\in(0,1)$ and $N,k\in{\mathbb N}$,
\begin{equation}\label{03281718}
I_2(s,N,k)\lesssim \|f_k\|_{{Y(\mathcal{X})}}
\lesssim\|f\|_{{Y(\mathcal{X})}}
+ \|f-f_k\|_{{Y(\mathcal{X})}}.
\end{equation}
To estimate $I_3(s,N,k)$, from an argument
similar to that used in the estimation of
\eqref{03281205},
we deduce that, for any $s\in(0,1)$ and $N,k\in{\mathbb N}$,
\begin{equation*}
I_3(s,N,k) \lesssim
\left\|(f-f_k)\mathbf{1}_{B(x_{0},N)}\right\|_{Y(\mathcal{X})}.
\end{equation*}
Applying this, \eqref{03281218}, \eqref{03281716}, \eqref{03281717},
and \eqref{03281718} and first
letting $k\to\infty$ and then letting $N\to\infty$,
we find that
\eqref{06102109} holds for any $f\in Y({\mathcal{X}})/\mathbb{R}$.

Now, we prove \eqref{06102111}.
By the quasi-triangle inequality of
$\|\cdot\|_{Y(\mathcal{X})}$, we conclude that,
for any $s\in(0,1)$ and $N,k\in{\mathbb N}$,
\begin{align}\label{03281740}
&\notag s^{\frac{1}{q}}\left\|\left\{\int_{
[B(\cdot,K_0 s^{-\frac{1}{q}})]^{\complement}
\cap
[B(x_0,s^{-\frac{1}{q}})]^{\complement}}
\frac{|f(y)+a|^q}{U(\cdot,y)[\rho(\cdot,y)]^{sq}}\mathbf{1}_{E_N}(\cdot,y)
\, d\mu(y) \right\}^{\frac{1}{q}}\right\|_{{Y(\mathcal{X})}}\\
&\quad\lesssim\left\|\left\{s\int_{
[B(\cdot,K_0 s^{-\frac{1}{q}})]^{\complement}
\cap
[B(x_0,s^{-\frac{1}{q}})]^{\complement}}
\frac{|f_k(y)+a|^q}{U(\cdot,y)[\rho(\cdot,y)]^{sq}}\mathbf{1}_{E_N}(\cdot,y)
\, d\mu(y) \right\}^{\frac{1}{q}}\right\|_{{Y(\mathcal{X})}}\notag\\
&\qquad +\left\|\left\{s\int_{
[B(\cdot,K_0 s^{-\frac{1}{q}})]^{\complement}
\cap
[B(x_0,s^{-\frac{1}{q}})]^{\complement}}
\frac{|f(y)-f_k(y)|^q}{U(\cdot,y)[\rho(\cdot,y)]^{sq}}\mathbf{1}_{E_N}(\cdot,y)
\, d\mu(y) \right\}^{\frac{1}{q}}\right\|_{{Y(\mathcal{X})}}\notag\\
&\quad =: J_1(s,N,k) + J_2(s,N,k).
\end{align}
To deal with $J_1(s,N,K)$, from \eqref{03281728}
and the quasi-triangle inequality of
$\|\cdot\|_{Y(\mathcal{X})}$,
it follows that, for any $s\in(0,1)$ and
$N,k\in{\mathbb N}$,
\begin{align}\label{03281741}
J_1(s,N,k) &\notag\lesssim
\left\|(f_k+a)\mathbf{1}_{[B(x_0,
s^{-\frac{1}{p}})]^{\complement}}\right\|_{{Y(\mathcal{X})}}
\\
&\lesssim \left\|(f+a)\mathbf{1}_{[B(x_0,
s^{-\frac{1}{p}})]^{\complement}}
\right\|_{{Y(\mathcal{X})}}
+ \left\|(f-f_k)\mathbf{1}_{[B(x_0,
s^{-\frac{1}{p}})]^{\complement}}\right\|_{{Y(\mathcal{X})}}.
\end{align}
To deal with $J_2(s,N,k)$, by an argument
similar to that used in the estimation of
\eqref{03281205},
we obtain, for any $s\in(0,1)$ and
$N,k\in{\mathbb N}$,
\begin{align*}
J_2(s,N,k) \lesssim
\|(f-f_k)\mathbf{1}_{B(x_{0},N)}\|_{Y(\mathcal{X})}.
\end{align*}
From this, \eqref{03281218}, \eqref{03281740}, \eqref{03281741},
and the assumption that
$Y(\mathcal{X})$ has an absolutely continuous
quasi-norm and
first letting $k\to\infty$ and then
letting $N\to\infty$,
we infer that \eqref{06102111} holds for any
$f\in Y({\mathcal{X}})/\mathbb{R}$.
This, together with \eqref{06102109},
finishes the proof of Theorem \ref{Thm3}.
\end{proof}

\subsection{Failure of Maz'ya--Shaposhnikova Representation
on \\ Some	Spaces of Homogeneous Type}\label{example}

In this section, we give an example which shows that the
conclusion of Theorem~\ref{Thm2} fails on
some special spaces of homogeneous type.
This indicates that the weak reverse doubling
assumption on the measure under consideration in Theorem~\ref{Thm2}
is necessary in some sense.
To this end,
we first give such an example of underlying spaces under consideration.

\begin{proposition}\label{830}
Let $\mathcal{X}:=\{2^{2^k}\}_{k\in\mathbb{N}}$,
$\rho(x,y):=|x-y|$ for any $x,y\in\mathcal{X}$,
and $\mu(\{2^{2^k}\})=2^k$ for any $k\in\mathbb{N}$.
Then $(\mathcal{X},\rho,\mu)$ is a space of homogeneous type.
\end{proposition}

\begin{proof}
Let $x\in\mathcal{X}$ be an arbitrary point and
$r\in(0,\infty)$.
Then there exists $k\in\mathbb{N}$ such that $x=2^{2^k}$.
Next, we consider the following two cases for $r$.

\emph{Case (1)} $r\in(0,2^{2^{k+1}}-2^{2^k}]$. In this case,
\begin{align*}
\frac{\mu(B(2^{2^k},2r))}{\mu(B(2^{2^k},r))}
\leq\frac{1}{\mu(\{2^{2^k}\})}
\sum_{j=1}^{k+1}
\mu\left(\left\{2^{2^j}\right\}\right)=\frac{2^{k+2}-2}{2^k}\leq4.
\end{align*}

\emph{Case (2)} $r\in(2^{2^{i+1}}-2^{2^k},2^{2^{i+2}}-2^{2^k}]$ with
$i\in\mathbb{N}\cap[k,\infty)$. In this case,
\begin{align*}
\frac{\mu(B(2^{2k},2r))}{\mu(B(2^{2k},r))}\leq
\frac{1}{\mu(\{2^{2(i+1)}\})}
\sum_{j=1}^{i+2}\mu\left(\left\{2^{2j}\right\}\right)
=\frac{2^{i+3}-2}{2^{i+1}}\leq4.
\end{align*}
Combining the above two cases, we conclude that, for any
$r\in(0,\infty)$,
$$
\frac{\mu(B(x,2r))}{\mu(B(x,r))}\leq4,
$$
which, together with the arbitrariness of $x\in\mathcal{X}$,
further implies that
$(\mathcal{X},\rho,\mu)$ is a space of homogeneous type.
This finishes the proof of Proposition~\ref{830}.
\end{proof}

The following proposition
indicates that $(\mathcal{X},\rho,\mu)$ in Proposition~\ref{830}
does not satisfy the {\rm WRD} condition.

\begin{proposition}\label{1116}
Let $(\mathcal{X},\rho,\mu)$ be the same as in Proposition~\ref{830}.
Then $(\mathcal{X},\rho,\mu)$ does not satisfy the {\rm WRD} condition.
\end{proposition}

\begin{proof}
To prove that $(\mathcal{X},\rho,\mu)$ does not satisfy the {\rm WRD} condition,
it suffices to show that, for any given $x\in\mathcal{X}$
and $\lambda\in(1,\infty)$,
there exists a sequence $\{r_j\}_{j\in\mathbb{N}}$ in $(0,\infty)$
satisfying that $\lim_{j\to\infty}r_j=\infty$ such that,
for any $j\in\mathbb{N}$,
$\mu(B(x,\lambda r_j))=\mu(B(x,r_j))$,
by observing that this implies
$$
\varliminf_{r\to\infty}\frac{\mu(B(x,\lambda r))}{\mu(B(x,r))}=1.
$$
Without loss of generality,
we may assume that $x=2^{2^k}$, where $k\in\mathbb{N}$.
Choose $j_0\in\mathbb{N}$ such that $2^{2^{k+j_0}}>2\lambda$.
For any $j\in\mathbb{N}$, let
$r_j:=2^{2^{k+j_0+j}}-2^{2^k}$.
Then, for any $j\in\mathbb{N}$,
$2^{2^{k+j_0+j}}<x+r_j<x+\lambda r_j<2^{2^{k+j_0+j+1}}$
and hence
\begin{align*}
\mu\left(B\left(x,r_j\right)\right)
\ge\mu\left(\left\{2^{2^i}\right\}_{i=1}^{k+j_0+j}\right)
\ge\mu\left(B\left(x,\lambda r_j\right)\right)
\ge\mu\left(B\left(x,r_j\right)\right).
\end{align*}
Thus, the desired conclusion holds and therefore
$(\mathcal{X},\rho,\mu)$ does not satisfy the {\rm WRD} condition.
This finishes the proof of Proposition~\ref{1116}.
\end{proof}

The following proposition proves
that the Maz'ya--Shaposhnikova representation
\eqref{eq1.5} fails on $(\mathcal{X},\rho,\mu)$ as in Proposition~\ref{830}.

\begin{proposition}\label{835}
Let $(\mathcal{X},\rho,\mu)$ be the same as in Proposition~\ref{830}.
Then there exists a function $f\in L^2(\mathcal{X})\cap
\bigcup_{s\in(0,1)}\dot{W}^{s,1}_{L^2}(\mathcal{X})$ such that
$\|f\|_{L^2(\mathcal{X})}\in(0,\infty)$
and
\begin{align}\label{759}
\lim_{s\to0^+}s^2\int_\mathcal{X}\left\{\int_\mathcal{X}
\frac{|f(x)-f(y)|}{U(x,y)[\rho(x,y)]^s}\,d\mu(y)\right\}^2\,d\mu(x)=0.
\end{align}
Consequently, Theorem~\ref{Thm2}(ii)
with both $q=1$ and $Y=L^2$ fails in this setting.
\end{proposition}

\begin{proof}
Let
\begin{align*}
f(x):=
\begin{cases}
1&\text{if }x=4,\\
0&\text{if }x\in\mathcal{X}\setminus\{4\}.
\end{cases}
\end{align*}
Then $\|f\|_{L^2(\mathcal{X})}=2^\frac{1}{2}$
and, for any given $s\in(0,1)$,
\begin{align}\label{920}
\left\|f\right\|_{\dot{W}^{s,1}_{L^2}(\mathcal{X})}^2
&\notag=\int_\mathcal{X}\left\{\int_\mathcal{X}
\frac{|f(x)-f(y)|}{U(x,y)[\rho(x,y)]^s}\,d\mu(y)\right\}^2\,d\mu(x)\\
&=2\left[\sum_{j=2}^\infty\frac{2^j}{(2^j-2)(2^{2^j}-4)^s}\right]^2
+\sum_{j=2}^\infty2^j\left[\frac{2}{(2^j-2)(2^{2^j}-4)^s}\right]^2\nonumber\\
&\sim\left(\sum_{j=2}^\infty2^{-s2^j}\right)^2
+\sum_{j=2}^\infty2^{-j-s2^{(j+1)}}<\infty
\end{align}
and hence $f\in\bigcup_{s\in(0,1)}\dot{W}^{s,1}_{L^2}(\mathcal{X})$.

We turn to show \eqref{759}. By the estimation of \eqref{920},
we find that
\begin{align*}
&\varlimsup_{s\to0^+}s^2\int_\mathcal{X}\left\{\int_\mathcal{X}
\frac{|f(x)-f(y)|}{U(x,y)[\rho(x,y)]^s}\,d\mu(y)\right\}^2\,d\mu(x)\\
&\quad\sim\varlimsup_{s\to0^+}s^2
\left[\left(\sum_{j=2}^\infty2^{-s2^j}\right)^2
+\sum_{j=2}^\infty2^{-j-s2^{(j+1)}}\right]\\
&\quad\leq\varlimsup_{s\to0^+}\left[\left(s\sum_{j=2}^\infty2^{-s2^j}\right)^2
+2^{-1}s^2\right]
=\left(\varlimsup_{s\to0^+}s\sum_{j=2}^\infty2^{-s2^j}\right)^2=0
\end{align*}
and hence \eqref{759} holds. This finishes the proof of Proposition~\ref{835}.
\end{proof}

\begin{remark}\label{sharp1}
From Propositions~\ref{830},~\ref{1116}, and~\ref{835},
we deduce that the space
of homogeneous type $(\mathcal{X},\rho,\mu)$ in Propositions~\ref{830}
does not satisfy the WRD condition, on which the
Maz'ya--Shaposhnikova representation \eqref{eq1.5} fails.
In this sense,
the WRD condition in Theorems~\ref{Thm2} and~\ref{Thm3}
is necessary.
\end{remark}

\section{Maz'ya--Shaposhnikova Representation on
Domains Satisfying the WMD Condition}\label{32123}

In this section, we establish the Maz'ya--Shaposhnikova
representation of quasi-norms of ball
quasi-Banach function spaces on the open set $\Omega\subseteqq\mathcal{X}$.
In Subsection~\ref{3.1}, we prove Theorem~\ref{MS2}.
In Subsection~\ref{3.2}, we show Theorem~\ref{MS3}.
In Subsection~\ref{example2}, we prove that
the weak measure density assumption on the domain under consideration
is necessary in some sense.

Let $\Omega\subseteqq\mathcal{X}$ be an open set and
$\mathscr{M}(\Omega)$ denote the set of
all $\mu$-measurable functions on $\Omega$.
For any function $g$ defined on $\mathcal{X}$,
denote by $g|_\Omega$ the \emph{restriction} of $g$ to $\Omega$.
The following concept of the
space $Y(\Omega)$ of a ball quasi-Banach function space
can be found in \cite[Definition~2.6]{zyy23b}.

\begin{definition}\label{1635}
Let $Y(\mathcal{X})$ be a ball quasi-Banach function space.
The space $Y(\Omega)$ is defined to be the restriction of
$Y(\mathcal{X})$ to $\Omega$, that is,
\begin{align*}
Y(\Omega):=\left\{f\in\mathscr{M}(\Omega):
f=g|_\Omega\text{ for some }g\in Y(\mathcal{X})\right\};
\end{align*}
moreover, for any $f\in Y(\Omega)$, let
$\|f\|_{X(\Omega)}:=\inf\{\|g\|_{Y(\mathcal{X})}:
f=g|_\Omega,\ g\in Y(\mathcal{X})\}$.
\end{definition}

\begin{remark}\label{2048}
Let $Y(\mathcal{X})$ be a ball quasi-Banach function space
and $Y(\Omega)$ its restriction to $\Omega$.
\begin{enumerate}
\item[\rm(i)]
By \cite[Proposition~2.7]{zyy23b},
we conclude that, for any $f\in Y(\Omega)$,
$\|f\|_{Y(\Omega)}=
\|\widetilde{f}\|_{Y(\mathcal{X})}$,
where
\begin{align}\label{1448}
\widetilde{f}(x):=
\begin{cases}
f(x)&\text{if }x\in\Omega,\\
0&\text{if }x\in\Omega^\complement.
\end{cases}
\end{align}
Then, for any $f\in\mathscr{M}(\Omega)$, we have
$f\in Y(\Omega)$ if and only if $\widetilde{f}\in Y(\mathcal{X})$.
\item[\rm(ii)]
From \cite[Proposition~2.8]{zyy23b},
we infer that $Y(\Omega)$ satisfies all
the conditions of Definition~\ref{Debqfs}
with $\mathcal{X}$ therein replaced by $\Omega$.
\item[\rm(iii)]
By (ii) and (iv) of
\cite[Proposition~2.8]{zyy23b},
we find that $C_{\mathrm{b}}^\beta(\Omega)
\subset Y(\Omega)$ for any $\beta\in(0,\infty)$.
\end{enumerate}
\end{remark}

For any given open set $\Omega\subseteqq\mathcal{X}$,
we introduce the concept of
homogeneous fractional ball quasi-Banach Sobolev spaces
on $\Omega$ as follows.

\begin{definition}\label{2200}
Let $s\in(0,1)$, $q\in(0,\infty)$, $Y(\mathcal{X})$
be a ball quasi-Banach function space,
and $Y(\Omega)$ the restriction of $Y(\mathbb{R}^n)$ to $\Omega\subseteqq\mathcal{X}$.
The \emph{homogeneous fractional ball quasi-Banach Sobolev space
$\dot{W}^{s,q}_Y(\Omega)$} is defined
to be the set of all measurable functions $f$ on $\Omega$ such that
\begin{equation*}
\|f\|_{\dot{W}^{s,q}_Y(\Omega)}:=
\left\|\left\{\int_\Omega
\frac{|f(\cdot)-f(y)|^q}{U(\cdot,y)[\rho(\cdot,y)]^{sq}}
\,dy\right\}^{\frac{1}{q}}\right\|_{Y(\Omega)}<\infty.
\end{equation*}
\end{definition}

\begin{remark}
If $\Omega\subseteqq\mathbb{R}^n$ is
equipped with the standard Euclidean distance and
the Lebesgue measure
and $Y=L^q$ with $q\in[1,\infty)$,
then ${\dot{W}^{s,q}_{{Y}}(\Omega)}$
is precisely $\dot{W}^{s,q}(\Omega)$ with equivalent quasi-norms,
the classical homogeneous fractional Sobolev space on $\Omega$.
\end{remark}

We now introduce the quotient space of ball quasi-Banach function spaces
on $\Omega$.

\begin{definition}
Let $Y(\mathcal{X})$ be a ball quasi-Banach function space
and $\Omega\subseteqq\mathcal{X}$.
The \emph{quotient space $Y(\Omega)/\mathbb{R}$} is defined
to be the set of all equivalent classes $[f]$ of
measurable functions on $\Omega$ such that
\begin{align*}
\left\|[f]\right\|_{Y(\Omega)/\mathbb{R}}:=
\inf_{a\in\mathbb{R}}\|f+a\|_{Y(\Omega)}<\infty.
\end{align*}
\end{definition}

\subsection{Proof of Theorem~\ref{MS2}}\label{3.1}

To show Theorem~\ref{MS2}, we need two lemmas.
The first technical lemma gives
a lower estimate related to the domain $\Omega$
of the space of homogeneous type $\mathcal{X}$,
with the WMD condition on $\Omega$
and the WRD condition on $\mathcal{X}$ being required.

\begin{lemma}\label{2158}
Let $q\in(0,\infty)$,
$(\mathcal{X},\rho,\mu)$ be a space of homogeneous type
satisfying the $\rm{WRD}$ condition,
and $\Omega\subseteqq\mathcal{X}$ satisfy the $\rm{WMD}$ condition.
Then there exists
a positive constant $C$ such that, for any $x\in\Omega$,
\begin{align*}
\varliminf_{s\to0^+}s\int_{\Omega\cap
[B(x,s^{-\frac{1}{q}})]^\complement}\frac{1}
{U(x,y)[\rho(x,y)]^{sq}}\,d\mu(y)\ge C.
\end{align*}
\end{lemma}

\begin{proof}
From the assumption that $\Omega$
satisfies the WMD condition,
we deduce that
there exist $x_0\in\mathcal{X}$ and
two constants $r_{x_0},C\in(0,\infty)$
such that, for any $r\in(r_{x_0},\infty)$,
\begin{align}\label{1541}
\frac{\mu(B(x_0,r)\cap\Omega)}{\mu(B(x_0,r))}\ge C.
\end{align}
Let $x\in\Omega$ and $r_x:=(2K_0)^2[\rho(x_0,x)+r_{x_0}]$.
Then Definition~\ref{Deqms}(iii) implies that,
for any $r\in(r_x,\infty)$,
$$
B\left(x,\frac{r}{(2K_0)^2}\right)\subset
B\left(x_0,\frac{r}{2K_0}\right)\subset B(x,r),
$$
which, together with \eqref{1541} and \eqref{ud}, implies that
\begin{align}\label{1549}
\mu(B(x,r)\cap\Omega)&\ge
\mu\left(B\left(x_0,\frac{r}{2K_0}\right)\cap\Omega\right)\ge
C\mu\left(B\left(x_0,\frac{r}{2K_0}\right)\right)\nonumber\\
&\ge C\mu\left(B\left(x,\frac{r}{(2K_0)^2}\right)\right)
\ge\widetilde{C}\mu(B(x,r)),
\end{align}
where $\widetilde{C}:=\frac{C}{L_{(\mu)}(2K_0)^{2d}}$.

On the other hand, by the assumption that $(\mathcal{X},\rho,\mu)$
satisfies the WRD condition and
Propositions~\ref{12346781} and~\ref{1017},
we conclude that there exist
constants $\widetilde{r}_x\in(0,\infty)$
and $\lambda,C_{(\mu)}\in(1,\infty)$
such that, for any $r\in(\widetilde{r}_x,\infty)$,
\begin{align*}
\mu(B(x,\lambda r))\ge C_{(\mu)}\,\mu(B(x,r)),
\end{align*}
which, further implies that
there exist $\Lambda\in(1,\infty)$ and
$\widetilde{C}_{(\mu)}\in(\frac{1}{\widetilde{C}},\infty)$
such that
\begin{align}\label{1613}
\mu(B(x,\Lambda r))\ge\widetilde{C}_{(\mu)}\,\mu(B(x,r)),
\end{align}
Let $R_x:=\max\{r_x,\,\widetilde{r}_x\}$.
From this, \eqref{1549}, \eqref{1613},
and \eqref{ud}, we infer that
\begin{align*}
&\varliminf_{s\to0^+}s\int_{\Omega\cap
[B(x,s^{-\frac{1}{q}})]^\complement}\frac{1}
{U(x,y)[\rho(x,y)]^{sq}}\,d\mu(y)\\
&\quad=\varliminf_{\genfrac{}{}{0pt}{}{s\in(0,(R_x)^{-q})}{s\to 0}}
s\sum_{j=1}^{\infty}
\int_{\Omega\cap B(x,\Lambda^js^{-\frac{1}{q}})
\setminus B(x,\Lambda^{j-1}s^{-\frac{1}{q}})}
\frac{1}{U(x,y)[\rho(x,y)]^{sq}}\,d\mu(y)\\
&\quad\geq\varliminf_{\genfrac{}{}{0pt}{}{s\in(0,(R_x)^{-q})}{s\to 0}}
s\sum_{j=1}^{\infty}
(\Lambda^js^{-\frac{1}{q}})^{-sq}
\frac{\mu(\Omega\cap
B(x,\Lambda^js^{-\frac{1}{q}}))-\mu(\Omega\cap
B(x,\Lambda^{j-1}s^{-\frac{1}{q}}))}{\mu(B(x,\Lambda^js^{-\frac{1}{q}}))}\\
&\quad\geq\varliminf_{\genfrac{}{}{0pt}{}{s\in(0,(R_x)^{-q})}{s\to 0}}
s^{1+s}\sum_{j=1}^{\infty}\Lambda^{-sqj}
\frac{\widetilde{C}\widetilde{C}_{(\mu)}\mu(
B(x,\Lambda^{j-1}s^{-\frac{1}{q}}))-\mu(
B(x,\Lambda^{j-1}s^{-\frac{1}{q}}))}{\mu(B(x,\Lambda^js^{-\frac{1}{q}}))}\\
&\quad\sim\lim_{s\to0^+}s^{s+1}\sum_{j=1}^\infty
\Lambda^{-sqj}\sim1.
\end{align*}
This finishes the proof of Lemma \ref{2158}.
\end{proof}

The following Fatou's lemma on $Y(\Omega)$
can be easily deduced from \cite[Lemma~2.9]{yhyy21}
and Remark~\ref{2048}(i).

\begin{lemma}\label{Fatou}
Let $(\mathcal{X},\rho,\mu)$ be a space of homogeneous type,
$\Omega\subseteqq\mathcal{X}$,
$Y(\mathcal{X})$ be a ball quasi-Banach function space, and
$Y(\Omega)$ the restriction of $Y(\mathcal{X})$ to $\Omega$.
Then, for any sequences
$\{f_k\}_{k\in\mathbb{N}}$ in $Y(\Omega)$,
\begin{align*}
\left\|\varliminf_{k\to\infty}\left|f_k\right|\right\|_{Y(\Omega)}
\leq\varliminf_{k\to\infty}\left\|f_k\right\|_{Y(\Omega)}.
\end{align*}
\end{lemma}

Next, we are ready to prove Theorem~\ref{MS2}.

\begin{proof}[Proof of Theorem~\ref{MS2}]
To show (i), let
$f\in C_{\mathrm{b}}^\beta(\Omega)\cap
\bigcup_{s\in(0,1)}\dot{W}^{s,q}_Y(\Omega)$
and $\widetilde{f}$ be as in \eqref{1448}.
Then $\widetilde{f}\in C_{\mathrm{b}}^\beta(\mathcal{X})$
and $\|\widetilde{f}\|_{Y(\mathcal{X})}=\|f\|_{Y(\Omega)}$.

We first prove the upper estimate in (i).
By the assumption that $f\in\bigcup_{s\in(0,1)}\dot{W}^{s,q}_Y(\Omega)$
and Definition~\ref{2200}
and following the proof of Lemma~\ref{12202129}, we find that,
for any $R\in(0,\infty)$,
\begin{align}\label{924}
\lim_{s\to0^+}s^\frac{1}{q}
\left\|\left\{\int_{\Omega\cap B(\cdot,R)}\frac{|f(\cdot)-f(y)|^q}
{U(\cdot,y)[\rho(\cdot,y)]^{sq}}\,d\mu(y)
\right\}^\frac{1}{q}\right\|_{Y(\Omega)}=0.
\end{align}
On the other hand, from Lemma~\ref{12202120}
with $f$ therein replaced by $\widetilde{f}$, we deduce that,
for any $R\in(1,\infty)$,
\begin{align*}
&\varlimsup_{s\to0^+}s^\frac{1}{q}
\left\|\left\{\int_{\Omega\setminus B(\cdot,R)}\frac{|f(\cdot)-f(y)|^q}
{U(\cdot,y)
[\rho(\cdot,y)]^{sq}}\,d\mu(y)\right\}^\frac{1}{q}\right\|_{Y(\Omega)}\\
&\quad\leq\sup_{s\in(0,1)}s^\frac{1}{q}
\left\|\left\{\int_{\mathcal{X}\setminus B(\cdot,R)}
\frac{|\widetilde{f}(\cdot)-\widetilde{f}(y)|^q}
{U(\cdot,y)[\rho(\cdot,y)]^{sq}}\,d\mu(y)
\right\}^\frac{1}{q}\right\|_{Y(\mathcal{X})}
\lesssim\left\|\widetilde{f}\right\|_{Y(\mathcal{X})}
=\|f\|_{Y(\Omega)}.
\end{align*}
This, together with \eqref{924}, implies that
\begin{align*}
\varlimsup_{s\to0^+}s^\frac{1}{q}
\left\|\left\{\int_{\Omega}\frac{|f(\cdot)-f(y)|^q}
{U(\cdot,y)[\rho(\cdot,y)]^{sq}}\,d\mu(y)
\right\}^\frac{1}{q}\right\|_{Y(\Omega)}
\lesssim\|f\|_{Y(\Omega)},
\end{align*}
which completes the proof of the upper estimate in (i).

Now, we prove the lower estimate in (i).
Notice that, for any given $x_0\in\Omega$
and for any $x\in\Omega$ and $s\in(0,1)$,
\begin{align*}
B\left(x,K_0\rho(x,x_0)+K_0s^{-\frac{1}{q}}\right)
\supset
\left[B\left(x,K_0 s^{-\frac{1}{q}}\right)
\cup
B\left(x_0,s^{-\frac{1}{q}}\right)\right],
\end{align*}
which further implies that
\begin{align}\label{III}
&\notag\left\|\left\{s\int_{\Omega\cap[B(\cdot,
K_0\rho(\cdot,x_0)+K_0s^{-\frac{1}{q}})]^\complement}\frac{|f(\cdot)|^q}
{U(\cdot,y)[\rho(\cdot,y)]^{sq}}\,d\mu(y)
\right\}^\frac{1}{q}\right\|_{Y(\Omega)}\\
&\quad\lesssim
\left\|\left\{s\int_{\Omega\cap[B(\cdot,
K_0\rho(\cdot,x_0)+K_0s^{-\frac{1}{q}})]^\complement}
\frac{|f(\cdot)-f(y)|^q}
{U(\cdot,y)[\rho(\cdot,y)]^{sq}}\,d\mu(y)
\right\}^\frac{1}{q}\right\|_{Y(\Omega)}
\nonumber\\
&\qquad+\left\|\left\{s\int_{\Omega\cap[B(\cdot,K_0 s^{-\frac{1}{q}})
]^\complement\cap[B(x_0,s^{-\frac{1}{q}})]^\complement
}\frac{|f(y)|^q}
{U(\cdot,y)[\rho(\cdot,y)]^{sq}}\,d\mu(y)
\right\}^\frac{1}{q}\right\|_{Y(\Omega)}
\nonumber\\
&\quad\leq\left\|\left\{s\int_\Omega\frac{|f(\cdot)-f(y)|^q}
{U(\cdot,y)[\rho(\cdot,y)]^{sq}}\,d\mu(y)
\right\}^\frac{1}{q}\right\|_{Y(\Omega)}
\nonumber\\
&\qquad+\left\|\left\{s\int_{[B(\cdot,K_0s^{-\frac{1}{q}})]^\complement
\cap [B(x_0,s^{-\frac{1}{q}})]^\complement}
\frac{|\widetilde{f}(y)|^q}
{U(\cdot,y)[\rho(\cdot,y)]^{sq}}\,d\mu(y)
\right\}^\frac{1}{q}
\right\|_{Y(\mathcal{X})}.
\end{align}
On the one hand, by Lemmas~\ref{Fatou} and~\ref{2158},
we conclude that
\begin{align}\label{I}
&\notag\varliminf_{s\to0^+}\left\|\left\{s\int_{\Omega\cap[B(\cdot,
K_0\rho(\cdot,x_0)+K_0s^{-\frac{1}{q}})]^\complement}\frac{|f(\cdot)|^q}
{U(\cdot,y)[\rho(\cdot,y)]^{sq}}\,d\mu(y)\right\}^\frac{1}{q}\right\|_{Y(\Omega)}\\
&\quad\ge\left\|\left\{\varliminf_{s\to0^+}s\int_{\Omega\cap[B(\cdot,
K_0\rho(\cdot,x_0)+K_0s^{-\frac{1}{q}})]^\complement}\frac{1}
{U(\cdot,y)[\rho(\cdot,y)]^{sq}}\,d\mu(y)\right\}^\frac{1}{q}|f(\cdot)|
\right\|_{Y(\Omega)}\gtrsim\|f\|_{Y(\Omega)}.
\end{align}
On the other hand, from Lemma~\ref{02201535}(i),
we infer that
\begin{align*}
\lim_{s\to0^+}\left\|\left\{s
\int_{[B(\cdot,K_0s^{-\frac{1}{q}})]^\complement
\cap[B(x_0,s^{-\frac{1}{q}})]^\complement}
\frac{|\widetilde{f}(y)|^q}
{U(\cdot,y)[\rho(\cdot,y)]^{sq}}\,d\mu(y)
\right\}^\frac{1}{q}\right\|_{Y(\mathcal{X})}=0,
\end{align*}
which, combined with both \eqref{III} and \eqref{I},
completes the proof of the lower estimate in (i).
This finishes the proof of (i).

Next, we show (ii). The proof of the upper estimate in (ii)
is actually the same as that in (i).
The proof of the lower estimate in (ii) is quite similar to that of
(i) with Lemma~\ref{02201535}(i) replaced by
Lemma~\ref{02201535}(ii)
and hence we omit the details,
which completes the proof of Theorem~\ref{MS2}.
\end{proof}

\subsection{Proof of Theorem~\ref{MS3}}\label{3.2}

In this subsection, we turn to prove Theorem~\ref{MS3}.

\begin{proof}[Proof of Theorem~\ref{MS3}]
We first show (i).
Let $f\in C_{\mathrm{b}}^\beta(\Omega)\cap
\bigcup_{s\in(0,1)}\dot{W}^{s,q}_Y(\Omega)$
and $\widetilde{f}$ be as in \eqref{1448}.
Fix $x_0\in\Omega$.
By the proof of Theorem~\ref{MS2}(i), we find that
it suffices to prove that, for some $R\in(0,\infty)$,
\begin{align}\label{107}
\sup_{s\in(0,1)}s^\frac{1}{q}
\left\|\left\{\int_{[B(\cdot,R)]^\complement}
\frac{|\widetilde{f}(\cdot)-\widetilde{f}(y)|^q}
{U(\cdot,y)[\rho(\cdot,y)]^{sq}}\,d\mu(y)
\right\}^\frac{1}{q}\right\|_{Y(\mathcal{X})}
\lesssim\left\|\widetilde{f}\right\|_{Y(\mathcal{X})}
\end{align}
and
\begin{align}\label{108}
\lim_{s\to0^+}\left\|\left\{s\int_{[B(\cdot,K_0s^{-\frac{1}{q}})]^\complement\cap
[B(x_0,s^{-\frac{1}{q}})]^\complement
}\frac{|\widetilde{f}(y)|^q}
{U(\cdot,y)[\rho(\cdot,y)]^{sq}}\,d\mu(y)
\right\}^\frac{1}{q}\right\|_{Y(\mathcal{X})}=0.
\end{align}
Notice that $\widetilde{f}\in C_{\mathrm{b}}^\beta(\mathcal{X})$.
From this, \eqref{03281117}, and \eqref{03281211},
it follows that both \eqref{107} and \eqref{108} hold.
This finishes the proof of (i).

Now, we show (ii). Let $f\in[Y(\Omega)/\mathbb{R}]
\cap\bigcup_{s\in(0,1)}\dot{W}^{s,q}_Y(\Omega)$.
Then, by Proposition~\ref{2357},
we conclude that there exists $a\in\mathbb{R}$ such that
$f+a\in Y(\Omega)$ and $\|f+a\|_{Y(\Omega)}=\|f\|_{Y(\Omega)/\mathbb{R}}$.
Let $g:=\widetilde{f+a}$ be defined as in \eqref{1448}
with $f$ therein replaced by $f+a$.
Then we are easy to prove that $g\in Y(\mathcal{X})$
and $\|g\|_{Y(\mathcal{X})}=\|f+a\|_{Y(\Omega)}$.
From the proof of Theorem~\ref{MS2}(ii),
we deduce that it suffices to show that, for some $R\in(0,\infty)$,
\begin{align}\label{121}
\sup_{s\in(0,1)}s^\frac{1}{q}
\left\|\left\{\int_{[B(\cdot,R)]^\complement}\frac{|g(\cdot)-g(y)|^q}
{U(\cdot,y)[\rho(\cdot,y)]^{sq}}\,d\mu(y)
\right\}^\frac{1}{q}\right\|_{Y(\mathcal{X})}
\lesssim\|g\|_{Y(\mathcal{X})}
\end{align}
and
\begin{align}\label{122}
\lim_{s\to0^+}\left\|\left\{s\int_{ [B(\cdot,K_0s^{-\frac{1}{q}})]^\complement\cap
[B(x_0,s^{-\frac{1}{q}})]^\complement
}\frac{|g(y)|^q}
{U(\cdot,y)[\rho(\cdot,y)]^{sq}}\,d\mu(y)
\right\}^\frac{1}{q}\right\|_{Y(\mathcal{X})}=0.
\end{align}
Indeed, by \eqref{06102109} and \eqref{06102111},
we find that both \eqref{121} and \eqref{122} hold.
This finishes the proof of (ii)
and hence Theorem~\ref{MS3}.
\end{proof}

\subsection{Failure of Maz'ya--Shaposhnikova Representation
on Some	Domains $\Omega\subseteqq\mathbb{R}^n$}\label{example2}

In this section, we give an example of a domain that
does not satisfy the $\rm{WMD}$ condition
and prove that the conclusion of Theorem~\ref{MS2}
fails in this setting.
This indicates that the assumption that $\Omega$ satisfies the $\rm{WMD}$ condition
in Theorem~\ref{MS2}
is necessary in some sense.
We consider the case where the underlying space
$(\mathcal{X},\rho,\mu)$
is the Euclidean space $\mathbb{R}$ equipped with the
standard Euclidean distance and the Lebesgue measure.

\begin{proposition}\label{1233}
Let $\Omega:=\bigcup_{j=1}^\infty(4^j,4^j+2^j)$. Then
$\Omega$ does not satisfy the $\rm{WMD}$ condition
and there exists $f\in L^1(\Omega)\cap
\bigcup_{s\in(0,1)}\dot{W}^{s,1}(\Omega)$
such that $\|f\|_{L^1(\Omega)}\in(0,\infty)$
and
\begin{align}\label{1523}
\lim_{s\to0^+}s
\int_{\Omega}\int_\Omega\frac{|f(x)-f(y)|}
{|x-y|^{n+s}}\,dy\,dx=0.
\end{align}
Consequently, Theorem~\ref{MS2} with both $q=1$
and $Y=L^1$ fails in this setting.
\end{proposition}

\begin{proof}
We first show that $\Omega$ does not satisfy the $\rm{WMD}$ condition.
Notice that
\begin{align*}
\frac{|B(0,4^j+2^j)\cap\Omega|}{4^j+2^j}\leq2^{1-j}\to0
\end{align*}
as $j\to\infty$. Thus, $\Omega$ does not satisfy the $\rm{WMD}$ condition.

Next, we prove that there exists $f\in L^1(\Omega)
\cap\bigcup_{s\in(0,1)}\dot{W}^{s,1}(\Omega)$
such that $\|f\|_{L^1(\Omega)}\in(0,\infty)$.
Let
\begin{align*}
f(x):=
\begin{cases}
1&\text{if }x\in(4,6),\\
0&\text{if }x\in\Omega\setminus(4,6).
\end{cases}
\end{align*}
Then $\|f\|_{L^1(\Omega)}=2$.
For any $s\in(0,1)$,
\begin{align}\label{1526}
\|f\|_{\dot{W}^{s,1}(\Omega)}
&\notag=2\int_4^6\sum_{j=2}^\infty\int_{4^j}^{4^j+2^j}(y-x)^{-1-s}\,dy\,dx\\
&\sim\sum_{j=2}^\infty\int_{4^j}^{4^j+2^j}y^{-1-s}\,dy
\sim\sum_{j=2}^\infty2^{-j}4^{-js}\leq\frac{1}{2},
\end{align}
where the positive equivalence constants are
independent of $s$,
and hence $f\in\bigcup_{s\in(0,1)}\dot{W}^{s,1}(\Omega)$.
Moreover, \eqref{1526} further implies that \eqref{1523} holds.
This finishes the proof of Proposition~\ref{1233}.
\end{proof}

\begin{remark}\label{sharp2}
From Proposition~\ref{1233} and Theorem~\ref{MS2},
we infer that $\Omega$ in Proposition~\ref{1233}
does not satisfy the WMD condition, on which the
Maz'ya--Shaposhnikova representation,
namely the conclusion of Theorems~\ref{MS2}
and~\ref{MS3}, fails.
In this sense,
the WMD condition in Theorems~\ref{MS2}
and~\ref{MS3}
is necessary.
\end{remark}

\section{Applications to Specific Ball
Quasi-Banach Function Spaces}\label{sec4}
In this section, we apply the main theorems
of this article, namely
Theorems~\ref{Thm2},~\ref{Thm3},~\ref{MS2}, and \ref{MS3},
to weighted (or variable) Lebesgue spaces,
weighted Lorentz spaces, weighted Orlicz spaces,
generalized Morrey (or Lorenz--Morrey, Orlicz--Morrey)
spaces, and generalized block
(or Lorentz-block, Orlicz-block) spaces.
Throughout this section, we \emph{always assume} that
$(\mathcal{X},\rho,\mu)$
is a space of homogeneous type satisfying the WRD condition
and
$\Omega\subseteqq\mathcal{X}$ is an open
set satisfying the $\rm{WMD}$ condition.

\subsection{Weighted Lebesgue Spaces}

Recall that
weighted Lebesgue spaces
(see Definition \ref{twl}) are
quasi-Banach function spaces (see \cite[p.\,86]{shyy17}),
but it may not satisfy the conditions of
a Banach function space in \cite{bs88}.
By a slight modification of the proof of \cite[Theorem~4.1]{pyyz24}
with \cite[Theorem 2.12]{pyyz24} replaced by Theorem~\ref{MS2},
we obtain the following conclusion and omit the details.

\begin{theorem}\label{01011838}
Let $0<q\leq p \leq r<\infty$ and $\omega\in A_{\frac{r}{p}}(\mathcal{X})$.
Then there exist $C,\widetilde{C}\in(0,\infty)$ such that,
for any $f\in [L^{r}_{\omega}(\Omega)/\mathbb{R}]\cap\bigcup_{s\in(0,1)}
\dot{W}^{s,q}_{L^{r}_{\omega}}(\Omega)$,
\begin{align*}
C\|f\|_{L^{r}_{\omega}(\Omega)/\mathbb{R}}
&\leq \varliminf_{s \to 0^+}
s^{\frac{1}{q}}\left\|\left\{\int_\Omega
\frac{|f(\cdot)-f(y)|^q}{U(\cdot,y)[\rho(\cdot,y)]^{sq}}
\, d\mu(y) \right\}^{\frac{1}{q}}\right\|_{L^{r}_{\omega}(\Omega)}\\
&\leq\varlimsup_{s \to 0^+}
s^{\frac{1}{q}}\left\|\left\{\int_\Omega
\frac{|f(\cdot)-f(y)|^q}{U(\cdot,y)[\rho(\cdot,y)]^{sq}}
\,d\mu(y)\right\}^{\frac{1}{q}}\right\|_{L^{r}_{\omega}(\Omega)}
\leq\widetilde{C}\|f\|_{L^{r}_{\omega}(\Omega)/\mathbb{R}}.
\end{align*}
\end{theorem}

\begin{remark}
When $\mathcal{X}=\mathbb{R}^n=\Omega$,
$\rho$ is the standard Euclidean distance,
and $\mu$ is the $n$-dimensional Lebesgue measure,
Theorem~\ref{01011838} in this case reduces to \cite[Theorem~4.1]{pyyz24}.
The other cases of Theorem~\ref{01011838} are new.
\end{remark}

The following is a corollary of Theorem \ref{01011838}
with $\omega\equiv1$.

\begin{corollary}
Let $0<q \leq r < \infty$.
Then there exist $C,\widetilde{C}\in(0,\infty)$ such that,
for any $f\in[L^{r}(\Omega)/\mathbb{R}]\cap\bigcup_{s\in(0,1)}
\dot{W}^{s,q}_{L^{r}}(\Omega)$,
\begin{align*}
C\|f\|_{L^{r}(\Omega)/\mathbb{R}}
&\leq \varliminf_{s \to 0^+}
s^{\frac{1}{q}}\left[\int_\Omega\left\{\int_\Omega
\frac{|f(x)-f(y)|^q}{U(x,y)[\rho(x,y)]^{sq}}
\,d\mu(y)\right\}^{\frac{r}{q}}d\mu(x)\right]^\frac{1}{r}
\\
&\leq \varlimsup_{s \to 0^+}
s^{\frac{1}{q}}\left[\int_\Omega\left\{\int_\Omega
\frac{|f(x)-f(y)|^q}{U(x,y)[\rho(x,y)]^{sq}}
\,d\mu(y)\right\}^{\frac{r}{q}}d\mu(x)\right]^\frac{1}{r}
\leq\widetilde{C}\|f\|_{L^{r}(\Omega)/\mathbb{R}}.
\end{align*}	
\end{corollary}

\subsection{Variable Lebesgue Spaces}
Let $r$ be a positive measurable function on $\mathcal{X}$. Let
$$\widetilde{r}_-:=\underset{x\in \mathcal{X}}
{\mathrm{ess\,inf}}\,r(x)\ \ \text{and}
\ \ \widetilde{r}_+:=\underset{x\in\mathcal{X}}
{\mathrm{ess\,sup}}\,r(x).$$
Fix $x_0\in \mathcal{X}$. The function $r$
is said to be \emph{globally
log-H\"older continuous on $\mathcal{X}$} if there exist
$r_{\infty}\in{\mathbb R}$ and a positive
constant $C$ such that, for any
$x,y\in \mathcal{X}$,
\begin{equation*}
|r(x)-r(y)|\leq \frac{C}{\log[e+\frac{1}{\rho(x,y)}]}
\ \ \text{and}\ \
|r(x)-r_\infty|\leq \frac{C}{\log[e+\rho(x_0,x )]}.
\end{equation*}
The \emph{variable Lebesgue space
$L^{r(\cdot)}(\mathcal{X})$} is defined to
be the set
of all measurable functions $f$ on $\mathcal{X}$
such that
\begin{equation*}
\|f\|_{L^{r(\cdot)}(\mathcal{X})}:=\inf\left\{\lambda
\in(0,\infty):\int_{\mathcal{X}}\left[\frac{|f(x)|}
{\lambda}\right]^{r(x)}\,d\mu(x)\le1\right\}<\infty.
\end{equation*}
The study of the variable Lebesgue space
can be traced back to \cite{kr91}.
These spaces were employed to
investigate both the related nonlinear elliptic boundary value problems
and the mapping properties of Nemytskii operators.
For more studies on variable Lebesgue spaces, we
refer to \cite{b18,bbd21,cf13,cw14,dhhr11,dhr09,ns12,n50,n51}.

Applying a slight modification of the proof of \cite[Theorem~4.20]{pyyz24}
with \cite[Theorem 2.12]{pyyz24} replaced by Theorem~\ref{MS2},
we obtain the following conclusion and omit the details.

\begin{theorem}\label{01011835}
Let $r$ be a positive globally log-H\"older continuous function.
Assume that $0<q< \widetilde r_{-}\leq \widetilde
r_{+}<\infty$.
Then there exist $C,\widetilde{C}\in(0,\infty)$ such that,
for any $f\in [L^{r(\cdot)}(\Omega)/\mathbb{R}]\cap\bigcup_{s\in(0,1)}
\dot{W}^{s,q}_{L^{r(\cdot)}}(\Omega)$,
\begin{align*}
C\|f\|_{L^{r(\cdot)}(\Omega)/\mathbb{R}}
&\leq \varliminf_{s \to 0^+}
s^{\frac{1}{q}}\left\|\left\{\int_\Omega
\frac{|f(\cdot)-f(y)|^q}{U(\cdot,y)[\rho(\cdot,y)]^{sq}}
\, d\mu(y) \right\}^{\frac{1}{q}}\right\|_{L^{r(\cdot)}
(\Omega)}\notag\\
&\leq \varlimsup_{s \to 0^+}
s^{\frac{1}{q}}\left\|\left\{\int_\Omega
\frac{|f(\cdot)-f(y)|^q}{U(\cdot,y)[\rho(\cdot,y)]^{sq}}
\, d\mu(y) \right\}^{\frac{1}{q}}\right\|_{L^{r(\cdot)}
(\Omega)}
\leq\widetilde{C}\|f\|_{L^{r(\cdot)}(\Omega)/\mathbb{R}}.
\end{align*}
\end{theorem}

\begin{remark}
When $\mathcal{X}=\mathbb{R}^n=\Omega$,
$\rho$ is the standard Euclidean distance,
and $\mu$ is the $n$-dimensional Lebesgue measure,
Theorem~\ref{01011835} in this case reduces to \cite[Theorem~4.20]{pyyz24}.
The other cases of Theorem~\ref{01011835} are new.
\end{remark}

\subsection{Weighted Lorentz Spaces}
Let $\omega$ be a weight on $\mathcal{X}$.
For any $r,\tau\in(0,\infty)$,
the \emph{weighted Lorentz space $L^{r,\tau}_{\omega}(\mathcal{X})$}
is defined to be the set of all measurable
functions
$f$ on $\mathcal{X}$ such that
\begin{equation*}
\|f\|_{L^{r,\tau}_{\omega}(\mathcal{X})}
:=\left\{\int_0^{\infty}
\left[t^{\frac{1}{r}}f^*(t)\right]^\tau
\,\frac{dt}{t}\right\}^{\frac{1}{\tau}}<\infty,
\end{equation*}
where, for any $t\in[0,\infty)$,
$f^*(t):=\inf\left\{s\in(0,\infty):
\omega(\left\{x\in \mathcal{X}:|f(x)|>s\right\})\leq t\right\}$.
If $\omega\equiv1$,
then $L^{r,\tau}_{\omega}(\mathcal{X})$ is the usual Lorentz space
introduced by Lorentz \cite{l50,l51},
which is denoted simply by $L^{r,\tau}(\mathcal{X})$.
Lorentz \cite{c64} established the intermediate spaces
of Lebesgue spaces via the real interpolation method.
We refer to
\cite{adnop23,chk1982,cf80,cf84,g93,h66,kk91,osttw12,st01}
for more studies on weighted Lorentz spaces.

By a slight modification of the proof of \cite[Theorem~4.22]{pyyz24}
with \cite[Theorem 2.12]{pyyz24} replaced by Theorem~\ref{MS2},
we obtain the following conclusion and omit the details.

\begin{theorem}\label{2047}
Let $r,\tau\in(0,\infty)$, $0<q\le p<\min\{r,\,\tau\}$,
and $\omega\in A_{\frac{r}{p}}(\mathcal{X})$.
Then there exist $C,\widetilde{C}\in(0,\infty)$ such that,
for any $f\in [L^{r,\tau}_{\omega}(\Omega)/\mathbb{R}]
\cap\bigcup_{s\in(0,1)}
\dot{W}^{s,q}_{L^{r,\tau}_{\omega}}(\Omega)$,
\begin{align*}
C\|f\|_{L^{r,\tau}_{\omega}(\Omega)/\mathbb{R}}
&\leq \varliminf_{s \to 0^+}
s^{\frac{1}{q}}\left\|\left\{\int_\Omega
\frac{|f(\cdot)-f(y)|^q}{U(\cdot,y)[\rho(\cdot,y)]^{sq}}
\, d\mu(y) \right\}^{\frac{1}{q}}\right\|_{L^{r,\tau}_{\omega}(\Omega)}
\notag\\
&\leq \varlimsup_{s \to 0^+}
s^{\frac{1}{q}}\left\|\left\{\int_\Omega
\frac{|f(\cdot)-f(y)|^q}{U(\cdot,y)[\rho(\cdot,y)]^{sq}}
\, d\mu(y) \right\}^{\frac{1}{q}}\right\|_{L^{r,\tau}_{\omega}(\Omega)}
\leq\widetilde{C}\|f\|_{L^{r,\tau}_{\omega}(\Omega)/\mathbb{R}}.
\end{align*}
\end{theorem}

\begin{remark}
When $\mathcal{X}=\mathbb{R}^n=\Omega$,
$\rho$ is the standard Euclidean distance,
$\mu$ is the $n$-dimensional Lebesgue measure,
and $\omega\equiv1$,
Theorem~\ref{2047} in this case
reduces to \cite[Theorem~4.22]{pyyz24}.
The other cases of Theorem~\ref{2047} are new.
\end{remark}

\subsection{Weighted Orlicz Spaces\label{s6.8}}

A non-decreasing function $\Phi:[0,\infty)
\to [0,\infty)$ is called an \emph{Orlicz function}
if $\Phi(0)= 0$, $\Phi(t)>0$ for any
$t\in(0,\infty)$, and
$\lim_{t\to\infty}\Phi(t)=\infty$.
The function $\Phi$ is said to be of \emph{lower}
(resp. \emph{upper}) \emph{type}
$r\in{\mathbb R}$ if there exists
$C_{(r)}\in[1,\infty)$ such that,
for any $t\in(0,\infty)$ and
$s\in(0,1)$ [resp. $s\in(1,\infty)$],
$\Phi(st)\leq C_{(r)} s^r\Phi(t)$.
Assume that
$\Phi$
is an Orlicz function with positive lower
type $r_{\Phi}^-$ and positive upper type
$r_{\Phi}^+$.
Let $\omega$ be a weight on $\mathcal{X}$.
The \emph{weighted Orlicz space
$L^{\Phi}_{\omega}(\mathcal{X})$}
is defined to be the set of all measurable
functions $f$ on $\mathcal{X}$ such that
\begin{equation*}
\|f\|_{L^{\Phi}_{\omega}(\mathcal{X})}:=\inf\left\{\lambda\in
(0,\infty):\int_{\mathcal{X}}\Phi\left(\frac{|f(x)|}
{\lambda}\right)\omega(x)\,d\mu(x)\le1\right\}<\infty.
\end{equation*}
Then $\|\cdot\|_{L^{\Phi}_{\omega}(\mathcal{X})}$ is a quasi-norm
and thereby $L^{\Phi}_{\omega}(\mathcal{X})$ is a ball quasi-Banach function space.
If $\Phi$ is convex, then $\|\cdot\|_{L^{\Phi}_{\omega}(\mathcal{X})}$ is a norm
and thereby $L^{\Phi}_{\omega}(\mathcal{X})$ is a ball Banach function space.
The unweighted Orlicz space $L^{\Phi}_{1}(\mathcal{X})$
was originally introduced in \cite{bo1931,o1932},
which is widely applied to various branches of analysis.
We refer to
\cite{b21,dfmn21,kk91,lyz2023,ns14,rr91,rr02,ylk2017}
for more studies on Orlicz spaces.

Let $\Phi$ and $\Psi$ be
two Orlicz functions.
Write $\Phi\approx\Psi$
if there exists a positive constant $C$ such that,
for any $t\in(0,\infty)$,
$\Phi(C^{-1}t)\le\Psi(t)\le\Phi(Ct)$.
In this case, $L^{\Phi}_{\omega}(\mathcal{X})=L^{\Psi}_{\omega}(\mathcal{X})$
with equivalent quasi-norms.
For any given Orlicz function $\Phi$,
if $1\le r_{\Phi}^-\le r_{\Phi}^+<\infty$,
then there exists a convex Orlicz function $\Psi$ such that
$\Phi\approx\Psi$.
In this case, $L^{\Phi}_{\omega}(\mathcal{X})$ is a Banach
space with the norm $\|\cdot\|_{L^{\Psi}_{\omega}(\mathcal{X})}$.
In the remainder of this subsection, we always assume that $\Phi$
is convex if $1\le r_{\Phi}^-\le r_{\Phi}^+<\infty$.
Let $\Phi^{(\frac{1}{p})}(t):=\Phi(t^\frac{1}{p})$
for any $t\in[0,\infty)$.
Then $r_{\Phi^{(\frac{1}{p})}}^{\pm}=\frac{r_{\Phi}^{\pm}}{p}$.
From this and the obvious fact that any convex Orlicz function
is strictly increasing, we may assume that any
Orlicz function is strictly increasing.

For any given Orlicz function $\Phi$ with
$1<r_{\Phi}^-\le r_{\Phi}^+<\infty$,
let $\widetilde{\Phi}$ be its \emph{complementary function}, that is,
for any $t\in[0,\infty)$,
\begin{align}\label{1539}
\widetilde{\Phi}(t):=\sup\left\{tu-\Phi(u):u\in[0,\infty)\right\}.
\end{align}
Then $\widetilde{\Phi}$ is also an Orlicz function with
$1<r_{\widetilde{\Phi}}^-\le r_{\widetilde{\Phi}}^+<\infty$
and $\widetilde{\widetilde{\Phi}}=\Phi$.
The complementary pair $(\Phi,\widetilde\Phi)$ of Orlicz functions
satisfies
$
\frac{1}{r_{\Phi}^{\pm}}+\frac{1}{r_{\widetilde{\Phi}}^{\mp}}=1
$ and
$t\le\Phi^{-1}(t)\,\widetilde\Phi^{-1}(t)
\le2t
$
for any $t\in[0,\infty)$.
Moreover,
$L^{\Phi}_{\omega}(\mathcal{X})$ and $L^{\widetilde{\Phi}}_{\omega}(\mathcal{X})$
are the dual and the associate spaces of each other.
If $1<r_{\Phi}^-\le r_{\Phi}^+<\infty$ and $\omega\in A_{r_{\Phi}^-}$,
then the Hardy--Littlewood maximal operator $\mathcal{M}$ is
bounded on $L^{\Phi}_{\omega}(\mathcal{X})$.

Let $\Phi$ be an Orlicz function with
lower type $r_{\Phi}^-$ and upper type $r_{\Phi}^+$.
Then, by the above fact
and $(\frac{r_{\Phi}^+}{p})'=r_{\widetilde{\Phi^{(\frac{1}{p})}}}^-$,
we conclude that, if $p\in(0,r_{\Phi}^-)$ and
$\omega\in A_{(\frac{r_{\Phi}^+}{p})'}(\mathcal{X})$, then
$
[L^{\Phi}_{\omega}(\mathcal{X})]^\frac{1}{p}
=L^{\Phi^{(\frac{1}{p})}}_{\omega}(\mathcal{X})
$
is a ball Banach function space
and the Hardy--Littlewood maximal operator $\mathcal{M}$ is bounded on
$[L^{\Phi^{(\frac{1}{p})}}_{\omega}(\mathcal{X})]'
=L^{\widetilde{\Phi^{(\frac{1}{p})}}}_{\omega}(\mathcal{X})$.

By a slight modification of the proof of \cite[Theorem~4.25]{pyyz24}
with \cite[Theorem 2.12]{pyyz24} replaced by Theorem~\ref{MS2},
we obtain the following conclusion and omit the details.

\begin{theorem}\label{01011841}
Let $\Phi$ be an Orlicz function with lower
type $r_{\Phi}^-$ and upper
type $r_{\Phi}^+$.
Assume that $0<q\le p< r_{\Phi}^-\leq r_{\Phi}^+<\infty$
and $\omega\in A_{(\frac{r_{\Phi}^+}{p})'}(\mathcal{X})$.
Then there exist $C,\widetilde{C}\in(0,\infty)$ such that,
for any $f\in[L^{\Phi}_{\omega}(\Omega)/\mathbb{R}]\cap\bigcup_{s\in(0,1)}
\dot{W}^{s,q}_{L^{\Phi}_{\omega}}(\Omega)$,
\begin{align*}
C\|f\|_{L^{\Phi}_{\omega}(\Omega)/\mathbb{R}}
&\leq \varliminf_{s \to 0^+}
s^{\frac{1}{q}}\left\|\left\{\int_\Omega
\frac{|f(\cdot)-f(y)|^q}{U(\cdot,y)[\rho(\cdot,y)]^{sq}}
\, d\mu(y) \right\}^{\frac{1}{q}}\right\|_{L^{\Phi}_{\omega}
(\Omega)}\notag\\
&\leq \varlimsup_{s \to 0^+}
s^{\frac{1}{q}}\left\|\left\{\int_\Omega
\frac{|f(\cdot)-f(y)|^q}{U(\cdot,y)[\rho(\cdot,y)]^{sq}}
\, d\mu(y) \right\}^{\frac{1}{q}}\right\|_{L^{\Phi}_{\omega}
(\Omega)}
\leq\widetilde{C}\|f\|_{L^{\Phi}_{\omega}(\Omega)/\mathbb{R}}.
\end{align*}
\end{theorem}

\begin{remark}
When
$\mathcal{X}=\mathbb{R}^n=\Omega$,
$\rho$ is the standard Euclidean distance,
$\omega\equiv1$,
and $\mu$ is the $n$-dimensional Lebesgue measure,
Theorem~\ref{01011841} in this case coincides with \cite[Theorem~4.25]{pyyz24}.
The other cases of Theorem~\ref{01011841} are new.
\end{remark}

\subsection{Generalized Morrey and Generalized Block Spaces}

The Morrey spaces were introduced by Morrey~\cite{m38}
in connection with PDEs.
For any function $\phi:\mathcal{X}\times(0,\infty)\to(0,\infty)$
and any ball $B=B(x,r)$
with $x\in \mathcal{X}$ and $r\in(0,\infty)$, we shall write
$\phi(B)$ in place of $\phi(x,r)$.
The \emph{local Lebesgue space}
$L^p_{\mathrm{loc}}(\mathcal{X})$ with $p\in(0,\infty)$ is
defined to be the set of all measurable functions
$f$ on $\mathcal{X}$ satisfying that,
for any $x\in \mathcal{X}$, there exists a positive
constant $r_x$ such that
$\|f\mathbf{1}_{B(x,r_x)}\|_{L^{p}(\mathcal{X})}<\infty$.
Now,
we recall the definition of generalized Morrey spaces.

\begin{definition}
Let $\phi:\mathcal{X}\times(0,\infty)\to(0,\infty)$ and $p\in(0,\infty)$.
The \emph{generalized Morrey space} $L_{p,\phi}(\mathcal{X})$
is defined to be the set of all $f\in L^p_{\mathrm{loc}}(\mathcal{X})$ such that
\begin{equation*}
\|f\|_{L_{p,\phi}(\mathcal{X})}
:= \sup_B \frac 1{\phi(B)}\left[ \frac 1{\mu(B)}
\int_{B} |f(x)|^p \,d\mu(x) \right]^{\frac{1}{p}}<\infty,
\end{equation*}
where the supremum is taken over all balls $B\subset\mathcal{X}$.
\end{definition}

The generalized Morrey space $L_{p,\phi}(\mathcal{X})$ is a quasi-Banach
space equipped with the quasi-norm $\|\cdot\|_{L_{p,\phi}(\mathcal{X})}$.
If $p\in[1,\infty)$, then $L_{p,\phi}(\mathcal{X})$ is a Banach space.
If  $p\in(0,\infty)$ and there exists a positive
constant $C$ such that, for any ball $B\subset\mathcal{X}$,
$C^{-1}\le\phi(B)[\mu(B)]^{\frac{1}{p}}\le C$,
then $L_{p,\phi}(\mathcal{X})=L^p(\mathcal{X})$.
In the case $\mathcal{X}=\mathbb{R}^n$,
the generalized Morrey space $L_{p,\phi}(\mathbb{R}^n)$
was introduced by Nakai \cite{n94} with
$\phi(x,r):=[|B(x,r)|^{-1}w(x,r)]^{\frac{1}{p}}$,
where $w:\mathbb{R}^n\times(0,\infty)\to(0,\infty)$.
We refer to \cite{a75,a15,cf87,sfh20a,sfh20b,st05,tnyyz24,y2025,ysy10}
for more studies on generalized Morrey spaces.

Next, we recall the concept of blocks.

\begin{definition}
Let $\phi:\mathcal{X}\times (0,\infty) \to (0,\infty)$
and $q\in(1,\infty]$.
A function $b$ is called a \emph{$[\phi,q]$-block}
if there exists a ball $B$, which is called the \emph{supporting ball} of $b$,
such that
\begin{enumerate}
\item[\rm (i)]
$\mathrm{supp\,}b\subset{B}$;
\item[\rm (ii)]
$\|b\|_{L^q(\mathcal{X})} \le \frac{1}{[
\mu(B)]^{\frac{1}{q'}}\phi(B)}$,
where $\frac{1}{q}+\frac{1}{q'}=1$.
\end{enumerate}
\end{definition}

Denote by $B[\phi,q]$ the set of all $[\phi,q]$-blocks.
Now, we recall the concept of generalized block spaces.
In what follows, for any given quasi-normed vector space $Y$,
let $Y^*$ be the dual space of $Y$.

\begin{definition}
Let $\phi:\mathcal{X}\times (0,\infty) \to (0,\infty)$
and $q\in(1,\infty]$.
Assume that ${L}_{q',\phi}(\mathcal{X})\not=\{0\}$.
The \emph{generalized block space}
$B^{[\phi,q]}(\mathcal{X})$ is defined to be the set
of all $f\in[L_{q',\phi}(\mathcal{X})]^*$
such that
there exist a sequence $\{b_j\}_{j\in\mathbb{N}}$ in $B[\phi,q]$
and a sequence $\{\lambda_j\}_{j\in\mathbb{N}}$ in $[0,\infty)$
satisfying $\sum_{j\in\mathbb{N}}\lambda_j <\infty$
such that
\begin{align}\label{expression}
f=\sum_{j\in\mathbb{N}}
\lambda_j b_j\ \text{in}\ \left[L_{q',\phi}(\mathcal{X})\right]^*.
\end{align}
For any $f\in B^{[\phi,q]}(\mathcal{X})$, define
\begin{equation*}
\|f\|_{B^{[\phi,q]}(\mathcal{X})}
:=\inf\left\{\sum_{j\in\mathbb{N}}\lambda_j:
f=\sum_{j\in\mathbb{N}}
\lambda_j b_j\ \text{in}\ \left[L_{q',\phi}(\mathcal{X})\right]^*\right\},
\end{equation*}
where the infimum is taken over all decompositions
of $f$ as in \eqref{expression}.
\end{definition}

The block space (the space generated by blocks) was introduced by
Lu et al.~\cite{ltw82} and Taibleson and Weiss~\cite{tw83}
in connection with the convergence of Fourier series.
The generalized block space $B^{[\phi,q]}(\mathcal{X})$ is a Banach space
equipped with the norm $\|\cdot\|_{B^{[\phi,q]}(\mathcal{X})}$,
which was introduced in \cite{n08}.
As was shown in \cite[Theorem~3.28]{tnyyz24} that,
under some additional assumptions on $\phi$,
the generalized
Morrey space $L_{q,\phi}(\mathcal{X})$ with $q\in(1,\infty)$
and the generalized block space $B^{[\phi,q']}(\mathcal{X})$
are mutually the associate space.
We refer to \cite{l84,st09,st15}
for more studies on generalized block spaces.

We consider the following conditions on $\phi$.

\begin{definition}
Let $\phi:\mathcal{X}\times(0,\infty)\to(0,\infty)$ and
$p\in(0,\infty)$.
\begin{enumerate}
\item[\rm(i)]
$\phi$ is said to satisfy the \emph{doubling condition} if
there exists a positive constant $C$ such that,
for any $x,y\in \mathcal{X}$ and $r,s\in(0,\infty)$
with $\frac{r}{s}\in[2^{-1},2]$,
\begin{align}\label{DC}
\frac1C\le\frac{\phi(x,r)}{\phi(x,s)}&\le C.
\end{align}

\item[\rm(ii)]
$\phi$ is said to satisfy the \emph{nearness condition} if
there exists a positive constant $C$ such that,
for any $x,y\in \mathcal{X}$ and $r\in(0,\infty)$
satisfying $\rho(x,y)\leq r$,
\begin{align}\label{NC}
\frac1C\le\frac{\phi(x,r)}{\phi(y,r)}&\le C.
\end{align}

\item[\rm(iii)]
$\phi$ is said to be \emph{almost increasing}
(resp. \emph{almost decreasing}) if
there exists a positive constant $C$ such that,
for any $x\in \mathcal{X}$ and $0<r<s<\infty$,
$\phi(x,r)\le C\phi(x,s)$
[resp. $C\phi(x,r)\ge \phi(x,s)$].

\item[\rm(iv)]
$\mathcal{G}^{\rm dec}_p$ is defined to be the set of all
functions $\phi:\mathcal{X}\times(0,\infty)\to(0,\infty)$
such that
$r\mapsto [\mu(B(x,r))]^{\frac{1}{p}}\phi(x,r)$ is almost increasing
and that
$r\mapsto\phi(x,r)$ is almost decreasing.
\end{enumerate}
\end{definition}

For a function $\alpha(\cdot):\mathcal{X}\to\mathbb{R}$, let
$\alpha_-:=\inf_{x\in \mathcal{X}}\alpha(x)$ and
$\alpha_+:=\sup_{x\in \mathcal{X}}\alpha(x)$.
Recall that $\alpha(\cdot):\mathcal{X}\to \mathbb{R}$ is
said to be \emph{log-H\"older continuous}
if there exists a positive constant
$C_{\alpha(\cdot)}$ such that,
for any $x,y\in \mathcal{X}$ with $\rho(x,y)\in(0,\frac{1}{2})$,
\begin{equation*}
|\alpha(x)-\alpha(y)|\le-\frac{C_{\alpha(\cdot)}}{\log\rho(x,y)}.
\end{equation*}
Let $\alpha_*\in\mathbb{R}$,
$\alpha(\cdot):\mathcal{X}\to\mathbb{R}$ be log-H\"older continuous with
$-\infty<\alpha_-\le\alpha_+<\infty$,
and, for any $x\in \mathcal{X}$,
\begin{equation*}
\phi(x,r):=
\begin{cases}
r^{\alpha(x)}&\text{if }r\in(0,\frac{1}{2}),\\
r^{\alpha_*}&\text{if }r\in[\frac{1}{2},\infty).
\end{cases}
\end{equation*}
Then $\phi$ satisfies both \eqref{DC} and \eqref{NC};
see \cite[Proposition~3.3]{n10}.
Moreover, if $\alpha_+,\alpha_*\in(-\infty,0]$ and
$[\mu(B(x,r))]^{\frac{1}{p}}\phi(x,r)$ is almost increasing,
then $\phi\in\mathcal{G}^{\rm dec}_p$.

Next, applying Theorems~\ref{MS2} and \ref{01011838},
we obtain the following conclusion.

\begin{theorem}\label{thm:Mor/R Gag}
Let $0<q<p<\infty$
and $\phi\in\mathcal{G}^{\rm dec}_p$ satisfy
both \eqref{NC} and that there exists
$C_{(\phi)}\in(0,\infty)$ such that,
for any $x\in \mathcal{X}$ and $r\in(0,\infty)$,
\begin{equation}\label{int phi}
\int_r^{\infty}\frac{\phi(x,t)}{t}\,dt\le C_{(\phi)}\,\phi(x,r).
\end{equation}
Assume that $\phi$ satisfies that there exists
$\widetilde{C}_{(\phi)}\in[1,\infty )$ such that,
for any ball $B\subset\mathcal{X}$,
\begin{equation}\label{const}
\frac{1}{\widetilde{C}_{(\phi)}}\le\phi(B)\left[
\mu(B)\right]^\frac{1}{p}\le\widetilde{C}_{(\phi)}
\end{equation}
or assume that $\phi$ satisfies both
\begin{align}\label{r 0}
\lim_{r\to0^+}\inf_{x\in B}\phi(x,r)=\infty
\end{align}
for any ball $B\subset\mathcal{X}$ and
\begin{align}\label{r infty}
\lim_{r\to\infty}\phi(x,r)\left[\mu(B(x,r))\right]^{\frac{1}{p}}=\infty
\end{align} for any $x\in \mathcal{X}$.
Then there exist $C,\widetilde{C}\in(0,\infty)$ such that,
for any $f\in [L_{p,\phi}(\Omega)/{\mathbb R}]\cap\bigcup_{s\in(0,1)}
\dot{W}^{s,q}_{L_{p,\phi}}(\Omega)$,
\begin{align}\label{Gag M}
C\|f\|_{L_{p,\phi}(\Omega)/{\mathbb R}}
&\nonumber\leq
\varliminf_{s \to 0^+}
s^{\frac{1}{q}}\left\|\left\{\int_\Omega
\frac{|f(\cdot)-f(y)|^q}{U(\cdot,y)[\rho(\cdot,y)]^{sq}}
\, d\mu(y) \right\}^{\frac{1}{q}}\right\|_{L_{p,\phi}(\Omega)}\\
&\leq
\varlimsup_{s \to 0^+}
s^{\frac{1}{q}}\left\|\left\{\int_\Omega
\frac{|f(\cdot)-f(y)|^q}{U(\cdot,y)[\rho(\cdot,y)]^{sq}}
\, d\mu(y) \right\}^{\frac{1}{q}}\right\|_{L_{p,\phi}(\Omega)}
\leq\widetilde{C}\|f\|_{L_{p,\phi}(\Omega)/{\mathbb R}}.
\end{align}
\end{theorem}

\begin{proof}
First, we prove the lower estimate in \eqref{Gag M}.
From Remark~\ref{2048}(i),
\cite[Proposition 3.31]{tnyyz24}, and the assumption that
$\phi$ satisfies both \eqref{DC} and \eqref{int phi},
we deduce that there exists $\theta\in(0,1)$ such that,
for any $f\in L_{p,\phi}(\Omega)$,
\begin{equation*}
\|f\|_{L_{p,\phi}(\Omega)}
\le
\sup_{B}\frac{\|f\|_{L^p_{({\mathcal M}
\mathbf{1}_B)^{1-\theta}}(\Omega)}}{\phi(B)[\mu(B)]^\frac{1}{p}}
\le
C\|f\|_{L_{p,\phi}(\Omega)},
\end{equation*}
where $C\in(0,\infty)$ is independent of $f$,
which, together with Proposition~\ref{2357},
further implies that, for any $f\in L_{p,\phi}(\Omega)/{\mathbb R}$,
\begin{equation*}
\|f\|_{L_{p,\phi}(\Omega)/{\mathbb R}}
\le
\sup_{B}\frac{\|f\|_{L^p_{({\mathcal M}\mathbf{1}_B)^{1-\theta}}(
\Omega)/{\mathbb R}}}{\phi(B)[\mu(B)]^\frac{1}{p}}
\le
C\|f\|_{L_{p,\phi}(\Omega)/{\mathbb R}}.
\end{equation*}
By this, the fact that
$\sup_B\big[({\mathcal M}\mathbf{1}_B)^{1-\theta}\big]_{A_1}<\infty$
(see, for example, \cite[Theorem~281]{sfh20a}),
and Theorem~\ref{01011838}, we find that,
for any $f\in L_{p,\phi}(\Omega)/{\mathbb R}$,
\begin{align*}
\|f\|_{L_{p,\phi}(\Omega)/{\mathbb R}}
&\leq
\sup_{B}\frac{\|f\|_{L^{p}_{({\mathcal M}\mathbf{1}_B)^{1-\theta}}(
\Omega)/{\mathbb R}}}{\phi(B)[\mu(B)]^\frac{1}{p}} \\
&\lesssim
\sup_{B}\frac1{\phi(B)[\mu(B)]^\frac{1}{p}}
\varliminf_{s \to 0^+}
s^{\frac{1}{q}}
\left\|\left\{\int_\Omega
\frac{|f(\cdot)-f(y)|^q}{U(\cdot,y)[\rho(\cdot,y)]^{sq}}
\, d\mu(y) \right\}^{\frac{1}{q}}\right\|_{L^{p}_{({\mathcal M}
\mathbf{1}_B)^{1-\theta}}(\Omega)} \\
&\leq
\varliminf_{s \to 0^+}
s^{\frac{1}{q}}
\sup_{B}\frac1{\phi(B)[\mu(B)]^\frac{1}{p}}
\left\|\left\{\int_\Omega
\frac{|f(\cdot)-f(y)|^q}{U(\cdot,y)[\rho(\cdot,y)]^{sq}}
\, d\mu(y) \right\}^{\frac{1}{q}}\right\|_{L^{p}_{({\mathcal M}
\mathbf{1}_B)^{1-\theta}}(\Omega)} \\
&\lesssim\varliminf_{s \to 0^+}
s^{\frac{1}{q}}
\left\|\left\{\int_\Omega
\frac{|f(\cdot)-f(y)|^q}{U(\cdot,y)[\rho(\cdot,y)]^{sq}}
\, d\mu(y) \right\}^{\frac{1}{q}}\right\|_{L_{p,\phi}(\Omega)}.
\end{align*}

Now, we show the upper estimate in \eqref{Gag M}.
Let $p_1\in(q,p)$. Then $\phi^{p_1}\in
\mathcal{G}^{\rm dec}_{\frac{p}{p_1}}$ satisfies \eqref{NC}.
Moreover, from \cite[Lemma~7.1]{n08},
we infer that
$\phi^{p_1}$ satisfies \eqref{int phi}.
By $\frac{p}{p_1}>1$ and
\cite[Theorem~3.20]{tnyyz24}, we conclude that
$[L_{p,\phi}(\mathcal{X})]^\frac{1}{p_1}=
L_{\frac{p}{p_1},\phi^{p_1}}(\mathcal{X})$
is a ball Banach function space.
From \cite[Theorems~3.28 and~3.30]{tnyyz24},
we deduce that the Hardy--Littlewood maximal
operator is bounded on
\begin{equation*}
B^{[\phi^{p_1},(\frac{p}{p_1})']}(\mathcal{X})
=
\left[L_{\frac{p}{p_1},\phi^{p_1}}(\mathcal{X})\right]'
=
\left(\left[L_{p,\phi}(\mathcal{X})\right]^\frac{1}{p_1}\right)'.
\end{equation*}
By these and the proof of Theorem~\ref{MS2},
we find that the last inequality in \eqref{Gag M} holds.
This finishes the proof of Theorem~\ref{thm:Mor/R Gag}.
\end{proof}

\begin{remark}
To the best of our knowledge,
Theorem~\ref{thm:Mor/R Gag} is completely new.
\end{remark}

The following theorem is an application of Theorem~\ref{MS3}.

\begin{theorem}\label{thm:Mor Gag}
Let $p\in(0,\infty)$, $\beta\in(0,\infty)$,
and $\phi\in\mathcal{G}^{\rm dec}_p$ satisfy
both $\eqref{NC}$ and \eqref{int phi}.
Assume that $\phi$ satisfies \eqref{const} or
assume that $\phi$ satisfies both \eqref{r 0} and \eqref{r infty}.
Then there exist $C,\widetilde{C}\in(0,\infty)$ such that,
for any $f\in C_{\mathrm{b}}^\beta(\Omega)
\cap\bigcup_{s\in(0,1)}\dot{W}^{s,p}_{L_{p,\phi}}(\Omega)$,
\begin{align}\label{Gag 2}
C\|f\|_{L_{p,\phi}(\Omega)}
&\notag\leq
\varliminf_{s \to 0^+}
s^{\frac{1}{p}}\left\|\left\{\int_\Omega
\frac{|f(\cdot)-f(y)|^p}{U(\cdot,y)[\rho(\cdot,y)]^{sp}}
\, d\mu(y) \right\}^{\frac{1}{p}}\right\|_{{L_{p,\phi}(\Omega)}}\\
&\leq
\varlimsup_{s \to 0^+}
s^{\frac{1}{p}}\left\|\left\{\int_\Omega
\frac{|f(\cdot)-f(y)|^p}{U(\cdot,y)[\rho(\cdot,y)]^{sp}}
\, d\mu(y) \right\}^{\frac{1}{p}}\right\|_{{L_{p,\phi}(\Omega)}}
\leq\widetilde{C}\|f\|_{L_{p,\phi}(\Omega)}.
\end{align}
\end{theorem}

\begin{proof}
Let $\theta\in(0,1]$.
Then $\phi^{p\theta}\in\mathcal{G}^{\rm dec}_\frac{1}{\theta}$
satisfies \eqref{NC}, which, combined with \cite[Theorem~3.20]{tnyyz24},
implies that $L_{\frac{1}{\theta}, \phi^{p\theta}}(\mathcal{X})$
is a ball Banach function space.
Moreover, from the proof of \cite[Lemma~7.1]{n08},
we infer that, for any $x\in \mathcal{X}$ and $r\in(0,\infty)$,
\begin{equation*}
\int_r^{\infty}\frac{[\phi(x,t)]^{p\theta}}{t}\,dt
\le C_{(\phi^{p\theta})}\left[\phi(x,r)\right]^{p\theta}
\end{equation*}
with
\begin{equation*}
C_{(\phi^{p\theta})}
=
\begin{cases}
2C_{(\phi)}\left(\frac{1-p\theta}{p\theta}\right)^{1-p\theta}
&\text{if }p\theta\in(0,1),\\
C_{(\phi)}\left(\frac{C^2}{\log 2}\right)^{p\theta-1}&\text{if } p\theta\in[1,\infty),
\end{cases}
\end{equation*}
where $C$ and $C_{(\phi)}$ are, respectively, the constants
in \eqref{DC} and \eqref{int phi}.
Thus, $\phi^{p\theta}$ satisfies \eqref{int phi}.
By \cite[Theorems~3.28 and~3.30]{tnyyz24},
we conclude that,
if $\theta\in(0,1)$,
then the Hardy--Littlewood maximal operator $\mathcal M$ is bounded on
\begin{equation*}
B^{[\phi^{p\theta},(\frac{1}{\theta})']}(\mathcal{X})
=
\left[L_{\frac{1}{\theta},\phi^{p\theta}}(\mathcal{X})\right]'
=
\left(\left[L_{p,\phi}(\mathcal{X})\right]^\frac{1}{p\theta}\right)'.
\end{equation*}
If $\theta\to1^-$, then $(\frac{1}{\theta})'\to\infty$ and hence
\begin{equation*}
\lim_{\theta\to1^-}\|\mathcal M\|_{L^{(\frac{1}{\theta})'}(\mathcal{X})
\to L^{(\frac{1}{\theta})'}(\mathcal{X})}<\infty
\ \ \text{and}\ \
\lim_{\theta\to1^-}C_{(\phi^{p\theta})}<\infty.
\end{equation*}
From the proof of \cite[Theorem~3.30]{tnyyz24}, we deduce that
\begin{equation*}
\lim_{\theta\to1^-}\|\mathcal M\|_{B^{[\phi^{p\theta},(\frac{1}{\theta})']}(\mathcal{X})
\to B^{[\phi^{p\theta},(\frac{1}{\theta})']}(\mathcal{X})}
<\infty.
\end{equation*}
Therefore, the Hardy--Littlewood maximal operator $\mathcal M$
is endpoint bounded on
$$
\left([L_{p,\phi}(\mathcal{X})]^\frac{1}{p}\right)'
=\left[L_{1,\phi^p}(\mathcal{X})\right]'.
$$
By this and Theorem~\ref{MS3}(i), we find that
\eqref{Gag 2} holds,
which completes the proof of Theorem~\ref{thm:Mor Gag}.
\end{proof}

\begin{remark}
To the best of our knowledge,
Theorem~\ref{thm:Mor Gag} is completely new.
\end{remark}

Let $\mathcal{X}=\mathbb{R}^n$, $p\in(0,\infty)$, $\lambda_*\in(0,n]$,
and $\lambda(\cdot):\mathbb{R}^n\to(0,n]$.
Let $\Omega\subseteqq\mathbb{R}^n$
satisfy the WMD condition.
For any $x\in\mathbb{R}^n$, let
\begin{equation}\label{lambda(dot)}
\phi(x,r)=
\begin{cases}
r^{-\frac{\lambda(x)}{p}}&\text{if }r\in(0,\frac{1}{2}), \\
r^{-\frac{\lambda_*}{p}}&\text{if }r\in[\frac{1}{2},\infty).
\end{cases}
\end{equation}
Denote $L_{p,\phi}(\mathbb{R}^n)$ by
$L^{p,\lambda(\cdot);\lambda_*}(\mathbb{R}^n)$.
Then we have the following corollary.

\begin{corollary}
Let $p\in(0,\infty)$.
Assume that $\lambda(\cdot):\mathbb{R}^n\to(0,n]$ is log-H\"older continuous,
$0<\lambda_-\le\lambda_+\le n$, and $0<\lambda_*<n$.
\begin{enumerate}
\item[\rm (i)]
If $q\in(0,p)$, then there exist $c,\widetilde{c}\in(0,\infty)$ such that,
for any $f\in [L^{p,\lambda(\cdot);\lambda_*}(\Omega)/{\mathbb R}]
\cap\bigcup_{s\in(0,1)}
\dot{W}^{s,q}_{L^{p,\lambda(\cdot);\lambda_*}}(\Omega)$,
\begin{align*}
c\|f\|_{L^{p,\lambda(\cdot);\lambda_*}(\Omega)/{\mathbb R}}
&\leq
\varliminf_{s \to 0^+}
s^{\frac{1}{q}}\left\|\left[\int_\Omega
\frac{|f(\cdot)-f(y)|^q}{|\cdot-y|^{n+sq}}
\, dy \right]^{\frac{1}{q}}\right\|_{{L^{p,\lambda(\cdot);\lambda_*}(\Omega)}}\\
&\leq
\varlimsup_{s \to 0^+}
s^{\frac{1}{q}}\left\|\left[\int_\Omega
\frac{|f(\cdot)-f(y)|^q}{|\cdot-y|^{n+sq}}
\, dy \right]^{\frac{1}{q}}\right\|_{{L^{p,\lambda(\cdot);\lambda_*}(\Omega)}}
\leq\widetilde{c}\|f\|_{L^{p,\lambda(\cdot);\lambda_*}(\Omega)/{\mathbb R}}.\notag
\end{align*}
\item[\rm (ii)]
If $q=p$ and $\beta\in(0,1]$,
then there exist $C,\widetilde{C}\in(0,\infty)$ such that,
for any $f\in C_{\mathrm{b}}^\beta(\Omega)
\cap\bigcup_{s\in(0,1)}
\dot{W}^{s,p}_{L^{p,\lambda(\cdot);\lambda_*}}(\Omega)$,
\begin{align*}
C\|f\|_{L^{p,\lambda(\cdot);\lambda_*}(\Omega)}
&\leq
\varliminf_{s \to 0^+}
s^{\frac{1}{p}}\left\|\left[\int_\Omega
\frac{|f(\cdot)-f(y)|^p}{|\cdot-y|^{n+sp}}
\, dy \right]^{\frac{1}{p}}\right\|_{L^{p,\lambda(\cdot);\lambda_*}(\Omega)}\\
&\leq
\varlimsup_{s \to 0^+}
s^{\frac{1}{p}}\left\|\left[\int_\Omega
\frac{|f(\cdot)-f(y)|^p}{|\cdot-y|^{n+sp}}
\, dy \right]^{\frac{1}{p}}\right\|_{L^{p,\lambda(\cdot);\lambda_*}(\Omega)}
\leq\widetilde{C}\|f\|_{L^{p,\lambda(\cdot);\lambda_*}(\Omega)}.\notag
\end{align*}
\end{enumerate}
\end{corollary}

By a slight modification of the proof of \cite[Theorem~3.37]{tnyyz24}
with \cite[Theorem 2.11(i)]{tnyyz24} replaced by Theorem~\ref{MS2},
we obtain the following conclusion and omit the details.

\begin{theorem}\label{thm:B Gag}
Let $p,q\in(1,\infty)$ satisfy $\frac{1}{p}+\frac{1}{q}=1$
and $\phi\in\mathcal{G}^{\rm dec}_{p}$ satisfy $\eqref{NC}$.
Assume that $\phi$ satisfies \eqref{const} or
assume that $\phi$ satisfies both \eqref{r 0} and \eqref{r infty}.
If $q_1\in(0,1]$ and $\beta\in(0,\infty)$,
then there exist $C,\widetilde{C}\in(0,\infty)$ such that,
for any $f\in C_{\mathrm{b}}^\beta(\Omega)
\cap\bigcup_{s\in(0,1)}\dot{W}^{s,q_1}_{B^{[\phi,q]}}(\Omega)$,
\begin{align*}
C\|f\|_{B^{[\phi,q]}(\Omega)}
&\leq
\varliminf_{s \to 0^+}
s^{\frac{1}{q_1}}\left\|\left\{\int_\Omega
\frac{|f(\cdot)-f(y)|^{q_1}}{U(\cdot,y)[\rho(\cdot,y)]^{sq_1}}
\, d\mu(y) \right\}^{\frac{1}{q_1}}\right\|_{B^{[\phi,q]}(\Omega)}\\
&\leq
\varlimsup_{s \to 0^+}
s^{\frac{1}{q_1}}\left\|\left\{\int_\Omega
\frac{|f(\cdot)-f(y)|^{q_1}}{U(\cdot,y)[\rho(\cdot,y)]^{sq_1}}
\, d\mu(y) \right\}^{\frac{1}{q_1}}\right\|_{B^{[\phi,q]}(\Omega)}
\leq\widetilde{C}\|f\|_{B^{[\phi,q]}(\Omega)}.\notag
\end{align*}
\end{theorem}

\begin{remark}
To the best of our knowledge,
Theorem~\ref{thm:B Gag} is completely new.
\end{remark}

Let $\mathcal{X}=\mathbb{R}^n$, $p\in(0,\infty)$, $\lambda_*\in(0,n]$,
and $\lambda(\cdot):\mathbb{R}^n\to(0,n]$.
For $\phi$ defined in \eqref{lambda(dot)},
denote $B^{[\phi,q]}(\mathbb{R}^n)$ by
$B^{\lambda(\cdot);\lambda_*,q}(\mathbb{R}^n)$.
The following corollary is a direct consequence of
Theorem~\ref{thm:B Gag}.

\begin{corollary}
Let $p,q\in(1,\infty)$ satisfy $\frac{1}{p}+\frac{1}{q}=1$.
Assume that $\lambda(\cdot):\mathbb{R}^n\to(0,n]$ is log-H\"older continuous,
$0<\lambda_-\le\lambda_+\le n$, and $0<\lambda_*<n$.
If $q_1\in(0,1]$ and $\beta\in(0,1]$,
then there exist $C,\widetilde{C}\in(0,\infty)$ such that,
for any $f\in C_{\mathrm{b}}^\beta(\Omega)
\cap\bigcup_{s\in(0,1)}\dot{W}^{s,q}_{B^{\lambda(\cdot);
\lambda_*,q}}(\Omega)$,
\begin{align*}
C\|f\|_{B^{\lambda(\cdot);\lambda_*,q}(\Omega)}
&\leq
\varliminf_{s \to 0^+}
s^{\frac{1}{q_1}}\left\|\left[\int_\Omega
\frac{|f(\cdot)-f(y)|^{q_1}}{|\cdot-y|^{n+sq_1}}
\, dy \right]^{\frac{1}{q_1}}\right\|_{B^{\lambda(\cdot);
\lambda_*,q}(\Omega)}\\
&\leq
\varlimsup_{s \to 0^+}
s^{\frac{1}{q_1}}\left\|\left[\int_\Omega
\frac{|f(\cdot)-f(y)|^{q_1}}{|\cdot-y|^{n+sq_1}}
\, dy \right]^{\frac{1}{q_1}}\right\|_{B^{\lambda(\cdot);
\lambda_*,q}(\Omega)}
\leq\widetilde{C}
\|f\|_{B^{\lambda(\cdot);\lambda_*,q}(\Omega)}.
\end{align*}
\end{corollary}

\subsection{Generalized Lorentz--Morrey and Generalized Lorentz-Block Spaces}

In this subsection,
we aim to establish the Maz'ya--Shaposhnikova representation
on generalized Lorentz--Morrey spaces
and generalized Lorentz-block spaces.
To this end, we begin with recalling the
concept of generalized Lorentz--Morrey spaces.
For a function $\phi:\mathcal{X}\times(0,\infty)\to(0,\infty)$
and a ball $B:=B(x,r)$ with $x\in \mathcal{X}$ and
$r\in(0,\infty)$, we shall write $\phi(B)$ in place of $\phi(x,r)$.
The \emph{locally Lorentz space} $L^{p,p_1}_{\mathrm{loc}}(\mathcal{X})$
with $p,p_1\in(0,\infty)$ is
defined to be the set of all measurable functions $f$ on $\mathcal{X}$ such that,
for any $x\in \mathcal{X}$, there exists a positive
constant $r_x$ satisfying
$\|f\mathbf{1}_{B(x,r_x)}\|_{L^{p,p_1}(\mathcal{X})}<\infty$.

\begin{definition}
Let $\phi:\mathcal{X}\times(0,\infty)\to(0,\infty)$ and $p,p_1\in(0,\infty)$.
The \emph{generalized Lorentz--Morrey space}
$L_{p,p_1,\phi}(\mathcal{X})$
is defined to be the set of all $f\in
L^{p,p_1}_{\mathrm{loc}}(\mathcal{X})$ such that
\begin{align*}
\|f\|_{L_{p,p_1,\phi}(\mathcal{X})}
:=\sup_B \frac{\|f\mathbf{1}_B\|_{L^{p,p_1}(\mathcal{X})}}
{\phi(B)[\mu(B)]^\frac{1}{p}}<\infty,
\end{align*}
where the supremum is taken over all balls $B\subset\mathcal{X}$.
\end{definition}

It is easy to prove that the Lorentz--Morrey space
$L_{p,p_1,\phi}(\mathcal{X})$ is a quasi-Banach space equipped
with the quasi-norm $\|\cdot\|_{L_{p,p_1,\phi}(\mathcal{X})}$.
If $p,p_1\in(1,\infty)$, then
there exists a norm $\||\cdot\||_{L_{p,p_1,\phi}(\mathcal{X})}$
such that, for any $f\in L_{p,p_1,\phi}(\mathcal{X})$,
\begin{align*}
\|f\|_{L_{p,p_1,\phi}(\mathcal{X})}
\le\||f\||_{L_{p,p_1,\phi}(\mathcal{X})}
\le\frac{p}{p-1}\|f\|_{L_{p,p_1,\phi}(\mathcal{X})}.
\end{align*}
That is,
$L_{p,p_1,\phi}(\mathcal{X})$ is a Banach space equipped with the
norm $\||\cdot\||_{L_{p,p_1,\phi}(\mathcal{X})}$.
If  $p,p_1\in(0,\infty)$ and there exists a positive
constant $C$ such that, for any ball $B\subset\mathcal{X}$,
$C^{-1}\le\phi(B)[\mu(B)]^\frac{1}{p}\le C$,
then $L_{p,p_1,\phi}(\mathcal{X})=L^{p,p_1}(\mathcal{X})$.
For more studies of Lorentz--Morrey spaces,
we refer to \cite{h24,llsy24}.

We recall the concept of generalized
Lorentz-block spaces. We begin with the following definition
of $[\phi,q,q_1]$-blocks.

\begin{definition}
Let $\phi:\mathcal{X}\times(0,\infty)\to(0,\infty)$
and $q,q_1\in(1,\infty)$.
A function $b$ is called a
\emph{$[\phi,q,q_1]$-block}
if there exists a ball $B$, called the \emph{supporting ball} of $b$,
such that
\begin{enumerate}
\item[\rm (i)]
$\mathrm{supp\,}b\subset{B}$;
\item[\rm (ii)]
$\|b\|_{L^{q,q_1}(\mathcal{X})} \le
\frac{1}{[\mu(B)]^\frac{1}{q'}\phi(B)}$,
where $\frac{1}{q}+\frac{1}{q'}=1$.
\end{enumerate}
Denote by $B[\phi,q,q_1]$ the set of all $[\phi,q,q_1]$-blocks.
\end{definition}

\begin{definition}
Let $\phi:\mathcal{X}\times(0,\infty)\to(0,\infty)$
and $q,q_1\in(1,\infty)$.
Assume that ${L}_{q',q_1',\phi}(\mathcal{X})\not=\{0\}$.
The \emph{generalized Lorentz-block space} $B^{[\phi,q,q_1]}(\mathcal{X})$
is defined to be the set of all
$f\in[L_{q',q_1',\phi}(\mathcal{X})]^*$
such that
there exist a sequence $\{b_j\}_{j\in\mathbb{N}}$ in $B[\phi,q,q_1]$
and a sequence $\{\lambda_j\}_{j\in\mathbb{N}}$ in $[0,\infty)$
satisfying $\sum_{j\in\mathbb{N}}\lambda_j<\infty$
such that
\begin{equation}\label{L-expression}
f=\sum_j \lambda_j
b_j\ \text{in}\ \left[L_{q',q_1',\phi}(\mathcal{X})\right]^*.
\end{equation}
For any $f\in B^{[\phi,q,q_1]}(\mathcal{X})$, define
\begin{equation*}
\|f\|_{B^{[\phi,q,q_1]}(\mathcal{X})}
:=\inf\left\{\sum_{j\in\mathbb{N}}\lambda_j:
f=\sum_{j\in\mathbb{N}} \lambda_j
b_j\ \text{in}\ \left[L_{q',q_1',\phi}(\mathcal{X})\right]^*\right\},
\end{equation*}
where the infimum is taken over all decompositions
of $f$ as in \eqref{L-expression}.
\end{definition}

\begin{remark}
Let $\phi:\mathcal{X}\times(0,\infty)\to(0,\infty)$
and $q,q_1\in(1,\infty)$.
As was shown in \cite[Theorem~3.51]{tnyyz24} that,
under some additional assumptions on $\phi$,
the generalized
Lorentz--Morrey space $L_{q,q_1,\phi}(\mathcal{X})$
and the generalized Lorentz-block space $B^{[\phi,q',q_1']}(\mathcal{X})$
are mutually the associate space.
\end{remark}

Applying Theorems~\ref{MS2} and \ref{2047},
we obtain the following Maz'ya--Shaposhnikova formula
on generalized Lorentz--Morrey spaces.

\begin{theorem}\label{thm:LMor/R Gag}
Let $p,p_1,q\in(0,\infty)$ satisfy $q<\min\{p,\,p_1\}$
and let $\phi\in\mathcal{G}^{\rm dec}_p$ satisfy
both $\eqref{NC}$ and \eqref{int phi}.
Assume that $\phi$ satisfies \eqref{const} or
assume that $\phi$ satisfies both \eqref{r 0} and \eqref{r infty}.
Then there exist $C,\widetilde{C}\in(0,\infty)$ such that,
for any $f\in [L_{p,p_1,\phi}(\Omega)/{\mathbb R}]\cap
\bigcup_{s\in(0,1)}\dot{W}^{s,q}_{L_{p,p_1,\phi}}(\Omega)$,
\begin{align}\label{Gag LM}
C\|f\|_{L_{p,p_1,\phi}(\Omega)/{\mathbb R}}
&\notag\leq
\varliminf_{s \to 0^+}
s^{\frac{1}{q}}\left\|\left\{\int_\Omega
\frac{|f(\cdot)-f(y)|^q}{U(\cdot,y)[\rho(\cdot,y)]^{sq}}
\, d\mu(y) \right\}^{\frac{1}{q}}\right\|_{L_{p,p_1,\phi}(\Omega)}\\
&\leq
\varlimsup_{s \to 0^+}
s^{\frac{1}{q}}\left\|\left\{\int_\Omega
\frac{|f(\cdot)-f(y)|^q}{U(\cdot,y)[\rho(\cdot,y)]^{sq}}
\, d\mu(y) \right\}^{\frac{1}{q}}\right\|_{L_{p,p_1,\phi}(\Omega)}
\leq\widetilde{C}\|f\|_{L_{p,p_1,\phi}(\Omega)/{\mathbb R}}.
\end{align}
\end{theorem}

\begin{proof}
We first porve the lower estimate in \eqref{Gag LM}.
From Remark~\ref{2048}(i),
\cite[Proposition 3.54]{tnyyz24}, and the assumption that
$\phi$ satisfies both \eqref{DC} and \eqref{int phi},
we infer that there exists $\theta\in(0,1)$ such that,
for any $f\in L_{p,p_1,\phi}(\Omega)$,
\begin{equation*}
\|f\|_{L_{p,p_1,\phi}(\Omega)}
\le
\sup_{B}\frac{\|f\|_{L^{p,p_1}_{({\mathcal M}
\mathbf{1}_B)^{1-\theta}}(\Omega)}}
{\phi(B)[\mu(B)]^\frac{1}{p}}
\le
C\|f\|_{L_{p,p_1,\phi}(\Omega)},
\end{equation*}
where $C\in(0,\infty)$ is independent of $f$,
which further implies that,
for any $f\in L_{p,p_1,\phi}(\Omega)/{\mathbb R}$,
\begin{equation*}
\|f\|_{L_{p,p_1,\phi}(\Omega)/{\mathbb R}}
\le
\sup_{B}\frac{\|f\|_{L^{p,p_1}_{({\mathcal M}\mathbf{1}_B)^{1-\theta}}
(\Omega)/{\mathbb R}}}{\phi(B)[\mu(B)]^\frac{1}{p}}
\le
C\|f\|_{L_{p,p_1,\phi}(\Omega)/{\mathbb R}}.
\end{equation*}
By this, the fact that
$\sup_B\big[({\mathcal M}\mathbf{1}_B)^{1-\theta}\big]_{A_1}<\infty$
(see, for example, \cite[Theorem~281]{sfh20a}),
and Theorem~\ref{2047}, we conclude that,
for any $f\in L_{p,p_1,\phi}(\Omega)/{\mathbb R}$,
\begin{align*}
\|f\|_{L_{p,p_1,\phi}(\Omega)/{\mathbb R}}
&\leq
\sup_{B}\frac{\|f\|_{L^{p,p_1}_{({\mathcal M}\mathbf{1}_B)^{1-\theta}}
(\Omega)/{\mathbb R}}}{\phi(B)[\mu(B)]^\frac{1}{p}}\\
&\lesssim
\sup_{B}\frac1{\phi(B)[\mu(B)]^\frac{1}{p}}
\varliminf_{s \to 0^+}
s^{\frac{1}{q}}
\left\|\left\{\int_\Omega
\frac{|f(\cdot)-f(y)|^q}{U(\cdot,y)[\rho(\cdot,y)]^{sq}}
\, d\mu(y) \right\}^{\frac{1}{q}}\right\|_{L^{p,p_1}_{({\mathcal M}
\mathbf{1}_B)^{1-\theta}}(\Omega)}\\
&\leq
\varliminf_{s \to 0^+}
s^{\frac{1}{q}}
\sup_{B}\frac1{\phi(B)[\mu(B)]^\frac{1}{p}}
\left\|\left\{\int_\Omega
\frac{|f(\cdot)-f(y)|^q}{U(\cdot,y)[\rho(\cdot,y)]^{sq}}
\, d\mu(y) \right\}^{\frac{1}{q}}\right\|_{L^{p,p_1}_{({\mathcal M}
\mathbf{1}_B)^{1-\theta}}(\Omega)}\\
&\lesssim\varliminf_{s \to 0^+}
s^{\frac{1}{q}}
\left\|\left\{\int_\Omega
\frac{|f(\cdot)-f(y)|^q}{U(\cdot,y)[\rho(\cdot,y)]^{sq}}
\, d\mu(y) \right\}^{\frac{1}{q}}\right\|_{L_{p,p_1,\phi}(\Omega)}.
\end{align*}
This finishes the lower estimate in \eqref{Gag LM}.

Next, we show the upper estimate in \eqref{Gag LM}.
Let $p_2\in(q,\min\{p,\,p_1\})$. Then $\phi^{p_2}\in
\mathcal{G}^{\rm dec}_{\frac{p}{p_2}}$ satisfies \eqref{NC}.
Moreover, from \cite[Lemma~7.1]{n08}, we deduce that
$\phi^{p_2}$ satisfies \eqref{int phi}.
By $\frac{p}{p_2}>1$, $\frac{p_1}{p_2}>1$,
and \cite[Theorem 3.45]{tnyyz24}, we find that
$[L_{p,p_1,\phi}(\mathcal{X})]^\frac{1}{p_2}=L_{\frac{p}{p_2},
\frac{p_1}{p_2},\phi^{p_2}}(\mathcal{X})$ is a ball Banach function space.
From \cite[Theorems~3.51 and~3.53]{tnyyz24},
we infer that the Hardy--Littlewood maximal
operator $\mathcal M$ is bounded on
\begin{equation*}
B^{[\phi^{p_2},(\frac{p}{p_2})',(\frac{p_1}{p_2})']}(\mathcal{X})
=
\left[L_{\frac{p}{p_2},\frac{p_1}{p_2},\phi^{p_2}}(\mathcal{X})\right]'
=
\left(\left[L_{p,p_1,\phi}(\mathcal{X})\right]^\frac{1}{p_2}\right)'.
\end{equation*}
By these and the proof of Theorem~\ref{MS2},
we conclude that the last inequality in \eqref{Gag LM} holds.
This finishes the proof of Theorem~\ref{thm:LMor/R Gag}.
\end{proof}

\begin{remark}
To the best of our knowledge,
Theorem~\ref{thm:LMor/R Gag} is completely new.
\end{remark}

Applying a slight modification of the proof of \cite[Theorem 3.57]{tnyyz24}
with \cite[Theorem 2.11(i)]{tnyyz24} replaced by Theorem~\ref{MS2},
we find that the following conclusion and omit the details.

\begin{theorem}\label{thm:LB Gag}
Let $p,p_1,q,q_1\in(1,\infty)$
satisfy both $\frac{1}{p}+\frac{1}{q}=1$ and
$\frac{1}{p_1}+\frac{1}{q_1}=1$
and let $\phi\in\mathcal{G}^{\rm dec}_{p}$ satisfy $\eqref{NC}$.
Assume that $\phi$ satisfies \eqref{const} or
assume that $\phi$ satisfies both \eqref{r 0} and \eqref{r infty}.
If $q_2\in(0,1]$ and $\beta\in(0,\infty)$,
then there exist $C,\widetilde{C}\in(0,\infty)$ such that,
for any $f\in C_{\mathrm{b}}^\beta(\Omega)
\cap\bigcup_{s\in(0,1)}\dot{W}^{s,q_2}_{B^{[\phi,q,q_1]}}(\Omega)$,
\begin{align*}
C\|f\|_{B^{[\phi,q,q_1]}(\Omega)}
&\leq\varliminf_{s \to 0^+}
s^{\frac{1}{q_2}}\left\|\left\{\int_\Omega
\frac{|f(\cdot)-f(y)|^{q_2}}{U(\cdot,y)[\rho(\cdot,y)]^{sq_2}}
\, d\mu(y) \right\}^{\frac{1}{q_2}}\right\|_{B^{[\phi,q,q_1]}(\Omega)}\\
&\leq\varlimsup_{s \to 0^+}
s^{\frac{1}{q_2}}\left\|\left\{\int_\Omega
\frac{|f(\cdot)-f(y)|^{q_2}}{U(\cdot,y)[\rho(\cdot,y)]^{sq_2}}
\, d\mu(y) \right\}^{\frac{1}{q_2}}\right\|_{B^{[\phi,q,q_1]}(\Omega)}
\leq\widetilde{C}\|f\|_{B^{[\phi,q,q_1]}(\Omega)}.\notag
\end{align*}
\end{theorem}

\begin{remark}
To the best of our knowledge,
Theorem~\ref{thm:LB Gag} is completely new.
\end{remark}

\subsection{Generalized Orlicz--Morrey and
Generalized Orlicz-Block Spaces}

The Orlicz--Morrey spaces were introduced by
\cite{n04} to unify Orlicz and Morrey spaces
and were studied in \cite{n08studia}.
For the definitions of the Orlicz function $\Phi$ and its positive lower
type $r_{\Phi}^-$ and its positive upper type $r_{\Phi}^+$,
see Section~\ref{s6.8}.
In this section, we always assume that $\Phi$ is strictly increasing.
If $1\le r_{\Phi}^-\le r_{\Phi}^+<\infty$,
we always assume that $\Phi$ is convex.
The \emph{locally Orlicz space} $L^{\Phi}_{\mathrm{loc}}(\mathcal{X})$
is defined to be the set of all measurable functions $f$ on $\mathcal{X}$ such that,
for any $x\in \mathcal{X}$, there exists a positive
constant $r_x$ satisfying
$\|f\mathbf{1}_{B(x,r_x)}\|_{L^{\Phi}(\mathcal{X})}<\infty$.
We now introduce the generalized Orlicz--Morrey space as follows.

\begin{definition}\label{defn:OM}
Let $\varphi:\mathcal{X}\times(0,\infty)\to(0,\infty)$ and
$\Phi$ be an Orlicz function with positive lower type
$r_{\Phi}^-$ and positive upper type $r_{\Phi}^+$.
The \emph{generalized Orlicz--Morrey space} $L^{(\Phi,\varphi)}(\mathcal{X})$
is defined to be the set of all $f\in L^{\Phi}_{\mathrm{loc}}(\mathcal{X})$ such that
$
\|f\|_{L^{(\Phi,\varphi)}(\mathcal{X})}
:= \sup_B \|f\|_{\Phi,\varphi,B}<\infty,
$
where the supremum is taken over all balls $B\subset\mathcal{X}$ and
\begin{equation*}
\|f\|_{\Phi,\varphi,B}
:=
\inf\left\{\lambda\in(0,\infty):
\frac 1{\mu(B)}\int_{B}\Phi\left(\frac{|f(x)|}{\lambda}
\right)\,d\mu(x)\le\varphi(B)\right\}.
\end{equation*}
\end{definition}

It is easy to prove that $L^{(\Phi,\varphi)}(\mathcal{X})$ is a quasi-Banach space
because
$
\|\cdot\|_{\Phi,\varphi,B}=
\|\cdot\|_{L^{\Phi}(B,\,\mu_B)},
$
where $\mu_B:=\frac{\mu}{\mu(B)\varphi(B)}$,
which is a quasi-norm on the Orlicz space $L^{\Phi}(B,\,\mu_B)$.
If $1\le r_{\Phi}^-\le r_{\Phi}^+<\infty$, then
$L^{(\Phi,\varphi)}(\mathcal{X})$ is a Banach space equipped
with the norm $\|\cdot\|_{L^{(\Phi,\varphi)}(\mathcal{X})}$.
If $\varphi(B)=[\mu(B)]^{-1}$, then
$L^{(\Phi,\varphi)}(\mathcal{X})=L^{\Phi}(\mathcal{X})$ which is an Orlicz space.
If $\Phi(t)=t^p$ with $p\in(0,\infty)$ and $\varphi=\phi^p$,
then $L^{(\Phi,\varphi)}(\mathcal{X})=L_{p,\phi}(\mathcal{X})$
which is a generalized Morrey space.

In what follows, if there is no confusion,
we always omit the variable in integrals.
For any given ball $B$, from H\"older's
inequality for Orlicz spaces,
we deduce that, for any $f\in L^{\Phi}(B,\,d\mu_B)$
and $g\in L^{\widetilde\Phi}(B,\,d\mu_B)$,
\begin{align}\label{int fg}
\int_B |fg|\,d\mu
&\notag=\mu(B)\varphi(B)\int_B |fg|\,d\mu_B\\
&\le2\mu(B)\varphi(B)\|f\|_{L^{\Phi}(B,\,d\mu_B)}
\|g\|_{L^{\widetilde\Phi}(B,\,d\mu_B)}\notag\\
&=2\mu(B)\varphi(B)\|f\|_{\Phi,\varphi,B}
\|g\|_{\widetilde\Phi,\varphi,B}.
\end{align}

Next, we define the generalized Orlicz-block space.

\begin{definition}\label{defn:pP block}
Let $\varphi:\mathcal{X}\times(0,\infty)\to(0,\infty)$ and
$\Psi$ be an Orlicz function with $1<r_{\Psi}^-\le r_{\Psi}^+<\infty$.
A function $b$ on $\mathcal{X}$ is called a \emph{$(\varphi,\Psi)$-block}
if there exists a ball $B$, called the \emph{supporting ball} of $b$,
such that
\begin{enumerate}
\item[\rm (i)]
$\mathrm{supp\,}b\subset{B}$;
\item[\rm (ii)]
$b\in L^{\Psi}(\mathcal{X})$ and
$\|b\|_{\Psi,\varphi,B} \le \frac{1}{\mu(B)\varphi(B)}$.
\end{enumerate}
Denote by $B(\varphi,\Psi)$ the set of all $(\varphi,\Psi)$-blocks.
\end{definition}

By \eqref{int fg} and Definitions~\ref{defn:OM} and
\ref{defn:pP block}, we find that,
for any $b\in B(\varphi,\Psi)$ with its supporting ball $B$,
\begin{equation}\label{bg}
\left|\int_\mathcal{X} bg\,d\mu\right|
\le
\int_B |bg|\,d\mu
\le
2\mu(B)\varphi(B)\|b\|_{\Psi,\varphi,B}\|g\|_{\widetilde\Psi,\varphi,B}
\le
2\|g\|_{L^{(\widetilde\Psi,\varphi)}(\mathcal{X})},
\end{equation}
where $\widetilde\Psi$ is the complementary Orlicz function of $\Psi$.
That is,
the mapping $g\mapsto\int_{\mathcal{X}} bg\,d\mu$
is a bounded linear functional on $L^{(\widetilde\Psi,\varphi)}(\mathcal{X})$
with norm not exceeding $2$.

\begin{definition}\label{defn:OB}
Let $\varphi:\mathcal{X}\times(0,\infty)\to(0,\infty)$ and
$\Psi$ be an Orlicz function with $1<r_{\Psi}^-\le r_{\Psi}^+<\infty$.
Assume that $L^{(\widetilde\Psi,\varphi)}(\mathcal{X})\not=\{0\}$.
The \emph{Orlicz-block space}
$B_{(\varphi,\Psi)}(\mathcal{X})$ is defined to be the set
of all
$f\in[L^{(\widetilde\Psi,\varphi)}(\mathcal{X})]^*$
such that
there exist a sequence $\{b_j\}_{j\in\mathbb{N}}$ in $B(\varphi,\Psi)$
and a sequence $\{\lambda_j\}_{j\in\mathbb{N}}$ in $[0,\infty)$
satisfying $\sum_{j\in\mathbb{N}}\lambda_j<\infty$
such that
\begin{equation}\label{OB expression}
f=\sum_{j\in\mathbb{N}}\lambda_j b_j\ \text{in}\
\left[L^{(\widetilde\Psi,\varphi)}(\mathcal{X})\right]^*.
\end{equation}
For any $f\in B_{(\varphi,\Psi)}(\mathcal{X})$, define
\begin{equation*}
\|f\|_{B_{(\varphi,\Psi)}(\mathcal{X})}
:=\inf\left\{\sum_{j\in\mathbb{N}}\lambda_j:
f=\sum_{j\in\mathbb{N}}\lambda_jb_j\ \text{in}\
\left[L^{(\widetilde\Psi,\varphi)}(\mathcal{X})\right]^*\right\},
\end{equation*}
where the infimum is taken over all decompositions
of $f$ as in \eqref{OB expression}.
\end{definition}

It is easy to show that $B_{(\varphi,\Psi)}(\mathcal{X})$ is a Banach space equipped
with the norm $\|\cdot\|_{B_{(\varphi,\Psi)}(\mathcal{X})}$.
The Orlicz-block space $B_{(\varphi,\Psi)}(\mathcal{X})$
was introduced in \cite{n11},
where it was proved to be the predual space
of $L^{(\widetilde\Psi,\varphi)}(\mathcal{X})$
in the case where
both $\mathcal{X} = \mathbb{R}^d$ and $\varphi$ is a function
from $(0,\infty)$ to $(0,\infty)$.
If $\Psi(t)=t^q$, $q\in(1,\infty)$, and $\varphi=\phi^{q'}$,
then $B_{(\varphi,\Psi)}(\mathcal{X})=B^{[\phi,q]}(\mathcal{X})$
which is a generalized block space.

Now, we first show that $L^{(\Phi,\varphi)}(\mathcal{X})$
is a ball quasi-Banach space.
The following lemma was proven by \cite[Lemma~4.1]{san21} in the case
where
$\mathcal{X}=\mathbb{R}^n$,
$r_{\Phi}^-\ge1$, and $\varphi:(0,\infty)\to(0,\infty)$.

\begin{lemma}\label{lem:chi norm}
Let $\varphi\in\mathcal{G}^{\rm dec}_1$ satisfy \eqref{NC} and
$\Phi$ be an Orlicz function with positive lower type
$r_{\Phi}^-$ and positive upper type $r_{\Phi}^+$.
Then there exists a positive constant $C$ such that,
for any ball $B$, $\mathbf{1}_{B}\in L^{(\Phi,\varphi)}(\mathcal{X})$ and
\begin{align}\label{chi norm}
\frac{1}{\Phi^{-1}(\varphi(B))}
\le
\left\|\mathbf{1}_B\right\|_{L^{(\Phi,\varphi)}(\mathcal{X})}
\le
\frac{C}{\Phi^{-1}(\varphi(B))}.
\end{align}
\end{lemma}

\begin{proof}
Let $B$ be a ball.
An elementary calculation shows that
\begin{align}\label{chiB}
\|\mathbf{1}_{B}\|_{\Phi,\varphi,B}
=\frac{1}{\Phi^{-1}(\varphi(B))}.
\end{align}
From this and Definition~\ref{defn:OM},
we infer that the first inequality in \eqref{chi norm} holds.
Next, we prove the second inequality in \eqref{chi norm}.
Let $\lambda:=\frac{1}{\Phi^{-1}(\varphi(B))}$.
Then it is enough to show that there exists a constant
$C\in[1,\infty)$ such that, for any ball $B'$
with $B\cap B'\ne\emptyset$,
\begin{equation}\label{C lambda}
\frac{1}{\mu(B')\varphi(B')}\int_{B'}\Phi
\left(\frac{\mathbf{1}_{B}}{C\lambda}\right)\,d\mu
\le 1.
\end{equation}
Since $B\cap B'\ne\emptyset$, we can take a positive
constant $k\in[1,\infty)$, independent of $B$ and $B'$,
such that $B'\subset kB$ or $B\subset kB'$ holds.
We consider the following two cases.

\emph{Case (1)}
$B'\subset kB$. In this case,
by the fact that $\Phi(t)\ge\Phi(0)$ for any $t\in[0,\infty)$,
the definition of $\lambda$, and the
properties of $\varphi$, we conclude that
\begin{align*}
\frac{1}{\mu(B')\varphi(B')}\int_{B'}\Phi
\left(\frac{\mathbf{1}_{B}}{\lambda}\right)\,d\mu
\le\frac{1}{\varphi(B')}\Phi\left(\frac{1}{\lambda}\right)
=\frac{\varphi(B)}{\varphi(B')}
\lesssim 1.
\end{align*}

\emph{Case (2)}
$B\subset kB'$.
In this case, from the fact that $\Phi(t)\ge\Phi(0)$
for any $t\in[0,\infty)$,
the definition of $\lambda$, and the
properties of $\varphi$, it follows that
\begin{align*}
\frac{1}{\mu(B')\varphi(B')}\int_{B'}\Phi
\left(\frac{\mathbf{1}_{B}}{\lambda}\right)\,d\mu
\le
\frac{1}{\mu(B')\varphi(B')}\int_{B}\Phi
\left(\frac{1}{\lambda}\right)\,d\mu
=
\frac{\mu(B)\varphi(B)}{\mu(B')\varphi(B')}
\lesssim 1.
\end{align*}
By the both cases, we find that \eqref{C lambda}
and hence \eqref{chi norm} holds, which then completes
the proof of Lemma~\ref{lem:chi norm}.
\end{proof}

Using Definition~\ref{defn:OM} and Lemma~\ref{lem:chi norm},
we immediately obtain the following conclusion.

\begin{theorem}\label{thm:OM bBfs}
Let $\varphi\in\mathcal{G}^{\rm dec}_1$ satisfy \eqref{NC} and
$\Phi$ be an Orlicz function with positive lower type
$r_{\Phi}^-$ and positive upper type $r_{\Phi}^+$.
Then $L^{(\Phi,\varphi)}(\mathcal{X})$ is a ball quasi-Banach function space.
Moreover, if $1\le r_{\Phi}^-\le r_{\Phi}^+<\infty$,
then $L^{(\Phi,\varphi)}(\mathcal{X})$ is a ball Banach function space.
\end{theorem}

Now, we prove that $B_{(\varphi,\Psi)}(\mathcal{X})$ is
a ball quasi-Banach function space.
We show that, in some
special cases, the generalized Orlicz-block
space reduces to the Orlicz space.

\begin{lemma}\label{lem:LPs}
Let $\varphi:\mathcal{X}\times(0,\infty)\to(0,\infty)$ and
$\Psi$ be an Orlicz function with $1<r_{\Psi}^-\le r_{\Psi}^+<\infty$.
If there exists a positive constant $C_*$ such that, for any ball $B$,
\begin{equation}\label{var const}
\frac 1{C_*}\le\varphi(B)\mu(B)\le C_*,
\end{equation}
Then $B_{(\varphi,\Psi)}(\mathcal{X})=L^{\Psi}(\mathcal{X})$ and
there exists a positive constant $C$ such that,
for any $f\in B_{(\varphi,\Psi)}(\mathcal{X})$,
\begin{equation*}
C^{-1}\|f\|_{L^{\Psi}(\mathcal{X})}\le
\|f\|_{B_{(\varphi,\Psi)}(\mathcal{X})}\le C\|f\|_{L^{\Psi}(\mathcal{X})}.
\end{equation*}
\end{lemma}

\begin{proof}
From the assumption that $\varphi$ satisfies \eqref{var const}, we deduce that
$L^{(\widetilde\Psi,\varphi)}(\mathcal{X})=L^{\widetilde\Psi}(\mathcal{X})$
with equivalent norms.
Let $f\in B_{(\varphi,\Psi)}(\mathcal{X})$.
Then there exist a sequence $\{b_j\}_{j\in\mathbb{N}}$ in $B(\varphi,\Psi)$
and a sequence $\{\lambda_j\}_{j\in\mathbb{N}}$ in $[0,\infty)$
satisfying $\sum_{j\in\mathbb{N}}\lambda_j<\infty$
such that $f=\sum_{j\in\mathbb{N}}\lambda_jb_j$
in $[L^{(\widetilde\Psi,\varphi)}(\mathcal{X})]^*=[L^{\widetilde\Psi}(\mathcal{X})]^*
=L^{\Psi}(\mathcal{X})$ and $2\|f\|_{B_{(\varphi,\Psi)}(\mathcal{X})}\ge
\sum_{j\in\mathbb{N}}\lambda_j$.
By this and \eqref{bg}, we conclude that,
for any $g\in L^{\widetilde\Psi}(\mathcal{X})$,
\begin{equation*}
\left|\int_\mathcal{X} gf\,d\mu\right|
\le
\sum_{j=1}^\infty\lambda_j\int_{B_j} |b_jg|\,d\mu
\le
2\sum_{j=1}^\infty\lambda_j \|g\|_{L^{(\widetilde\Psi,\varphi)}(\mathcal{X})}
\sim
\|f\|_{B_{(\varphi,\Psi)}(\mathcal{X})}\|g\|_{L^{\widetilde\Psi}(\mathcal{X})},
\end{equation*}
which proves $f\in L^{\Psi}(\mathcal{X})$ and
$\|f\|_{L^{\Psi}}\lesssim\|f\|_{B_{(\varphi,\Psi)}(\mathcal{X})}$.

Conversely, let $f\in L^{\Psi}(\mathcal{X})$ and
$\lambda\in(\|f\|_{L^{\Psi}(\mathcal{X})},\infty)$. Then
\begin{align}\label{840}
\int_\mathcal{X}\Psi\left(\frac{|f|}{\lambda}\right)\,d\mu \le 1.
\end{align}
Choose
a sequence of balls $\{B_j\}_{j\in\mathbb{N}}$ such that
$B_j\subset B_{j+1}$ for any $j\in\mathbb{N}$,
$\bigcup_{j\in\mathbb{N}}B_j=\mathcal{X}$, and
\begin{equation*}
\int_{B_j^{\complement}} \Psi\left(\frac{|f|}{\lambda}\right)\,d\mu
\le
\left(2^{-j}\right)^{r_{\Psi}^+}
\end{equation*}
for any $j\in\mathbb{N}$,
which, together with \eqref{840} and \eqref{var const},
further implies that
\begin{align*}
\int_{B_1} \Psi\left(\frac{|f|}{\lambda}\right)\,d\mu
\le
1\le C_*\mu(B_1)\varphi(B_1)
\end{align*}
and, for any $j\in\mathbb{N}\cap[2,\infty)$,
\begin{align*}
\int_{B_{j-1}^{\complement}} \Psi\left(\frac{|f|}{\lambda}\right)\,d\mu
\le
\left(2^{-j+1}\right)^{r_{\Psi}^+} C_*\mu(B_j)\varphi(B_j).
\end{align*}
From this, we infer that there exists $C\in(0,\infty)$,
depending only on $\Psi$ and $C_*$, such that
\begin{align}\label{845}
\left\|f\mathbf{1}_{B_1}\right\|_{\Psi,\varphi,B_1}
\le C\|f\|_{L^{\Psi}(\mathcal{X})}
\end{align}
and, for any $j\in\mathbb{N}\cap[2,\infty)$,
\begin{align}\label{846}
\left\|f\mathbf{1}_{B_j\setminus B_{j-1}}\right\|_{\Psi,\varphi,B_j}
\le 2^{-j+1}C\|f\|_{L^{\Psi}(\mathcal{X})}.
\end{align}
Let $\lambda_1:=\varphi(B_1)\mu(B_1)
\|f\mathbf{1}_{B_1}\|_{\Psi,\varphi,B_1}$
and $b_1:=\lambda_1^{-1}f\mathbf{1}_{B_1}$
and, for any $j\in\mathbb{N}\cap[2,\infty)$,
let
\begin{align*}
\lambda_j:=\varphi(B_j)\mu(B_j)
\left\|f\mathbf{1}_{B_j\setminus B_{j-1}}\right\|_{\Psi,\varphi,B_j}
\end{align*}
and
$
b_j:=\lambda_j^{-1}f\mathbf{1}_{B_j\setminus B_{j-1}}.
$
Then, by \eqref{845}, \eqref{846}, \eqref{var const},
and Definition~\ref{defn:pP block}, we find that,
for any $j\in\mathbb{N}$,
$
\lambda_j\le 2^{-j+1}CC_*\|f\|_{L^{\Psi}(\mathcal{X})}
$
and $b_j\in B(\varphi,\Psi)$,
which implies that
\begin{align*}
f=
f\mathbf{1}_{B_1} + \sum_{j=2}^\infty
f\mathbf{1}_{B_j\setminus B_{j-1}}
=\sum_{j=1}^\infty\lambda_jb_j
\end{align*}
and
\begin{equation*}
\sum_{j=1}^\infty\lambda_j\le
CC_*\sum_{j=1}^\infty2^{-j+1}\|f\|_{L^{\Psi}(\mathcal{X})}
\le
2CC_*\|f\|_{L^{\Psi}(\mathcal{X})}.
\end{equation*}
That is,  $f\in B_{(\varphi,\Psi)}(\mathcal{X})$ and
$\|f\|_{B_{(\varphi,\Psi)}(\mathcal{X})}\le 2CC_*\|f\|_{L^{\Psi}(\mathcal{X})}$.
This finishes the proof of Lemma~\ref{lem:LPs}.
\end{proof}

In the proof of the next lemma, we need the following inequality:
Let $\varphi\in\mathcal{G}^{\rm dec}_1$ satisfy \eqref{NC}.
Then there exists $C\in[1,\infty)$ such that, for any balls $B_1$ and $B_2$,
\begin{equation}\label{B1B2}
\frac{\mu(B_1\cap B_2)}{\mu(B_2)\varphi(B_2)}
\le
\frac{C}{\varphi(B_1)};
\end{equation}
see \cite[Lemma~3.18]{tnyyz24}.

\begin{lemma}\label{lem:OB L1 loc}
Let $\varphi\in\mathcal{G}^{\rm dec}_1$ satisfy \eqref{NC} and
$\Psi$ be an Orlicz function with $1<r_{\Psi}^-\le r_{\Psi}^+<\infty$.
Then $B_{(\varphi,\Psi)}(\mathcal{X})\subset L^1_{\mathrm{loc}}(\mathcal{X})$
and the expression \eqref{OB expression} converges $\mu$-a.e. on $\mathcal{X}$.
\end{lemma}

\begin{proof}
Let $f\in B_{(\varphi,\Psi)}$.
Then there exist a sequence $\{b_j\}_{j\in\mathbb{N}}$ in $B(\varphi,\Psi)$
and a sequence $\{\lambda_j\}_{j\in\mathbb{N}}$ in $[0,\infty)$
satisfying $\sum_{j\in\mathbb{N}}\lambda_j<\infty$
such that $f=\sum_{j\in\mathbb{N}}\lambda_jb_j$
in $[L^{(\widetilde\Psi,\varphi)}(\mathcal{X})]^*$.
For any $j\in\mathbb{N}$,
let $B_j$ be the supporting ball of $b_j$.
An elementary calculation shows that, for any ball $B$,
\begin{equation*}
\left\|\mathbf{1}_{B\cap B_j}\right\|_{\widetilde\Psi,\varphi,B_j}
=\frac{1}{\widetilde\Psi^{-1}(\frac{\mu(B_j)\varphi(B_j)}{\mu(B\cap B_j)})}.
\end{equation*}
From this, \eqref{int fg}, and \eqref{B1B2}, we deduce that,
for any ball $B$,
\begin{align*}
\int_B\sum_{j=1}^\infty\lambda_j|b_j|\,d\mu
&\le\sum_{j=1}^\infty\lambda_j
\int_{B_j}|b_j|\mathbf{1}_{B\cap B_j}\,d\mu
\le
\sum_{j=1}^\infty\lambda_j\, \mu(B_j)\varphi(B_j)
\left\|b_j\right\|_{\Psi,\varphi,B_j}
\left\|\mathbf{1}_{B\cap B_j}\right\|_{\widetilde\Psi,\varphi,B_j}  \\
&\le\sum_{j=1}^\infty\lambda_j
\frac{1}{\widetilde\Psi^{-1}(\frac{\mu(B_j)\varphi(B_j)}{\mu(B\cap B_j)})}
\lesssim\sum_{j=1}^\infty\lambda_j
\frac{1}{\widetilde\Psi^{-1}(\varphi(B))}<\infty.
\end{align*}
Thus, $\sum_{j\in\mathbb{N}}\lambda_j b_j$ converges $\mu$-a.e. to
some function in $L^1_{\mathrm{loc}}(\mathcal{X})$
and hence $f\in L^1_{\mathrm{loc}}(\mathcal{X})$.
This finishes the proof of Lemma~\ref{lem:OB L1 loc}.
\end{proof}

By an argument similar to that used in the proof of
\cite[Lemmas~3.22 and~3.25]{tnyyz24},
we can prove the following two lemmas; we omit the details.

\begin{lemma}\label{lem:OB lattice}
Let $\varphi\in\mathcal{G}^{\rm dec}_1$ satisfy \eqref{NC} and
$\Psi$ be an Orlicz function with $1<r_{\Psi}^-\le r_{\Psi}^+<\infty$.
If $g,f\in\mathscr{M}(\mathcal{X})$ are such that
$|g|\le|f|$ $\mu$-a.e. on $\mathcal{X}$,
then $\|g\|_{B_{(\varphi,\Psi)}(\mathcal{X})}
\le\|f\|_{B_{(\varphi,\Psi)}(\mathcal{X})}$.
\end{lemma}

\begin{lemma}\label{lem:OB Fatou}
Let $\varphi\in\mathcal{G}^{\rm dec}_1$ satisfy \eqref{NC} and
$\Psi$ be an Orlicz function with $1<r_{\Psi}^-\le r_{\Psi}^+<\infty$.
Assume that $\varphi$ satisfies \eqref{var const}
or both
\begin{equation}\label{var r 0}
\lim_{r\to0^+}\inf_{x\in B}\varphi(x,r)=\infty
\end{equation}
for any ball $B\subset\mathcal{X}$ and
\begin{equation}\label{var r infty}
\lim_{r\to\infty}\varphi(x,r)\mu(B(x,r))=\infty
\end{equation}
for any $x\in \mathcal{X}$.
Let $\{f_m\}_{m\in\mathbb{N}}$ be a sequence in
$\mathscr{M}(\mathcal{X})$ and $f\in\mathscr{M}(\mathcal{X})$.
If $0\le f_m \uparrow f$ $\mu$-almost everywhere on $\mathcal{X}$ as $m\to\infty$,
then $\|f_m\|_{B_{(\varphi,\Psi)}(\mathcal{X})}\uparrow
\|f\|_{B_{(\varphi,\Psi)}(\mathcal{X})}$ as $m\to\infty$.
\end{lemma}

Using Definition~\ref{defn:OB}, Lemmas~\ref{lem:OB lattice} and \ref{lem:OB Fatou},
and the properties of the Lebesgue integral,
we immediately obtain the following conclusion; we omit the details.

\begin{theorem}\label{thm:OB bBfs}
Let $\varphi\in\mathcal{G}^{\rm dec}_1$ satisfy \eqref{NC} and
$\Psi$ be an Orlicz function with $1<r_{\Psi}^-\le r_{\Psi}^+<\infty$.
Assume that $\varphi$ satisfies \eqref{var const}
or both \eqref{var r 0} and \eqref{var r infty}.
Then $B_{(\varphi,\Psi)}(\mathcal{X})$ is a ball Banach function space.
\end{theorem}

By an argument similar to that used in
the proof of \cite[Theorem~3.28]{tnyyz24},
we can also show that the generalized Orlicz--Morrey space and the generalized Orlicz-block space
are associate spaces of each other;
we omit the details.

\begin{theorem}\label{thm:K dual}
Let $\varphi\in\mathcal{G}^{\rm dec}_1$ satisfy \eqref{NC} and
$(\Phi, \Psi)$ be a complementary pair of Orlicz
functions with $1<r_{\Phi}^-\le r_{\Phi}^+<\infty$.
Then $[B_{(\varphi,\Psi)}(\mathcal{X})]'= L^{(\Phi,\varphi)}(\mathcal{X})$
and, for any $g\in[B_{(\varphi,\Psi)}(\mathcal{X})]'$,
it holds $\|g\|_{[B_{(\varphi,\Psi)}(\mathcal{X})]'}
\sim\|g\|_{L^{(\Phi,\varphi)}(\mathcal{X})}$
with the positive equivalence constants independent of $g$.
In addition, assume that \eqref{var const} holds or
assume that both \eqref{var r 0} and \eqref{var r infty} hold.
Then $[L^{(\Phi,\varphi)}(\mathcal{X})]' = B_{(\varphi,\Psi)}(\mathcal{X})$
and, for any $f\in[L^{(\Phi,\varphi)}(\mathcal{X})]'$,
$\|f\|_{[L^{(\Phi,\varphi)}(\mathcal{X})]'} \sim \|f\|_{B_{(\varphi,\Psi)}(\mathcal{X})}$
with the positive equivalence constants independent of $f$.
\end{theorem}

Next, we prove the boundedness of the Hardy--Littlewood
maximal operator $\mathcal{M}$
on generalized Orlicz--Morrey spaces and generalized Orlicz-block spaces.
It is known that, if $1<r_{\Phi}^-\le r_{\Phi}^+<\infty$,
then there exists a positive constant $c$ such that,
for any $f\in L^1_{\mathrm{loc}}(\mathcal{X})$,
the modular inequality
\begin{equation}\label{modular}
\int_{\mathcal{X}}\Phi(\mathcal{M}f)\,d\mu
\le
c\int_{\mathcal{X}}\Phi(c|f|)\,d\mu
\end{equation}
holds, where the positive constant $c$ is independent of $f$;
see Kokilashvili and Krbec~\cite[Theorem~1.2.1]{kk91}
whose proof is available to spaces of homogeneous type.

\begin{theorem}\label{thm:Mf OM}
Let $\varphi\in\mathcal{G}^{\rm dec}_1$ satisfy \eqref{NC}
and $\Phi$ be an Orlicz function with $1<r_{\Phi}^-\le r_{\Phi}^+<\infty$.
Then the Hardy--Littlewood maximal operator
$\mathcal{M}$ is bounded on $L^{(\Phi,\varphi)}(\mathcal{X})$.
\end{theorem}

Theorem~\ref{thm:Mf OM} with $\mathcal{X}=\mathbb{R}^n$
and $\varphi:(0,\infty)\to(0,\infty)$
was shown in \cite{n04,n08studia}.
To prove Theorem~\ref{thm:Mf OM}, we need the following lemma,
which is exactly \cite[Lemma~4.3]{san21}.

\begin{lemma}\label{lem:fintB}
Let $\varphi:\mathcal{X}\times(0,\infty)\to(0,\infty)$
and $\Phi$ be an Orlicz function with
$1<r_{\Phi}^-\le r_{\Phi}^+<\infty$.
Then, for any $f\in L^{(\Phi,\varphi)}(\mathcal{X})$ and any ball $B$,
\begin{equation*}
\frac{1}{\mu(B)}\int_B|f|\,d\mu
\le
2\Phi^{-1}(\varphi(B))\,\|f\|_{\Phi,\varphi,B}.
\end{equation*}
\end{lemma}

\begin{proof}[Proof of Theorem~\ref{thm:Mf OM}]
Let $f\in L^{(\Phi,\varphi)}(\mathcal{X})$
and $\|f\|_{L^{(\Phi,\varphi)}(\mathcal{X})}=1$.
It suffices to prove that, for any ball $B(x,r)$ with $x\in\mathcal{X}$
and $r\in(0,\infty)$,
\begin{equation}\label{MfP}
\frac 1{\mu(B(x,r))}
\int_{B(x,r)}\Phi(\mathcal{M}f)\,d\mu
\lesssim\varphi(x,r).
\end{equation}
Let $k:=K_0(1+2K_0)$ and write $f=f_1+f_2$
with $f_1:=f\mathbf{1}_{B(x,kr)}$
and $f_2:=f\mathbf{1}_{[B(x,kr)]^\complement}$.
By \eqref{modular}, we find that
\begin{align}\label{1030}
\int_{B(x,r)} \Phi({\mathcal M}f_1)\,d\mu
&\nonumber\leq
c\int_{\mathcal{X}} \Phi(c|f_1|)\,d\mu
=c\int_{B(x,kr)} \Phi(c|f|)\,d\mu\\
&\le c^2
\mu(B(x,kr))\varphi(x,kr)
\sim
\mu(B(x,r))\varphi(x,r).
\end{align}
Now, notice that, for any given $z\in B(x,r)$,
\begin{align}\label{1022}
{\mathcal M}f_2(z)
=\sup_{B(y,s)\ni z}\fint_{B(y,s)} |f_2|\,d\mu
=\sup_{\genfrac{}{}{0pt}{}{B(y,s)\ni z}
{B(y,s)\cap[B(x,kr)]^{\complement}\ne\emptyset}}
\fint_{B(y,s)} |f|\,d\mu.
\end{align}
If the ball $B(y,s)$ with $y\in\mathcal{X}$
and $s\in(0,\infty)$ satisfies both $B(y,s)\ni z$
and $B(y,s)\cap [B(x,kr)]^{\complement}\ne\emptyset$,
then $r\le s$.
Actually,
if $r>s$, then,
by Definition~\ref{Deqms}(iii), we conclude that,
for any $w\in B(y,s)$,
\begin{align*}
\rho(x,w)
\le
K_0[\rho(x,z)+\rho(z,w)]
\le
K_0\rho(x,z)+K_0^2[\rho(z,y)+\rho(y,w)]
\le
K_0(r+2K_0s)<kr,
\end{align*}
that is, $w\in B(x,kr)$ and hence
$B(y,s)\subset B(x,kr)$, which contradicts
$B(y,s)\cap[B(x,kr)]^\complement\neq\emptyset$.
Therefore, from \eqref{1022} and
Lemma~\ref{lem:fintB}, it follows that, for any $z\in B(x,r)$,
\begin{align*}
{\mathcal M}f_2(z)
&\le\sup_{\genfrac{}{}{0pt}{}{B(y,s)\ni z}{r\le s}}
\fint_{B(y,s)}|f|\,d\mu
\lesssim
\sup_{\genfrac{}{}{0pt}{}{B(y,s)\ni z}{r\le s}}
\Phi^{-1}(\varphi(y,s)).
\end{align*}
Notice that, if $B(x,r)\cap B(y,s)\ne\emptyset$ and $r\le s$,
then $\varphi(y,s)\sim\varphi(x,s)\lesssim\varphi(x,r)$,
which further implies that
$
\Phi({\mathcal M}f_2(z))\lesssim\varphi(x,r).
$
By this and \eqref{1030},
we find that \eqref{MfP} holds.
This finishes the proof of Lemma~\ref{lem:fintB}.
\end{proof}

We turn to consider the generalized Orlicz-block space.

\begin{theorem}\label{thm:Mf OB}
Let $\varphi\in\mathcal{G}^{\rm dec}_1$ satisfy \eqref{NC}
and $\Psi$ be an Orlicz function with
$1<r_{\Psi}^-\le r_{\Psi}^+<\infty$.
Assume that there exists a positive constant $C$ such that
\begin{equation}\label{int vp}
\int_{r}^{\infty}\frac{\varphi(x,t)}{t}\,dt
\le C\varphi(x,r).
\end{equation}
Then the Hardy--Littlewood maximal operator
$\mathcal{M}$ is bounded on $B_{(\varphi,\Psi)}(\mathcal{X})$.
\end{theorem}

\begin{proof}
We first claim that there exists a positive
constant $C_0$ such that,
for any $b\in B(\varphi,\Psi)$,
\begin{equation}\label{Mb}
\|{\mathcal M}b\|_{B_{(\varphi,\Psi)}(\mathcal{X})}\le C_0.
\end{equation}
Indeed,
let $b\in B(\varphi,\Psi)$ and $B(x,r)$ be its supporting ball,
where $x\in \mathcal{X}$ and $r\in(0,\infty)$.
Let $c$ be the same positive constant as in \eqref{modular}
and let $k:=K_0(1+2K_0)$.
From \eqref{modular}, we infer that,
for any $\lambda\in(c\|b\|_{\Psi,\varphi,B(x,r)},\infty)$,
\begin{align*}
&\frac1{\mu(B(x,kr))\varphi(x,kr)}
\int_{B(x,kr)} \Psi\left(\frac{({\mathcal M}b)
\mathbf{1}_{B(x,kr)}}{\lambda}\right)\,d\mu\\
&\quad\le
\frac1{\mu(B(x,kr))\varphi(x,kr)}
\int_{\mathcal{X}} \Psi\left(\frac{{\mathcal M}b}{\lambda}\right)\,d\mu \\
&\quad\le
\frac{c}{\mu(B(x,kr))\varphi(x,kr)}
\int_{\mathcal{X}} \Psi\left(\frac{c|b|}{\lambda}\right)\,d\mu
\le
\frac{c\mu(B(x,r))\varphi(x,r)}{\mu(B(x,kr))\varphi(x,kr)}
\le
C_1,
\end{align*}
which, combined with Definition~\ref{defn:pP block}(ii)
and the assumption on $\varphi$,
further implies that
\begin{equation}\label{Mb chi}
\|\mathcal{M}b\|_{\Psi,\varphi,B(x,kr)}
\le
C_2\|b\|_{\Psi,\varphi,B(x,r)}
\le
\frac{C_2}{\mu(B(x,r))\varphi(x,r)}
\le
\frac{C_3}{\mu(B(x,kr))\varphi(x,kr)},
\end{equation}
where the positive constants $C_1$, $C_2$, and $C_3$
depend only on $c$, $K_0$, $L_{(\mu)}$, $\Psi$, and $\varphi$.

Next,
let $j\in{\mathbb Z}_+$.
For any given $z\notin B(x,2^jkr)$, we have
\begin{equation*}
{\mathcal M}b(z)
=\sup_{B(y,s)\ni z}\fint_{B(y,s)}|b|\,d\mu
=\sup_{B(y,s)\ni z,\,B(x,r)\cap B(y,s)\ne\emptyset}
\fint_{B(y,s)}|b|\,d\mu.
\end{equation*}
If $z\notin B(x,2^jkr)$, $z\in B(y,s)$,
and $B(x,r)\cap B(y,s)\ne\emptyset$, then $2^jr<s$.
Indeed, if $s\le2^jr$, then
Definition~\ref{Deqms}(iii) implies that,
for any $w\in B(x,r)\cap B(y,s)$,
\begin{align}\label{2jr<s}
\rho(x,z)
&\notag\le
K_0[\rho(x,y)+\rho(y,z)]\\
&\le
K_0\left\{K_0[\rho(x,w)+\rho(w,y)]+\rho(y,z)\right\}\notag\\
&<
K_0[K_0(r+s)+s]
\le
K_0(2K_0+1)2^jr
=
2^jkr,
\end{align}
which contradicts $z\notin B(x,2^jkr)$.
Thus, $2^jr<s$.
By this, $\mathrm{supp\,}b\subset B(x,r)$,
\eqref{dc}, and \eqref{int fg}, we conclude that,
for any $z\notin B(x,2^jkr)$,
\begin{align*}
{\mathcal M}b(z)
&=\sup_{\genfrac{}{}{0pt}{}{B(y,s)\ni z}{s>2^jr}}
\frac{1}{\mu(B(y,s))}\int_{B(y,s)}|b|\,d\mu \\
&\le\sup_{\genfrac{}{}{0pt}{}{B(y,s)\ni z}{s>2^jr}}
\frac{1}{\mu(B(y,s))}\int_{B(x,r)}|b|\,d\mu \le
\frac{C_4}{\mu(B(x,2^jr))}\int_{B(x,r)}|b|\,d\mu \\
&\le
\frac{2C_4\mu(B(x,r))\varphi(x,r) \|b\|_{\Psi,\varphi,B(x,r)}
\|1\|_{\widetilde\Psi, \varphi,B(x,r)}}
{\mu(B(x,2^jr))} \\
&\le
\frac{2C_4\|1\|_{\widetilde\Psi, \varphi,B(x,r)}}
{\mu(B(x,2^jr))}
=
\frac{2C_4}{\mu(B(x,2^jr))\widetilde\Psi^{-1}(\varphi(x,r))},
\end{align*}
where $C_4$ is a positive constant depending only on $K_0$ and $L_{(\mu)}$,
which,
together with \eqref{chiB} and the relations
$\Psi^{-1}(t)\widetilde\Psi^{-1}(t)\sim t$ and
$\mu(B(x,2^{j-1}r))\sim\mu(B(x,2^jr))$,
further implies that, for any $j\in\mathbb{N}$,
\begin{align}\label{Mb chi 2}
&\notag\left\|({\mathcal M}b)\mathbf{1}_{B(x,2^{j}kr)\setminus B(x,2^{j-1}kr)}\right\|
_{\Psi,\varphi, B(x,2^jkr)}\\
&\quad\le
\frac{2C_4}{\mu(B(x,2^{j-1}r))\widetilde\Psi^{-1}(\varphi(x,r))}
\|1\|_{\Psi,\varphi, B(x,2^jkr)}\nonumber\\
&\quad=
\frac{2C_4}{\mu(B(x,2^{j-1}r))\widetilde\Psi^{-1}(\varphi(x,r))
\Psi^{-1}(\varphi(x,2^jkr))} \notag \\
&\quad\le
\frac{C_5\widetilde\Psi^{-1}(\varphi(x,2^jkr))}
{\mu(B(x,2^jr))\varphi(x,2^jkr)\widetilde\Psi^{-1}(\varphi(x,r))},
\end{align}
where the positive constant $C_5$ depends only on $K_0$, $L_{(\mu)}$, and $\Psi$.

Now, let $\widetilde\lambda_0:=C_3$
and, for any $j\in\mathbb{N}$,
$
\widetilde\lambda_j
:=
\frac{C_5\widetilde\Psi^{-1}(\varphi(x,2^jkr))}{\widetilde\Psi^{-1}(\varphi(x,r))}.
$
Let
$
\widetilde{b}_0:=
\frac{({\mathcal M}b)\mathbf{1}_{B(x,kr)}}{\widetilde\lambda_0}
$
and, for any $j\in\mathbb{N}$,
$
\widetilde{b}_j:=
\frac{({\mathcal M}b)\mathbf{1}_{B(x,2^{j}kr)\setminus B(x,2^{j-1}kr)}}
{\widetilde\lambda_j}.
$
Then, from \eqref{Mb chi} and \eqref{Mb chi 2}, we deduce that
$\widetilde{b}_j\in B(\varphi,\Psi)$ for any $j\in\mathbb{N}\cup\{0\}$,
$
{\mathcal M}b=
\sum_{j=0}^\infty\widetilde\lambda_j \widetilde{b}_j,
$
and
\begin{align}\label{1531}
\left\|{\mathcal M}b\right\|_{B_{(\varphi,\Psi)}}
&\notag\le
\widetilde\lambda_0+\sum_{j=1}^\infty\widetilde\lambda_j
\lesssim
1 + \sum_{j=1}^\infty \frac{\widetilde\Psi^{-1}
(\varphi(x,2^jkr))}{\widetilde\Psi^{-1}(\varphi(x,r))} \\
&\sim
1 +\sum_{j=1}^\infty\frac{1}{\widetilde\Psi^{-1}(\varphi(x,r))}
\int_{2^{j-1}kr}^{2^{j}kr}
\frac{\widetilde\Psi^{-1}(\varphi(x,t))}{t}\,dt\nonumber\\
&\le
1 + \frac{1}{\widetilde\Psi^{-1}(\varphi(x,r))}
\int_r^{\infty}\frac{\widetilde\Psi^{-1}(\varphi(x,t))}{t}\,dt.
\end{align}
By \cite[Lemma~4.4]{san21} and \eqref{int vp},
we find that, for any $x\in \mathcal{X}$ and $r\in(0,\infty)$,
\begin{equation*}
\int_r^{\infty}\frac{\widetilde\Psi^{-1}(\varphi(x,t))}{t}\,dt
\lesssim
\widetilde\Psi^{-1}(\varphi(x,r)),
\end{equation*}
which, together with \eqref{1531},
completes the proof of \eqref{Mb} and hence the aforementioned claim.

Let $f\in B_{(\varphi,\Psi)}(\mathcal{X})$.
Then there exist a sequence $\{b_j\}_{j\in\mathbb{N}}$ in $B(\varphi,\Psi)$
and a sequence $\{\lambda_j\}_{j\in\mathbb{N}}$ in $[0,\infty)$ satisfying
$\sum_{j\in\mathbb{N}}\lambda_j\le2\|f\|_{B_{(\varphi,\Psi)}(\mathcal{X})}$
such that $f=\sum_{j\in\mathbb{N}}\lambda_jb_j$ $\mu$-almost everywhere
on $\mathcal{X}$,
which, combined with \eqref{Mb}, further implies that
\begin{equation*}
\left\|{\mathcal M}f\right\|_{B_{(\varphi,\Psi)}(\mathcal{X})}
\le
\sum_{j=1}^\infty\lambda_j
\left\|{\mathcal M}b_j\right\|_{B_{(\varphi,\Psi)}(\mathcal{X})}
\le
C_0\sum_{j=1}^\infty\lambda_j
\le
2C_0\|f\|_{B_{(\varphi,\Psi)}(\mathcal{X})}.
\end{equation*}
This finishes the proof of Theorem~\ref{thm:Mf OB}.
\end{proof}

The following is a characterization of
Orlicz--Morrey spaces in terms of $A_1$-weights,
which is a generalization of \cite[Proposition~285]{sfh20a}
from Morrey spaces on $\mathbb{R}^n$ to
Orlicz--Morrey spaces on spaces of homogeneous type.

\begin{proposition}\label{prop:OM Lw}
Let $\varphi$ satisfy both \eqref{DC} and \eqref{int vp},
and let $\Phi$ be an Orlicz function
with positive lower type $r_{\Phi}^-$
and positive upper type $r_{\Phi}^+$.
Then there exist constants $\theta\in(0,1)$ and $C\in[1,\infty)$,
depending only on $K_0,L_{\mu},\Phi$, and $\varphi$,
such that,
for any $f\in L^{(\Phi,\varphi)}(\mathcal{X},\mu)$,
\begin{equation}\label{OM Lw}
\|f\|_{L^{(\Phi,\varphi)}(\mathcal{X},\mu)}
\le\sup_{B}\|f\|_{L^{\Phi}_{({\mathcal M}
\mathbf{1}_B)^{1-\theta}}(\mathcal{X},\mu_B)}
\le C\|f\|_{L^{(\Phi,\varphi)}(\mathcal{X},\mu)},
\end{equation}
where the supremum is taken over all balls $B\subset\mathcal{X}$
and $\mu_B:=\frac{\mu}{\varphi(B)\mu(B)}$.
\end{proposition}

\begin{proof}
Since $\mathbf{1}_B\le({\mathcal M}\mathbf{1}_B)^{1-\theta}$
for any $\theta\in(0,1)$,
it follows that the first inequality in \eqref{OM Lw} holds.
Next, we prove the second inequality in \eqref{OM Lw}.
Let $f\in L^{(\Phi,\varphi)}(\mathcal{X},\mu)$.
Using \eqref{int vp} and \cite[Lemma~7.1]{n08},
we conclude that there exists a positive constant $\epsilon$ such that,
for any $x\in \mathcal{X}$ and $r\in(0,\infty)$,
\begin{equation}\label{int vp e}
\int_r^{\infty}\frac{\varphi(x,t)t^{\epsilon}}{t}\,dt
\le C_{\epsilon}\varphi(x,r)r^{\epsilon}.
\end{equation}
Take $\theta\in(0,1)$ such that $\theta\log_2L_{(\mu)}\le\epsilon$.
Let $B:=B(x,r)$, $k:=K_0(1+2K_0)$, and $j\in\mathbb{N}$.
Notice that, if $z\notin 2^jkB$, $z\in B(y,s)$,
and $B(x,r)\cap B(y,s)\ne\emptyset$, then $2^jr<s$;
see \eqref{2jr<s}.
In this case, $\mu(B(y,s))\sim\mu(B(x,s))\ge\mu(B(x,2^jr))=\mu(2^jB)$.
Then, for any given $z\notin 2^jkB$,
\begin{equation*}
{\mathcal M}\mathbf{1}_{B}(z)
=\sup_{B(y,s)\ni z}\fint_{B(y,s)}\mathbf{1}_{B(x,r)}
\le\sup_{\genfrac{}{}{0pt}{}{B(y,s)\ni z}{B(x,r)\cap B(y,s)\ne\emptyset}}
\frac{\mu(B(x,r))}{\mu(B(y,s))}
\lesssim
\frac{\mu(B)}{\mu(2^jB)},
\end{equation*}
which further implies that
\begin{equation*}
\left({\mathcal M}\mathbf{1}_{B}\right)^{1-\theta}
=
\left({\mathcal M}\mathbf{1}_{B}\right)^{1-\theta}
\left(\mathbf{1}_{2kB}
+\sum_{j=1}^{\infty}\mathbf{1}_{2^{j+1}kB\setminus 2^{j}kB}\right)
\lesssim
\mathbf{1}_{2kB}
+\sum_{j=1}^{\infty} \left[\frac{\mu(B)}{\mu(2^jB)}\right]^{1-\theta}
\mathbf{1}_{2^{j+1}kB\setminus 2^{j}kB}
\end{equation*}
and, for any $\lambda\in(\|f\|_{L^{(\Phi,\varphi)}(\mathcal{X})},\infty)$,
\begin{equation*}
\Phi\left(\frac{|f|}{\lambda}\right)
\left({\mathcal M}\mathbf{1}_{B}\right)^{1-\theta}
\lesssim
\sum_{j=0}^{\infty}
\left[\frac{\mu(B)}{\mu(2^jB)}\right]^{1-\theta}
\Phi\left(\frac{|f|\mathbf{1}_{2^{j+1}kB}}{\lambda}\right).
\end{equation*}
From this, \eqref{dc}, $\theta\log_2 L_{(\mu)}\leq\epsilon$,
and \eqref{int vp e},
we infer that
\begin{align*}
&\int_\mathcal{X}
\Phi\left(\frac{|f|}{\lambda}\right)
\left({\mathcal M}\mathbf{1}_{B}\right)^{1-\theta}\,d\mu_B\\
&\quad\lesssim
\sum_{j=0}^{\infty}
\left[\frac{\mu(B)}{\mu(2^jB)}\right]^{1-\theta}
\int_\mathcal{X}\Phi\left(\frac{|f|\mathbf{1}_{2^{j+1}kB}}{\lambda}\right)
\frac{d\mu}{\mu(B)\varphi(B)} \\
&\quad=\frac1{\varphi(B)}
\sum_{j=0}^{\infty} \varphi(2^jB)
\left[\frac{\mu(2^jB)}{\mu(B)}\right]^{\theta}
\int_\mathcal{X}\Phi\left(\frac{|f|\mathbf{1}_{2^{j+1}kB}}{\lambda}\right)
\frac{d\mu}{\mu(2^jB)\varphi(2^jB)} \\
&\quad\lesssim
\frac1{\varphi(B)}
\sum_{j=0}^{\infty} \varphi(2^jB)
\left[\frac{\mu(2^jB)}{\mu(B)}\right]^{\theta}
\int_{2^{j+1}kB} \Phi\left(\frac{|f|}{\lambda}\right)
\frac{d\mu}{\mu(2^{j+1}kB)\varphi(2^{j+1}kB)} \\
&\quad\le\frac1{\varphi(B)}
\sum_{j=0}^{\infty} \varphi(2^jB)
\left[\frac{\mu(2^jB)}{\mu(B)}\right]^{\theta}
\lesssim
\frac1{\varphi(x,r)}
\sum_{j=0}^{\infty} \varphi(x,2^jr)
\frac{(2^jr)^{\epsilon}}{r^{\epsilon}}\\
&\quad\sim\frac1{\varphi(x,r)r^{\epsilon}}
\sum_{j=0}^{\infty} \int_{2^jr}^{2^{j+1}r}
\frac{\varphi(x,t)t^{\epsilon}}{t}\,dt
=\frac1{\varphi(x,r)r^{\epsilon}}
\int_{r}^{\infty}\frac{\varphi(x,t)t^{\epsilon}}{t}\,dt
\le C_{\epsilon}.
\end{align*}
This shows
\begin{equation*}
\|f\|_{L^{\Phi}_{({\mathcal M}\mathbf{1}_B)^{1-\theta}}(\mathcal{X},\mu_B)}
\lesssim\|f\|_{L^{(\Phi,\varphi)}(\mathcal{X})},
\end{equation*}
where the implicit positive constant is independent of both $B$ and $f$.
This finishes the proof of the second inequality in \eqref{OM Lw}
and hence Proposition~\ref{prop:OM Lw}.
\end{proof}

Applying Theorems~\ref{MS2} and~\ref{01011841}
and Proposition~\ref{prop:OM Lw},
we obtain the following conclusion.

\begin{theorem}\label{thm:OM/R Gag}
Let $\Phi$ be an Orlicz function  with positive lower
type $r_{\Phi}^-$ and positive upper type $r_{\Phi}^+$,
and let $\varphi\in\mathcal{G}^{\rm dec}_1$
satisfy both \eqref{NC} and \eqref{int vp}.
Assume that $\varphi$ satisfies \eqref{var const}
or both
\eqref{var r 0} and \eqref{var r infty}.
Let $0<q<r_{\Phi}^-$.
Then there exist $C,\widetilde{C}\in(0,\infty)$ such that,
for any $f\in [L^{(\Phi,\varphi)}(\Omega)/{\mathbb R}]\cap
\bigcup_{s\in(0,1)}\dot{W}^{s,q}_{L^{(\Phi,\varphi)}}(\Omega)$,
\begin{align*}
C\|f\|_{L^{(\Phi,\varphi)}(\Omega)/{\mathbb R}}
&\leq
\varliminf_{s \to 0^+}
s^{\frac{1}{q}}\left\|\left\{\int_\Omega
\frac{|f(\cdot)-f(y)|^q}{U(\cdot,y)[\rho(\cdot,y)]^{sq}}
\, d\mu(y) \right\}^{\frac{1}{q}}\right\|_{L^{(\Phi,\varphi)}(\Omega)} \\
&\leq
\varlimsup_{s \to 0^+}
s^{\frac{1}{q}}\left\|\left\{\int_\Omega
\frac{|f(\cdot)-f(y)|^q}{U(\cdot,y)[\rho(\cdot,y)]^{sq}}
\, d\mu(y) \right\}^{\frac{1}{q}}\right\|_{L^{(\Phi,\varphi)}(\Omega)}
\leq
\widetilde{C}\|f\|_{L^{(\Phi,\varphi)}(\Omega)/{\mathbb R}}.\notag
\end{align*}
\end{theorem}

\begin{proof}
We first prove the lower estimate.
By Remark~\ref{2048}(i) and Proposition~\ref{prop:OM Lw}, we find that
there exists $\theta\in(0,1)$ such that,
for any $f\in L^{(\Phi,\varphi)}(\Omega)$,
\begin{equation*}
\|f\|_{L^{(\Phi,\varphi)}(\Omega)}
\le
\sup_{B} \|f\|_{L^{\Phi}_{({\mathcal M}\mathbf{1}_B)^{1-\theta}}(\Omega,\mu_B)}
\le
C\|f\|_{L^{(\Phi,\varphi)}(\Omega)}
\end{equation*}
with $\mu_B:=\frac{\mu}{\mu(B)\varphi(B)}$,
where $C\in(0,\infty)$ is independent of $f$,
which, together with Proposition~\ref{2357},
further implies that, for any
$f\in L^{(\Phi,\varphi)}(\Omega)/{\mathbb R}$,
\begin{equation*}
\|f\|_{L^{(\Phi,\varphi)}(\Omega)/{\mathbb R}}
\le\sup_{B}\|f\|_{L^{\Phi}_{({\mathcal M}\mathbf{1}_B)^{1-\theta}}
(\Omega,\mu_B)/{\mathbb R}}
\le C\|f\|_{L^{(\Phi,\varphi)}(\Omega)/{\mathbb R}}.
\end{equation*}
From this,
the fact that, for any ball $B\subset\mathcal{X}$,
$[(\mathcal{M}\mathbf{1}_{B})^{1-\theta}]_{A_1(\mu_B)}\lesssim1$
with the implicit positive constant independent of $B$,
and Theorem~\ref{01011841}, we deduce that,
for any $f\in [L^{(\Phi,\varphi)}(\Omega)/{\mathbb R}]\cap
\bigcup_{s\in(0,1)}\dot{W}^{s,q}_{L^{(\Phi,\varphi)}}(\Omega)$,
\begin{align*}
\|f\|_{L^{(\Phi,\varphi)}(\Omega)/{\mathbb R}}
&\leq
\sup_{B} \|f\|_{L^{\Phi}_{({\mathcal M}\mathbf{1}_B)^{1-\theta}}
(\Omega,\mu_B)/{\mathbb R}}\\
&\lesssim
\sup_{B}
\varliminf_{s \to 0^+}
s^{\frac{1}{q}}\left\|\left\{\int_\Omega
\frac{|f(\cdot)-f(y)|^q}{U(\cdot,y)[\rho(\cdot,y)]^{sq}}
\, d\mu(y) \right\}^{\frac{1}{q}}\right\|_{L^{\Phi}_{({\mathcal M}
\mathbf{1}_B)^{1-\theta}}(\Omega,\mu_B)}\\
&\leq\varliminf_{s \to 0^+}\sup_{B}
s^{\frac{1}{q}}\left\|\left\{\int_\Omega
\frac{|f(\cdot)-f(y)|^q}{U(\cdot,y)[\rho(\cdot,y)]^{sq}}
\, d\mu(y) \right\}^{\frac{1}{q}}\right\|_{L^{\Phi}_{({\mathcal M}
\mathbf{1}_B)^{1-\theta}}(\Omega,\mu_B)}\\
&\lesssim\varliminf_{s \to 0^+}
s^{\frac{1}{q}}\left\|\left\{\int_\Omega
\frac{|f(\cdot)-f(y)|^q}{U(\cdot,y)[\rho(\cdot,y)]^{sq}}
\, d\mu(y) \right\}^{\frac{1}{q}}\right\|_{L^{(\Phi,\varphi)}(\Omega)},
\end{align*}
which completes the proof of the lower estimate.

Then we show the upper estimate.
By Theorem~\ref{thm:OM bBfs},
we conclude that the Orlicz--Morrey space $L^{(\Phi,\varphi)}(\mathcal{X})$
is a ball quasi-Banach function space.
Take $p\in[q,r_{\Phi}^-)$ and let $\Phi^{(\frac{1}{p})}(t)
=\Phi(t^{\frac{1}{p}})$.
Then $[L^{(\Phi,\varphi)}(\mathcal{X})]^\frac{1}{p}
=L^{(\Phi^{(\frac{1}{p})},\varphi)}(\mathcal{X})$
is a ball Banach function space.
From Theorems~\ref{thm:K dual} and \ref{thm:Mf OB}
and $1< r_{\widetilde{\Phi^{(\frac{1}{p})}}}^-\le
r_{\widetilde{\Phi^{(\frac{1}{p})}}}^+<\infty$,
it follows that the Hardy--Littlewood maximal
operator $\mathcal{M}$ is bounded on
$[L^{(\Phi^{(\frac{1}{p})},\varphi)}(\mathcal{X})]'
=B_{(\varphi,\widetilde{\Phi^{(\frac{1}{p})})}}(\mathcal{X})$.
Then we find that Assumption I is satisfied for $Y=L^{(\Phi,\varphi)}$.
By this and the proof of Theorem~\ref{MS2}, we obtain
\begin{align*}
\varlimsup_{s \to 0^+}
s^{\frac{1}{q}}\left\|\left\{\int_{\Omega}
\frac{|f(\cdot)-f(y)|^q}{U(\cdot,y)[\rho(\cdot,y)]^{sq}}
\, d\mu(y) \right\}^{\frac{1}{q}}
\right\|_{L^{(\Phi,\varphi)}(\Omega)}
\lesssim
\|f\|_{L^{(\Phi,\varphi)}(\Omega)/{\mathbb R}}.
\end{align*}
This finishes the proof of Theorem~\ref{thm:OM/R Gag}.
\end{proof}

\begin{remark}
To the best of our knowledge,
Theorem~\ref{thm:OM/R Gag} is completely new.
\end{remark}

The following theorem is an application of Theorem~\ref{MS2}.

\begin{theorem}\label{thm:OB Gag}
Let $\Psi$ be an Orlicz function with $1<r_{\Psi}^-\le r_{\Psi}^+<\infty$
and let $\varphi\in\mathcal{G}^{\rm dec}_1$
satisfy both \eqref{NC} and \eqref{int vp}.
Assume that $\phi$ satisfies \eqref{var const} or
both \eqref{var r 0} and \eqref{var r infty}.
If $\beta\in(0,\infty)$,
then there exist $C,\widetilde{C}\in(0,\infty)$ such that,
for any $f\in C_{\mathrm{b}}^\beta(\Omega)\cap
\bigcup_{s\in(0,1)}\dot{W}^{s,q}_{B_{(\varphi,\Psi)}}(\Omega)$,
\begin{align*}
C\|f\|_{B_{(\varphi,\Psi)}(\Omega)}
&\leq
\varliminf_{s \to 0^+}
s^{\frac{1}{q}}\left\|\left\{\int_\Omega
\frac{|f(\cdot)-f(y)|^q}{U(\cdot,y)[\rho(\cdot,y)]^{sq}}
\, d\mu(y) \right\}^{\frac{1}{q}}
\right\|_{B_{(\varphi,\Psi)}(\Omega)}\\
&\leq
\varlimsup_{s \to 0^+}
s^{\frac{1}{q}}\left\|\left\{\int_\Omega
\frac{|f(\cdot)-f(y)|^q}{U(\cdot,y)[\rho(\cdot,y)]^{sq}}
\, d\mu(y) \right\}^{\frac{1}{q}}
\right\|_{B_{(\varphi,\Psi)}(\Omega)}
\leq
\widetilde{C}\|f\|_{B_{(\varphi,\Psi)}(\Omega)}.
\end{align*}
\end{theorem}

\begin{proof}
From Theorem~\ref{thm:OB bBfs},
we infer that $B_{(\varphi,\Psi)}(\mathcal{X})$
is a ball Banach function space.
By Theorems~\ref{thm:K dual} and \ref{thm:Mf OM},
we conclude that the Hardy--Littlewood maximal
operator $\mathcal{M}$ is bounded on
$[B_{(\varphi,\Psi)}(\mathcal{X})]'=L^{(\Phi,\varphi)}(\mathcal{X})$,
where $\Phi$ is the complementary Orlicz
function of $\Psi$ [see \eqref{1539}].
Using these and applying Theorem~\ref{MS2}(i),
we then obtain the desired estimates,
which completes the proof of Theorem~\ref{thm:OB Gag}.
\end{proof}

\begin{remark}
To the best of our knowledge,
Theorem~\ref{thm:OB Gag} is completely new.
\end{remark}

\section*{Declaration of competing interest}

The author declares that there is no competing interests.

\section*{Acknowledgements}

The authors would like to thank Professor Durvudkhan Suragan
for his insightful suggestion to consider the
original results of this article in the setting of domains,
which led to Theorems~\ref{MS2} and~\ref{MS3}.

\section*{Data availability}

No data was used for the research described in the article.

\bigskip

\noindent Eiichi Nakai

\medskip

\noindent  	Department of Mathematics, Ibaraki University, Mito, Ibaraki 310-8512, Japan

\smallskip

\noindent{\it E-mail:} \texttt{eiichi.nakai.math@vc.ibaraki.ac.jp}

\bigskip
\noindent Menghao Tang, Dachun Yang (Corresponding author)
and Wen Yuan

\medskip

\noindent  	Laboratory of Mathematics and Complex Systems
(Ministry of Education of China),
School of Mathematical Sciences, Beijing Normal University,
Beijing 100875, The People's Republic of China

\smallskip

\noindent {\it E-mails}: \texttt{mhtang@mail.bnu.edu.cn} (M. Tang)

\noindent\phantom{{\it E-mails:}} \texttt{dcyang@bnu.edu.cn} (D. Yang)

\noindent\phantom{{\it E-mails:}} \texttt{wenyuan@bnu.edu.cn} (W. Yuan)

\bigskip

\noindent Chenfeng Zhu

\medskip

\noindent  	School of Mathematical Sciences,
Zhejiang University of Technology, Hangzhou 310023,
The People's Republic of China

\smallskip

\noindent{\it E-mail:} \texttt{chenfengzhu@zjut.edu.cn}

\end{document}